\documentclass[english,reqno]{amsart}

\usepackage{amsmath}
\usepackage{amssymb} 
\usepackage[new]{old-arrows}
\usepackage{enumitem} 
\usepackage{mathtools} 
\usepackage[table]{xcolor} 
\usepackage[all]{xy} 
\usepackage{tikz} 
\usepackage{tikz-cd}
\usepackage{indentfirst} 
\usepackage{babel} 
\usepackage{setspace} 
\usepackage{stmaryrd}

\usepackage[colorlinks,linkcolor=red,anchorcolor=green,citecolor=blue]{hyperref} 
\hypersetup{linktocpage = true} 

\usepackage{rotating} 

\usepackage{ytableau} 
\usepackage{longtable} 
\newcolumntype{M}[1]{>{\centering\arraybackslash}m{#1}} 

\usepackage{mathpazo}
\usepackage{mathrsfs}
\DeclareFontFamily{OMS}{rsfs}{\skewchar\font'60}
\DeclareFontShape{OMS}{rsfs}{m}{n}{<-5>rsfs5 <5-7>rsfs7 <7->rsfs10 }{}
\DeclareSymbolFont{rsfs}{OMS}{rsfs}{m}{n}
\DeclareSymbolFontAlphabet{\scr}{rsfs}
\DeclareSymbolFontAlphabet{\scr}{rsfs}

\usepackage[T1]{fontenc}

\newcommand\cF{{\mathcal F}}

\DeclareMathOperator*{\red}{red}

\theoremstyle{plain}
\newtheorem{thm}{Theorem}[section]
\newtheorem{lemma}[thm]{Lemma}
\newtheorem{prop}[thm]{Proposition}
\newtheorem{cor}[thm]{Corollary}
\newtheorem{defn}[thm]{Definition}

\newtheorem{setup}[thm]{Setup}

\theoremstyle{remark}
\newtheorem{example}[thm]{Example}
\newtheorem{remark}[thm]{Remark}

\def\mod{\operatorname{mod}}

\def\dim{\operatorname{dim}}

\def\Ex{\operatorname{Ex}}

\def\min{\operatorname{min}}
\def\max{\operatorname{max}}

\def\hor{\operatorname{\hor}}
\def\ver{\operatorname{\ver}}

\def\sm{\operatorname{\textsubscript{\rm sm}}}
\def\sing{\operatorname{\textsubscript{\rm sing}}}

\def\ver{\operatorname{\textsubscript{\rm ver}}}
\def\hor{\operatorname{\textsubscript{\rm hor}}}

\def\red{\operatorname{\textsubscript{\rm red}}}

\setlist[itemize]{leftmargin=*}
\setlist[enumerate]{leftmargin=*}

\numberwithin{equation}{section} 

\makeatletter

\ifnum\@ptsize=0  \fi \ifnum\@ptsize=2  \fi

\title{Title} 

\subjclass[2010]{}

\author{Wenhao Ou}

\address{Wenhao Ou, Institute of Mathematics, Academy of Mathematics and Systems Science, Chinese Academy of Sciences, Beijing, 100190, China}
\email{wenhaoou@amss.ac.cn}

\begin{document}

\begin{abstract}
We prove that if  $X$ is a compact complex analytic variety, which has quotient singularities in codimension 2, then there is a projective bimeromorphic morphism $f\colon Y\to X$, such that $Y$ has quotient singularities, and that the indeterminacy locus of $f^{-1}$ has codimension at least 3 in $X$.  
\end{abstract}

\title{Orbifold modifications of complex analytic varieties}

\maketitle

\tableofcontents


\section{Introduction}

The theory of  holomorphic vector bundles is a central object in complex algebraic geometry and complex analytic geometry. 
The notion of stable vector bundles on complete curves was introduced by Mumford  in  \cite{Mumford1963}. 
Such notion of stability was then extended to torsion-free sheaves on any projective manifolds (see \cite{Takemoto1972}, \cite{Gieseker1977}), and is now known as the slope stability. 
An  important property of stable vector bundles is the following Bogomolov-Gieseker inequality, involving the Chern classes of the vector  bundle.
\begin{thm}
\label{thm:BG-inequality-intro}
Let $X$ be a projective manifold of dimension $n$,   let  $H$ be  an ample divisor, and let $\cF$ be a $H$-stable vector bundle of rank $r$ on $X$. 
Then 
\[   \Big(c_2(\cF)-\frac{r-1}{2r}c_1(\cF)^2 \Big)  \cdot  H^{n-2} \ge  0. \]
\end{thm}
When  $X$ is a surface, the inequality was proved in \cite{Bogomolov1978}. 
In higher dimensions, one may apply Mehta-Ramanathan   theorem in  \cite{MehtaRamanathan1981/82}   to  reduce to the case of surfaces, by taking hyperplane  sections. 
Later in \cite{Kawamata1992}, as a part of the proof for the three-dimensional abundance theorem, Kawamata extended the inequality to orbifold Chern  classes of  reflexive sheaves on projective surfaces with quotient singularities. 
The technique of taking hypersurface sections then  allows us to deduce the Bogomolov-Gieseker inequality for reflexive sheaves on projective varieties which have quotient singularities in codimension 2.

On the analytic side, let $(X,\omega)$ be a compact  K\"ahler manifold, and $(\cF, h)$ a Hermitian  holomorphic vector bundle on $X$.  
L\"ubke proved that  if $h$ satisfies the Einstein condition, then the following inequality holds (see \cite{Lub82}),
\[   \int_X \Big(c_2(\cF,h)-\frac{r-1}{2r}c_1(\cF,h)^2 \Big)  \wedge  \omega^{n-2} \ge 0. \] 
It is now well understood that if $\cF$ is slope stable, then it admits a Hermitian-Einstein metric. 
The case when $X$ is a complete curve was proved by Narasimhan-Seshadri  in \cite{NarasimhanSeshadri1965}, the case of projective surfaces was proved by Donaldson in \cite{Donaldson1985}, and the case of arbitrary compact K\"ahler manifolds was proved by Uhlenbeck-Yau in \cite{UhlenbeckYau1986}.   
Simpson extended the existence of  Hermitian-Einstein  metric to stable Higgs bundles, on compact and certain non compact K\"ahler manifolds, see \cite{Simpson1988}. 
Furthermore,  in \cite{BandoSiu1994},  Bando-Siu introduced the notion of admissible metrics and proved the existence of admissible  Hermitian-Einstein  metrics on stable reflexive sheaves. 

Comparing with the algebraic version, it is natural to expect a Bogomolov-Gieseker  type inequality,  for  stable coherent reflexive sheaves  on a  compact K\"ahler variety, 
which has at most quotient singularities  in codimension 2, see for example \cite{CHP23}.  
When the underlying space has quotient singularities only, for example when it is a surface, an orbifold version of Donaldson-Uhlenbeck-Yau theorem was   proved by Faulk in \cite{Faulk2022}.   
As a consequence, a Bogomolov-Gieseker type inequality holds in this case. 
However, for a general K\"ahler variety, we are not able to take hyperplane sections. 
So the algebraic method does not apply.  
When the underlying space is smooth in codimension 2, 
Bogomolov-Gieseker type inequalities have also been established, 
see for example \cite{CHP16}, \cite{Chen2022}, \cite{ChenWentworth2024} and \cite{Wu21}. 
In general, it was suggested in \cite{CGNPPW} that the existence of orbifold modifications would imply the inequality. 
The main objective of this paper is to study  this problem.  
We prove the following theorem.

\begin{thm}
\label{thm:main-thm}
Let $X$ be a compact complex analytic variety. 
Assume that $X$ has quotient singularities in codimension 2. 
In other words, there is a closed analytic  subset $V$ of codimension at least 3, such that $X\setminus V$ has quotient singularities. 
Then there is a projective bimeromorphic morphism $f\colon Y \to X$ such that $Y$ has quotient singularities, and that the indeterminacy locus of $f^{-1}$ has codimension at least 3 in $X$. 
\end{thm}

We   note that when $X$ is projective, a stronger version of Theorem \ref{thm:main-thm} was proved by Xu in \cite{LiTian2019}. 
The proof approximately proceeds as follows. 
We let $X^\circ$ be the largest open subset on which $X$ has quotient singularities. 
By considering the frame bundle on the local orbifold charts, we can obtain a smooth quasi-projective variety $ Y^\circ$, such that $X^\circ= Y^\circ/G$, where $G=\mathrm{GL}_n(\mathbb{C})$ and $n$ is the dimension of $X$ (see \cite{Kresch2009}). 
The action of $G$ on $Y^\circ$ only has finite stabilizers. 
There exists a  $G$-equivariant projective compactification $ {Y}$ of $Y^\circ$, such that $G$ acts linearly on $ {Y}$. 
Then up to applying Kirwan's partial resolution in \cite{Kirwan1985}, 
one may assume that  the strictly semi-stable locus in $ {Y}$ is empty.  
Then the GIT quotient $ {Y} \sslash G$ is a projective compactification of $X^\circ$ which has quotient singularities only. 
Finally, by using equivariant resolution of singularities to the birational map $ {Y} \sslash G \dashrightarrow X$, 
we can obtain  a projective compactification $Z$ of $X^\circ$ such that $Z$ has quotient singularities and the natural birational map $Z\dashrightarrow X$ is a morphism. 
The main difficulty to adapt this method in the setting of general complex analytic varieties is that, 
the theory of GIT relies essentially on the existence of an ample line bundle, which is not the case if $X$ is not projective.

In this paper, we take a different approach and do not consider actions of  positive dimensional groups. 
It consists of several steps as we will sketch them as follows.  
Let $X$ be a compact complex analytic variety which has quotient singularities in codimension 2. 
Let $S$ be the codimension 2 part of the singular locus and let $S_1,...,S_r$ be its irreducible components. 

In the first step, by blowing up $X$ at the intersections of the components of $S$, 
and then by taking an appropriate bimeromorphic model using the Minimal Model Program (MMP), 
we reduce to the case when  $S_1 ,..., S_r$  are pairwise disjoint. 
Furthermore, $X$ has  klt singularities.  

In the second step, around every point  $x\in X$, we take the index-one cover $V$ of the canonical sheaf $\omega_X$. 
The advantage we get is that $V$ has canonical singularities. 
In particular, it has hypersurface singularities in codimension 2. 
Since in general,  index-one covers  only exist  locally, and  $V$ is  singular, we work with the notion of  complex analytic orbispaces, 
which is a straightforward extension of orbifolds. 

Now we work locally around every point of $S$.  
In the third step, we assume that $X$ has canonical singularities, and  will blowup $X$ at centers strictly contained in $S$, 
so that  $X$ has hypersurface singularities around (the strict transform of) $S$. 
This is the main objective of Section \ref{section:double-point}.

Next in the fourth step, we blow up  further so that a defining equation of $X$ is in a  similar shape as the ones for Du Val singularities.
With this shape of defining equation, we may consider $X$ as a family of Du Val singularities.  
This will be done in Section \ref{section:equation}. 

In the fifth step, in Section \ref{section:local-cover}, 
by using smooth quasi-\'etale covering spaces of  Du Val singularities, 
we can construct covering spaces of $X$ which are smooth over  general points of $S$.  
In general, such a covering space will contain new and mysterious singularities, over the divisorial critical locus, which can also be regarded as the degeneracy locus of the family of Du Val singularities.

In the end, we first show in Section \ref{section:local-proof}  that the previous local constructions can be carried out globally, with the language of  complex analytic orbispaces.   
Then in Section \ref{section:modification},   we take some dlt modification on $X$, so that we  envelop the divisorial critical locus into a dlt reduced divisor. 
What we gain  is that, dlt pairs are  close  to snc pairs, and we have better understanding on finite covers whose divisorial critical locus is a reduced dlt divisor.  
With such a modification, we  can deduce  Theorem \ref{thm:main-thm}.

\vspace{2mm}

\noindent\textbf{Acknowledgment.} 
The author is grateful to Omprokash Das for conversations.
The author is supported by the National Key R\&D Program of China (No. 2021YFA1002300).

\section{Preliminaries}

We  fix some notation and prove some elementary results in this section. 
Throughout this paper, the symbol $\mathbb{D}$ stands for a disc  contained in $\mathbb{C}$ centered at $0$. 
We  will write $\mathbb{D}_r$ if we need to stress the radius $r$ of $\mathbb{D}$. 
We denote  the origin of  $\mathbb{C}^n$ by $\mathbf{0}_n$. 
By a polydisc of dimension $n$, we  refer to $\mathbb{D}^n$, 
which is always assumed  centered at the origin of $\mathbb{C}^n$.

\subsection{Complex analytic varieties and their singularities} 

A  complex analytic variety  $X$
is a reduced and irreducible  complex analytic space.   
We will denote by $X_{\sm}$ its smooth locus and by $X_{\sing}$ its singular locus.  
A smooth complex analytic variety is also called a complex manifold.   
We say that a normal compact complex analytic variety $X$ is $\mathbb{Q}$-factorial, if for any reflexive sheaf $\mathcal{L}$ of rank one on $X$, there is some integer $m> 0$ such that $(\mathcal{L}^{\otimes m})^{**}$ is locally free.

We underline the notion   of reduced divisors in a complex analytic variety $X$ which can be non normal. 
A prime divisor in $X$ is a closed irreducible and reduced  analytic subspace of codimension 1. 
A reduced divisor $D$ in $X$ is by definition a formal finite sum of distinct prime divisors. 
By abuse of notation, we also use $D$ to denote the support of it,  
which is a reduced closed analytic subspace of $X$, pure of codimension 1.

We say that a complex analytic variety $X$ has quotient singularities in codimension $2$ if there is a closed analytic subset $V\subseteq X$ of codimension at least $3$, such that $X\setminus V$ has quotient singularities.  
We refer to \cite[Section 2.3]{KollarMori1998} for the notion of terminal, canonical, klt and dlt singularities.  
As shown in \cite[Lemma 5.8]{GK20} (see also \cite[Proposition 9.3]{GKKP2011}), 
if $X$ has klt singularities, then it has quotient singularities in codimension $2$.

Let $X$ be a complex analytic variety of dimension $n$, 
and  let $S\subseteq X_{\sing}$ be an irreducible component. 
Assume that $S$ has codimension 2 in $X$, and that $X$ has quotient singularities at general points of $S$. 
Then as shown in \cite[Lemma 5.8]{GK20} (see also \cite[Proposition 9.3]{GKKP2011}), 
there is a proper closed analytic subset $Z$ of $S$, 
such that the following property holds. 
For every point $x\in S\setminus Z$, there is an open neighborhood $U$ of $x$ in $X$, 
such that $(x\in U) \cong (\mathbf{0}_{n-2} \in  \mathbb{D}^{n-2}) \times (o\in V)$, where $(o\in V)$ is a klt surface singularity.  
In  the next  lemma, we show that the singularity type $(o\in V)$ is the same as the one of any local surface intersecting $S$ transversally at $x$. 
Since $S\setminus Z$ is connected, these types are all the same on $S\setminus Z$. 
Therefore, we will say that  $X$ has the same type of singularities  at points of $S\setminus Z$.  
And if for example $(o\in V)$ is a Du Val singularity of type $A_r$, 
then we say that $X$ has $A_r$-type singularities at points of $S\setminus Z$.

\begin{lemma}
\label{lemma:singular-type-well-defined}
With the notation above, if $T\subseteq U$ is a  surface intersecting $S$ transversally at $x$, then 
$(x\in T)$ is isomorphic to $(o\in V)$ as surface singularities. 
\end{lemma}

\begin{proof}
We may assume hat $U=  \mathbb{D}^{n-2} \times V$. 
Then the subset $S\subseteq U$ is identified with $\mathbb{D}^{n-2} \times \{o\}$.
Let $\mu\colon \widetilde{V} \to V$ be the minimal resolution, and let $\rho\colon \widetilde{U} \to U $ be the desingularization  with 
$ \widetilde{U} = \mathbb{D}^{n-2} \times \widetilde{V}$.  
Let $\widetilde{T} = \rho^{-1}(T)$, and   let $\tilde{x}\in \widetilde{T}$ be a point lying over $x$.  
With the product structure   $\widetilde{U} = \mathbb{D}^{n-2} \times \widetilde{V}$, 
we can write $\tilde{x} = (\mathbf{0}_{n-2}, \tilde{v})$, where $\tilde{v} \in \widetilde{V}$ is contained in the exceptional divisor of $\mu$. 
Let   $\sigma \colon S \to \widetilde{U}$ be a section, which is the composition of the following sequence, 
\[S = \mathbb{D}^{n-2}  \times \{ o \}  \to  \mathbb{D}^{n-2}  \times \{ \tilde{v}\} \hookrightarrow   \widetilde{U}.\]
Since $T$ meets $S$ transversally at $x$, we deduce that, at $\tilde{x}$, 
the Zariski tangent space of $\widetilde{T}$ and the Zariski tangent space of $\sigma(S)$, 
both viewed as the subspace of the Zariski tangent space of $\widetilde{U}$, 
intersect transversally. 
Since $\dim\,  \widetilde{T} + \dim\, \sigma(S) = \dim \, \widetilde{U}$, and since $\widetilde{U}$ is smooth, 
we deduce that $\widetilde{T}$ is smooth at $\tilde{x}$. 
Therefore, we obtain that $\widetilde{T}$ is smooth. 
Furthermore, by the adjunction formula, we see that  $\widetilde{T} \to T$ is the minimal resolution. 
This implies that $(x\in T)$ is isomorphic to $(o\in V)$ as surface singularities. 
\end{proof}

In the following lemma, we recall a characterization of different  Du Val singularities by their defining equations (see also \cite[4.25]{KollarMori1998}).

\begin{lemma}
\label{lemma:ADE} 
Let $(o\in X)$ be the germ of a Du Val singularity, defined in a neighborhood of the origin of $\mathbb{C}^3$ by an equation of the shape 
\[
F(x,y,z) = x^2+ F_2(y,z) + F_3(y,z) + R(y,z) = 0,
\]
where $F_2$ and $F_3$ are homogeneous polynomials in $(y,z)$ of degree 2 and 3 respectively, and   $R(y,z) = 0 \mod \, (y,z)^4$. 
Then the following properties hold.
\begin{enumerate}
    \item $F_2 \neq 0$ if and only if $(o\in X)$ is of type $A_r$ for some $r\ge 1$. 
    \item $F_2=0$ and $F_3$ has three pairwise coprime  factors if and only if $(o\in X)$ is of type $D_4$. 
    \item $F_2=0$ and $F_3$ has exactly two   coprime  factors if and only if $(o\in X)$ is of type $D_r$ for some $r\ge 5$.
    \item $F_2=0$ and $F_3$ is a cube if and only if $(o\in X)$ is of type $E_r$ for some $r \in \{6,7,8\}$. 
\end{enumerate}
Furthermore, in the last case, if we can write 
\[
F(x,y,z) = x^2+  u(y,z) \cdot y^3+ u_a(y,z) \cdot   yz^a + u_b(y,z) \cdot   z^b, 
\]
where $a\ge 3$ and $b\ge 4$ are integers, 
$u$ is a unit, 
and $u_a$, $u_b$ are either units or zero, 
then the following assertions hold. 
\begin{enumerate}
    \item[(5)] $b=4$ and $u_b\neq 0$ if and only of $(o\in X)$ is of type $E_6$. 
    \item[(6)] $u_b=0$ or $b\ge 5$, and $u_a\neq 0$, and $a=3$ if and only if $(o\in X)$ is of type $E_7$. 
    \item[(7)] $u_a=0$ or $a\ge 4$, and $u_b\neq 0$, and $b=5$ if and only if $(o\in X)$ is of type $E_8$. 
\end{enumerate}
\end{lemma}

\begin{proof}
The item (1) follows from the Step 3 of \cite[4.25]{KollarMori1998}. 
For the item (2) and (3), as in Step 4 of \cite[4.25]{KollarMori1998}, we see that $(o\in X)$ is of type $D_r$ for some $r\ge 4$ if and only if $F_2=0$ and $F_3$ is not a cube. 
In particular, up to a linear change of coordinates, we can write 
\[F_3(y,z) = z(\alpha z+ y)(\beta z+y)  \]
where $\alpha,\beta\neq 0$ are complex numbers. 
Applying \cite[(4.24.3)]{KollarMori1998}, we can write 
\[
F_3(y,z)  + R(y,z) = (z+p(y,z) )(y^2 +(\alpha + \beta)yz + \alpha\beta  z^2 + q(y,z) ), 
\]
where $p,q$ are holomorphic functions whose terms have degrees at least 2 and 3 respectively. 
Replacing $z$ by $z-p(y,z)$, we can assume that  
\begin{equation} \label{eqn:Du-Val}
F_3(y,z)  + R(y,z) = z( y^2 +(\alpha + \beta)yz + \alpha\beta z^2 + q(y,z)).     
\end{equation}
If $\alpha=\beta$, then  up to replacing $y$ by $y-\alpha z$, we can assume that 
\begin{eqnarray*}
    F_3(y,z)  + R(y,z) = z( y^2 + q(y,z)).  
\end{eqnarray*}
Applying Weierstrass preparation theorem for $ y^2 + q(y,z)$ with respect to $y$, 
and noting that $q(y,z)=0\mod \, (y,z)^3$, 
we have 
\[
y^2 + q(y,z) = (\mathrm{unit})\cdot (y^2 + yv_1(z) + v_2(z))
\]
such that $v_1(z) = 0\mod \,  (z)^2$ and $v_2(z) = 0 \mod \, (z)^3$.  
Replacing $y$ by $y-\frac{1}{2}v_1(z)$, we can assume that 
\[
F_3(y,z)  + R(y,z) = (\mathrm{unit})\cdot z( y^2 + w(z)), 
\]
where $w(z)=0\mod \,  (z)^3$. 
Therefore, we have 
\[
F_3(y,z)  + R(y,z) = (\mathrm{unit})\cdot zy^2 +   (\mathrm{unit})\cdot z^s, 
\]
for some integer $s\ge 4$.  
It follows that $(o\in X)$ is of type $D_{s+1}$. 
This proves the item (3). 

If $\alpha \neq \beta$, replacing $y$ by $y-\frac{1}{2}(\alpha+\beta)z$ in \eqref{eqn:Du-Val}, we can assume that 
\[
F_3(y,z)  + R(y,z) =  z(y^2 +   \gamma z^2  +q(y,z)),
\] 
where $\gamma\neq 0$ is a complex number and $q(y,z)=0\mod \, (y,z)^3$.  
The same argument as above shows that  $(o\in X)$ is of type $D_{4}$. 
We hence deduce the item (2).

Now we assume that $F_3$ is a cube. 
Then the item (4)  is a consequence of the items (1)-(3).
The items (5) and (7) follow  from the Step 7 of \cite[4.25]{KollarMori1998}, and the item (6) follows from the Step 8 of \cite[4.25]{KollarMori1998}.  
This completes the proof of the lemma. 
\end{proof}

Let $f\colon Y\to X$ be a proper bimeromorphic morphism between complex analytic spaces. 
Since we do not assume that $Y$ is normal, 
we say that there is a $f$-ample and $f$-exceptional divisor $-H$, if $H$ is a closed analytic subspace of $Y$ contained in the $f$-exceptional locus, 
such that the ideal sheaf $\mathcal{I}$ of $H$ is a $f$-ample invertible  sheaf   (see \cite[Definition 1.1]{Nakayama1987}).    
In the proof of the main theorem, we apply several local constructions around a component of the singular locus. 
The following lemma enables us to extend local modifications  to the whole variety.

\begin{lemma}
\label{lemma:extend-morphism-1} 
Let $X$ be a complex analytic space  and let $U\subseteq X$ be an open subset. 
Assume that there is a proper bimeromorpihc  morphism $p\colon U'\to U$ such that the indeterminacy locus $Z$ of $p^{-1}$ is compact. 
Then $p$ extends to some proper bimeromorpihc  morphism $f\colon X'\to X$, which is an isomorphism over  $X \setminus Z$. 
Assume further that there is some $p$-exceptional  divisor in $U'$ which is  $p$-ample over $U$, then $f$ is projective over $X$. 
\end{lemma}

\begin{proof}
Let $W=X\setminus Z$. 
Since $p$ is an isomorphism over $U\cap W$, we deduce that  $U'$ and $W$ glue together along $U\cap W$ to a complex analytic space $X'$, and that $p$ extends to a proper  bimeromorphic morphism $f\colon X'\to X$. 
If there is  a   $p$-exceptional and $p$-ample  divisor $-H$, 
then   $H$ is compact since $Z$ is compact. 
Thus $H$ is naturally a closed analytic subspace of $X'$. 
We note that its ideal sheaf in $X$ is also invertible and  $f$-ample over  $X$. 
Hence  $f\colon X'\to X$ is projective.      
\end{proof}

\subsection{Finite morphisms between complex analytic varieties} 

In this subsection, we summarize some results concerning finite morphisms between complex analytic varieties. 
We recall from \cite[page 47]{GR84} that a morphism $f\colon Y\to X$ is called finite if it is closed and if every fiber is a finite set. 
In this case,  the critical locus of $f$ is the smallest closed subset  $Z\subseteq X$, such that  $f$ is \'etale over $X\setminus Z$, that is, $f$ is  locally biholomorphic on $Y\setminus f^{-1}(Z)$. 
By the divisorial critical locus of $f$, we refer to the codimension 1 part of $Z$. 
The morphism $f$ is said to be quasi-\'etale if its critical locus has codimension at least $2$.  
By abuse of notation, we say that $f$ is Galois of Galois group $G$ if its restriction over $X\setminus Z$ is Galois, and if there is an action of $G$ on $Y$, extending the Galois action of $G$ on $f^{-1}(X\setminus Z)$, such that the quotient $Y/G$ is canonically isomorphic to $X$. 
We refer to \cite[Th\'eor\`eme 4]{Cartan1957} for the quotient of a  complex  analytic space by a finite group.  
The following theorem is due to  Grauert-Remmert.

\begin{thm}
\label{thm:GR-cover} 
Let $X$ be a complex analytic variety,  
and let $X^\circ \subseteq X$ be a dense Zariski open subset.  
Assume that $X^\circ$ is normal and we have a finite \'etale morphism $p\colon Y^\circ \to X^\circ$. 
Then $p$ extends to a finite morphism $p\colon Y\to X$ with $Y$ normal, which is  unique up to isomorphism.
\end{thm}

\begin{proof} 
Let $r\colon X'\to X$ be the normalization. 
By assumption, $r$ is an isomorphism over $X^\circ$. 
Hence up to replacing $X$ by $X'$, we may assume that $X$ is normal. 
In this case, the theorem is proved in {\cite[Th\'eor\`eme XII.5.4]{SGA1}}.  
\end{proof}



We recall the notion of cyclic covers. Let $X$ be a complex analytic space and let $s_1,...,s_m$ be holomorphic functions on $X$. 
Assume that $s_i$ and $s_j$ do not have non-unit common factors if $i\neq j$.  
Let $k_1,...,k_m $ be positive integers. 
Then we have the following $\mathcal{O}_X$-algebra 
\[
\mathcal{A} = \mathcal{O}_X [T_1,...,T_m] /(T_1^{k_1} - s_1,..., T_m^{k_m}-s_m).   
\]
Let $Z = X \times \mathbb{C}^m$ so that  $(T_1,...,T_m)$ is  the coordinates system of $\mathbb{C}^m$. 
Then we can   define the following complex analytic space 
\[
X[\sqrt[k_1]{s_1}, ..., \sqrt[k_m]{s_m}]  =  \mathrm{Spec}_{\mathcal{O_X}}\, \mathcal{A}   
\]
as the closed subspace in $Z$ define by the ideal  $(T_1^{k_1} - s_1,..., T_m^{k_m}-s_m)$.

\begin{lemma}
\label{lemma:cyclic-cover-unit}
With the notation above, let $x\in X$ be a point and let $s'_1,...,s'_m$ be holomorphic functions on $X$. 
Assume that there are unit functions $\mu_1,...,\mu_m$ such that $s'_i=\mu_i \cdot s_i$ for all $i=1,...,m$.  
Then, up to shrinking $X$ around $x$, there is an isomorphism  over $X$  
\[
X[\sqrt[k_1]{s_1}, ..., \sqrt[k_m]{s_m}]  \cong X[\sqrt[k_1]{s'_1}, ..., \sqrt[k_m]{s'_m}]. 
\]
In particular, if $x' \in X[\sqrt[k_1]{s_1}, ..., \sqrt[k_m]{s_m}]$ is a point lying over $x$, 
then the germ $(x'\in X[\sqrt[k_1]{s_1}, ..., \sqrt[k_m]{s_m}])$  depends only on the integers $k_1,...,k_m$, and the stalks  $(\mathcal{I}_1)_x,..., (\mathcal{I}_m)_x$ at $x$ of the invertible  ideal sheaves generated by $s_1,...,s_m$ respectively.
\end{lemma}

\begin{proof}
Up to shrinking $X$ around $x$, we may assume that $\mu_i$ admits a $k_i$-th root $\gamma_i$. 
Then we have 
\begin{eqnarray*}
 &&    \mathcal{O}_X [T_1,...,T_m] /(T_1^{k_1} - s'_1,..., T_m^{k_m}-s'_m)   \\
 &=&   \mathcal{O}_X [T_1,...,T_m] /(T_1^{k_1} - \gamma_1^{k_1} s_1,..., T_m^{k_m}-  \gamma_m^{k_m}s_m) \\
 &  = &  \mathcal{O}_X [T_1,...,T_m] /( \gamma_1^{k_1} ( (\gamma_1^{-1}T_1)^{k_1} -  s_1),..., \gamma_m^{k_m} ((\gamma_m^{-1}T_m)^{k_m}-  s_m ) ) \\
 &  = &  \mathcal{O}_X [T_1,...,T_m] /(  (\gamma_1^{-1}T_1)^{k_1} -  s_1,...,   (\gamma_m^{-1}T_m)^{k_m}-  s_m  ) \\
 &\cong &  \mathcal{O}_X [T_1,...,T_m] /(T_1^{k_1} - s_1,..., T_m^{k_m}-s_m),   
\end{eqnarray*}
where the last isomorphism is an isomorphism of $\mathcal{O}_X$-algebras. 
This implies the assertion of the lemma. 
\end{proof}

With the notation above, we assume that  $X$ is a normal  complex analytic variety.   
Suppose that, for $i=1,...,m$, there is an integer $r_i>0$ and a   divisor $D_i$ such that  $r_iD_i$ is defined by   $s_i=0$.  
Then for any integer $n>0$,  we can construct the cyclic cover $X[\sqrt[nr_1]{s_1}, ..., \sqrt[nr_m]{s_m}]$ over $X$. 
Let \[X[\sqrt[nr_1]{s_1}, ..., \sqrt[nr_m]{s_m}]^{nor} \] 
be  its normalization and let  $p\colon X[\sqrt[nr_1]{s_1}, ..., \sqrt[nr_m]{s_m}]^{nor}\to X $ be the natural morphism.

\begin{lemma}
\label{lemma:cyclic-functorial}
With the notation above, we let $x\in X$ be a point and let $U\subseteq X$ be an open neighborhood of $x$. 
Assume that for each $i=1, ... ,m$, the divisor  $l_iD_i|_U$ is   defined by $t_i=0$, 
where $l_i > 0$ is an integer and  $t_i$ is a holomorphic function on $U$. 
Then, up to shrinking $U$,  every connected component of   $ U[\sqrt[nl_1]{t_1}, ..., \sqrt[nl_m]{t_m}]^{nor} $ is isomorphic to every connected component of $p^{-1}(U)$.
\end{lemma}

\begin{proof}
We remark that both $l_i$ and $r_i$ are divisible by the Cartier index of $D_i$ at $x$. 
Hence, up to shrinking $U$,  it is enough to prove the case when $l_i$ is the Cartier index of $D_i$ at $x$. 
Then there are integers $d_i > 0$ and   unit functions $\eta_i$ on $U$,  
such that $r_i=d_il_i$ and $s_i = \eta_i \cdot t_i^{d_i}$. 
Up to shrinking $U$, we may assume that $\eta_i$ admits a $(nr_i)$-th root $\gamma_i$.   
Then  the normalization of the $\mathcal{O}_U$-algebra
\begin{eqnarray*}
 &&\mathcal{O}_U [T_1,...,T_m] /(T_1^{n r_1} - s_1,..., T_m^{n r_m}-s_m)\\
& = & 
 \mathcal{O}_U [T_1,...,T_m] /(T_1^{n d_1 l_1} - \gamma_1^{n d_1 l_1} \cdot  t_1^{d_1},..., T_m^{n d_m l_m}- \gamma_m^{n d_m l_m} t_m^{d_m})  
\end{eqnarray*}
is isomorphic to the product of copies of the normalization of
\[
 \mathcal{O}_U [T_1,...,T_m] /(T_1^{nl_1} - \gamma_1^{nl_1} \cdot  t_1,..., T_m^{nl_m}-\gamma_m^{nl_m} \cdot t_m). 
\]
Since each $\gamma_i$  is a unit function, we deduce that 
\begin{eqnarray*}
    && \mathcal{O}_U [T_1,...,T_m] /(T_1^{nl_1} - \gamma_1^{nl_1} \cdot  t_1,..., T_m^{nl_m}-\gamma_m^{nl_m} \cdot t_m)  \\
 &\cong&  
 \mathcal{O}_U [T_1,...,T_m] /(T_1^{nl_1} - t_1,..., T_m^{nl_m}-t_m). 
\end{eqnarray*} 
This implies the lemma. 
\end{proof}

The following lemma summarizes some geometric properties on cyclic covers.

\begin{lemma}
\label{lemma:cyclic-cover-Cartier} 
With the notation above, we set  $X'= X[\sqrt[nr_1]{s_1}, ..., \sqrt[nr_m]{s_m}]^{nor}$. 
Assume that every $D_i$ is reduced and irreducible. 
Then the following properties hold. 
\begin{enumerate}
\item If   $D'_i = p^{-1}( D_i)$, then $D'_i$ is a Cartier divisor in $X'$. 
Furthermore, the ramification index of $p$ along any component of  $D'_i$ is equal to $n$. 
\item If $(X, \sum_{i=1}^m D_i)$ is snc, then so is $(X', \sum_{i=1}^m D'_i)$.
\item If $(X, \sum_{i=1}^m D_i)$ is dlt, then so is $(X', \sum_{i=1}^m D'_i)$.
\end{enumerate}

\end{lemma}

\begin{proof}
For the item (1), we first prove the statement on the ramification indices of $p$.
Since $X$ is normal and $D_i$ is reduced, we deduce that both $X$ and $D_i$ are smooth  at a general point $x\in D_i$. 
In a neighborhood $U$ of $x$ in $X$, every $s_j$ is a unit if $j\neq i$. 
Furthermore, $D_i$ is defined by $t_i=0$ for some holomorphic function $t_i$ on $U$, 
and $s_i = \eta \cdot t_i^{r_i}$ for some unit holomorphic function $\eta$ on $U$.  
By Lemma \ref{lemma:cyclic-functorial}, up to shrinking $U$,  
the preimage $p^{-1}(U)$ is a disjoint union of  copies of $Y$, 
where $Y=U[\sqrt[n]{t_i}]$. 
Since  $Y\to U$ is totally branched over $D_i\cap U$, 
we deduce that  the  ramification index of  $p$ along any component of  $D'_i$ is equal to $n$.

We will now prove that $D_i'$ is Cartier. 
We notice that $T_i$ is a holomorphic function on $X'$ and its zero locus is equal to $D_i$'. 
It is enough to show that the vanishing order of $T_i$ along every component of  $D_i'$ is 1. 
The vanishing order of $s_i$ along $D_i$ is $r_1$. 
From the result of the previous paragraph, it follows that the vanishing order of $s_i$, 
viewed as a holomorphic function on $X'$, along every component of $D_i'$ is equal to $nr_i$.  
By definition, $T_i^{nr_i} = s_i$. 
Thus the vanishing order of $T_i$ along every component of  $D_i'$ is 1. 

For the item (2), we let $D=\sum_{i=1}^m D_i$ and $D' =p^{-1}(D)$.  
We first observe the following fact.   
Let $x\in X$ be a point. 
Since $D_i$ is smooth, it is irreducible at $x$. 
Then in  an open neighborhood of $x$, 
$X$ is isomorphic to a  polydisc and each $D_i$ is either empty 
or isomorphic to  a  coordinate hyperplane. 
Then we deduce that $(X',D')$ is log smooth at every point $x'\in X'$ lying over $x$.  
In other words, there is an open neighborhood $U'$ of $x'$ such that  $(U',D'|_{U'})$ is snc.

Next we  will show that every irreducible component of $D'$ is normal.    
Without loss of the generality, we only need to show that $D'_1$ is normal.  
Assume by contradiction that there is a non-normal locus $V'\subseteq D'_1$. 
Then $V'$ is a non-klt center of $(X',D'_1)$ by \cite[Proposition 5.51]{KollarMori1998}. 
By \cite[Proposition 5.20]{KollarMori1998} and its proof, this implies that $V=p(V')$ is  a non-klt center  of the pair $(X,D_1)$. 
Since both $X$ and  $D_1$ are smooth. we can only have $V= \emptyset$. 
This is a contradiction. 

In the last step, let $S$ be a stratum of $D'$ and let $x'\in S$.  
We have seen that there is an open neighborhood $U'$ of $x'$ such that  $(U',D'|_{U'})$ is snc. 
Since the irreducible components of $D'$ are normal, up to shrinking $U'$, 
their restrictions on $U'$ are either empty or irreducible. 
Hence $S|_{U'}$ is a stratum of $D'|_{U'}$.  
It follows that $S$ is smooth at $x'$.  
Thus $(X',D')$ is snc. 

For the item (3), by the definition of dlt pairs, 
there is a Zariski closed subset $W\subseteq X$ of codimension at least 2,  
such that $(X,D)$ is snc over $X^\circ:= X\setminus W$. 
Furthermore, $W$ does not contain any non-klt center of $(X,D)$.   
Let $W'= p^{-1}(W)$ and let $X'^\circ = p^{-1}(X^\circ)$. 
Then the item (2)  implies that $(X',D')$ is snc on $X'^\circ$.  
Furthermore, by \cite[Proposition 5.20]{KollarMori1998} and its proof,    
if $Z'$ is a non-klt center  of $(X',D')$, then  $p(Z')$  is a non-klt center of $(X,D)$.   
Hence $W'$ does not contain any non-klt center of $(X',D')$.  
Therefore  $(X',D')$ is dlt. 
This completes the proof of the lemma. 
\end{proof}

In the following lemma, we will use the  basechange by a cyclic cover to kill the divisorial critical locus of a given finite morphism.

\begin{lemma}
    \label{lemma:local-dlt-cover}
Let $(X,\Delta)$ be a reduced dlt pair.  
Let $D_1,...,D_m$ be the irreducible components of $\Delta$. 
Assume that $r_iD_i$ is defined by $s_i=0$ for some holomorphic function $s_i$ on $X$ and for some integer $r_i > 0$. 
Let $\pi\colon Y\to X$ be a finite morphism with $Y$ normal,  
whose divisorial critical locus is contained in $\Delta$.   
Let $N$ be the lcm of all the coefficients in $\pi^*D_i$ for $i=1,...,m$. 
Then for  any positive integer $n$ divisible by $N$, the following properties hold. 
Let $V = X[\sqrt[nr_1]{s_1}, ..., \sqrt[nr_m]{s_m}]^{nor}$ and let $W$  be the normalization of   $Y\times_{X } V$. 
Then the natural morphism $\mu \colon W \to V$ is quasi-\'etale. 
Moreover, if $\pi^{-1}(X\setminus \Delta) $ is smooth, then $W$ is smooth in codimension 2. 
\begin{equation*}
  \xymatrixcolsep{3pc}
  \xymatrix{ 
  W \ar[r]  \ar[d]_\mu & Y \ar[d]^\pi\\
  V\ar[r]_p & X
  }
\end{equation*}
\end{lemma}

\begin{proof}
We fix a positive integer $n$ divisible by $N$. 
Let $p\colon V\to X$ be  the natural morphism.   
Then the  critical locus of $p$ is exactly $\Delta$. 
Let $H_i= p^{-1}(D_i)$ for $i=1,...,m$ and let $H= \sum_{i=1}^m H_i$.  
Then $\mu$ is quasi-\'etale over $V\setminus H$. 
Furthermore,  each $H_i$ is Cartier and $p^*D_i = nH_i$ by Lemma \ref{lemma:cyclic-cover-Cartier}.  
The latter property implies that $\mu$ is \'etale over general 
points of every component of $H_i$, for $n$ is divisible by $N$.  
As a consequence, $\mu$ is quasi-\'etale.

Suppose that $\pi^{-1}(X\setminus \Delta) $ is smooth. 
We assume by contradiction that  there is an irreducible component $Z\subseteq W_{\sing}$  such that $\dim Z = \dim W-2$. 
Since $p$ is \'etale over $X\setminus \Delta$, we deduce that $W\to Y$ is \'etale over $\pi^{-1}(X\setminus \Delta)$. 
Therefore,  $Z$ is contained in $(p\circ \mu)^{-1}(\Delta)$. 
Then $\mu(Z) \subseteq  H$. 
We note that $(V,H)$ is dlt by Lemma \ref{lemma:cyclic-cover-Cartier}.  
Since $H$ is Cartier, by (16.6.1.1) and (16.6.3.2) of \cite[Proposition 16.6]{Kollar1992}, 
we deduce that $V$ is smooth around general points of $\mu(Z)$. 
By the Zariski's purity theorem, we conclude that  $\mu$ is \'etale over general points of $\mu(Z)$. 
Thus $W$ is smooth  around general points of $Z$. 
This is a contradiction. 
\end{proof}

We note that, in the previous lemma, if $\pi$ is Galois, then the number  $N$ divides the degree of $\pi$. 
We will also need the following elementary result.

\begin{lemma}
    \label{lemma:cubic}
Let $S = \mathbb{D}^N$ be a polydisc with coordinates $(T_1,...,T_N)$, 
let  $H$ be the divisor   defined by  $T_1\cdots T_n =0$ for some integer  $1\le n \le N$, 
and let 
\[F(y,s)  = y^3 + a(s)y^2 + b(s)y + c(s)\] 
be a cubic polynomial in $y$ whose coefficients are holomorphic functions on $S$. 
Let $X\subseteq \mathbb{C} \times S$ be the closed analytic subspace defined by $F(y,s) = 0$, 
and let $\pi\colon X\to S$ be the natural projection. 
Assume that, either the divisorial  critical locus of $\pi$ is contained in  $H$, 
or the critical locus is $S$.   
In other words, either $F(y,s)$ has three distinct roots in $y$ for any point $s\in S\setminus H$ fixed, 
or  $F(y,s)$ has a multiple root in $y$ for a general point $s\in S$ fixed. 

Let $\overline{S} = S[\sqrt[6]{T_1}, ..., \sqrt[6]{T_n}]$ 
and let $\mu\colon \overline{S} \to S$ be the natural morphism. 
We define  
\[G(y,\overline{s}) =   y^3 + (a\circ \mu)(\overline{s})y^2 + (b\circ \mu)(\overline{s})y + (c\circ \mu)(\overline{s}) \] 
as the pullback of $F$ on $\mathbb{C} \times \overline{S}$. 
Then there is a factorization 
\[G(y,\overline{s}) = (y - p(\overline{s}))(y - q(\overline{s}))(y - r(\overline{s})) \]
for some holomorphic functions $p,q,r$ on $\overline{S}$. 
\end{lemma}

\begin{proof}
First we assume that the critical locus of $\pi$ is the entire space  $S$. 
In this case, $X$ is non-reduced, and has either one or two irreducible components. 
If there is only one component, then with the reduced structure, it is bimeromorphic to $S$. 
Since $S$ is smooth, it follows that $X_{\red}$ is isomorphic to $S$. 
It follows that  $F(y,s) = (y-p(s))^3$ for some holomorphic function $p$ on $S$. 
If $X$ has  two components $X_1$ and $X_2$, 
then one of them is reduced and the other one is non-reduced with multiplicity 2. 
Hence the same argument as before shows that $F(y,s) = (y-p(s))^2(y-q(s))$ for some holomorphic functions $p$ and $q$ on $S$.  
In conclusion, in this case, the function $F$ already admits a factorization.  
Hence its pullback $G$ also admits a factorization.

Now we assume that the critical locus of $\pi$ is a proper subset of $S$.  
Let $X'$ be the normalization of $X$ and let $\pi'\colon X' \to S$ be the natural morphism. 
Since the degree of $\pi'$ is 3, we see that the ramification index of $\pi'$ along any prime divisor in $X'$ is contained in $\{1,2,3\}$.  
Let $\overline{X}$ be the hypersurface defined by $G$ and let $Y$ be its normalization.  
Then $Y$ is also the normalizaiton of $X'\times_S \overline{S}$, 
and the natural finite morphism $Y\to \overline{S}$ is quasi-\'etale by Lemma \ref{lemma:local-dlt-cover}. 
The Zariski's purity theorem implies that it is indeed \'etale, for $\overline{S}$ is smooth. 
Since $\overline{S}$ is simply connected, we deduce that $Y$ has three connected components. 
Hence $\overline{X}$ has three irreducible components. 
The same argument as in the previous paragraph implies that each component is isomorphic to $\overline{S}$. 
Therefore,   there is a factorization 
\[G(y,\overline{s}) = (y - p(\overline{s}))(y - q(\overline{s}))(y - r(\overline{s}))\]
for some holomorphic functions $p,q,r$ on $\overline{S}$. 
\end{proof}

\subsection{Notion of  complex analytic orbispaces}

The notion of $\mathcal{C}^\infty$ orbifolds was first introduced in \cite{Satake1956}.  
A complex orbifold is   a $\mathcal{C}^\infty$ orbifold whose orbifold charts are complex manifolds such that the group actions on them are holomorphic.  
In this subsection, we will fix the notion of  (reduced)  complex analytic orbispaces, which is a straightforward extension of the notion of complex orbifolds, 
just by allowing orbifold charts to have singularities.  
Particularly, a complex orbifold is a complex analytic orbispace whose orbifold charts are all smooth. 

\begin{defn}\label{def:orbifold}
We take the following definitions.
\begin{enumerate}
    \item Let $U$ be a connected Hausdorff space. An  orbifold chart   of $U$ is a triple $(\widetilde{U}, G, \pi)$ such that $\widetilde{U}$ is a reduced complex analytic space of dimension $n$,  that $G$ is a finite group acting  holomorphically on $\widetilde{U}$, and that $\pi\colon \widetilde{U}\to U$ is a continuous map inducing a homeomorphism from $\widetilde{U}/G$ to $U$.   
    We denote by $\mathrm{ker}\, G$ the maximal subgroup of $G$ which acts trivially on $\widetilde{U}$.

    \item Let $\iota \colon U' \hookrightarrow  U$ be an open embedding, and let $(\widetilde{U}',G',\pi')$ be an orbifold chart of $U'$. 
          An  injection  from $(\widetilde{U}',G',\pi')$ to  $(\widetilde{U}, G, \pi)$ consists of an open embedding $\varphi \colon \widetilde{U}'\to \widetilde{U}$ and a group monomorphism $\lambda \colon G'\to G$ such that $\varphi$ is equivariant with respect to $\lambda$,
          that $\lambda$ induces an isomorphism from  $\mathrm{ker}\, G' $ to $\mathrm{ker}\, G$, 
          and that $\iota\circ \pi' = \pi\circ \varphi$.  

    \item Let $X$ be a connected second countable Hausdorff space. 
          An  orbifold atlas  on $X$ is a collection $\mathcal{U} = \{(\widetilde{U},G,\pi)\}$ of orbifold charts of open subsets of $X$ which covers $X$ and is compatible in the following sense. 
          For any two  orbifold charts $(\widetilde{U},G,\pi)$ and $(\widetilde{U}',G',\pi')$ in $\mathcal{U}$,  
          of open subsets $U$ and $U'$ respectively, for any point $x\in U\cap U'$, there is an open neighborhood $V$ of $x$ with an orbifold chart  $(\widetilde{V},H,\rho)$ such that there are injections from $(\widetilde{V},H,\rho)$ to  $(\widetilde{U},G,\pi)$ and  $(\widetilde{U}',G',\pi')$.  
          
    \item With the notation in (3), there is a unique maximal  orbifold atlas $\mathcal{U}_{max}$ on $X$ which contains $\mathcal{U}$. 
          Two orbifold atlases  are   equivalent if they are contained in the same maximal atlas.
          An orbifold chart of some open subset of $X$ is said to be compatible with $\mathcal{U}$ if it belongs to $\mathcal{U}_{max}$. 

    \item A complex analytic orbispace $\mathfrak{X}$  consists of 
          a connected second  countable Hausdorff space $X$ and a maximal orbifold atlas $\mathcal{U}_{max}$ on it. 
          If a collection  $(X_i,G_i,\pi_i)_{i\in I}$ of orbifold charts over $X$ forms an orbifold atlas, we write $\mathfrak{X} = (X_i,G_i,\pi_i)_{i\in I}$ for the complex analytic orbispace induced by $(X_i,G_i,\pi_i)_{i\in I}$, 
    and we call $X$ the quotient space of $\mathfrak{X}$. 
\end{enumerate}
\end{defn}

Throughout this paper, we may  write $\mathfrak{X} = (X_i,G_i)_{i\in I}$ or just $\mathfrak{X} = (X_i,G_i)$ for a complex analytic orbispace if no confusion seems likely to result.  
We note that the item (3) in the previous definition does make sense by  the same argument as in \cite[Lemma 4.1.1]{ChenRuan2002}. 
Furthermore, the  compatibility condition can be interpreted in an equivalent way as in the following lemma.

\begin{lemma}
\label{lemma:orbispace-def}
Let $X$ be a  connected second countable Hausdorff space.  
Let $(X_i,G_i,\pi_i)$ be a family of orbifold charts over $X$, such that  their images cover  $X$.  
Then $(X_i,G_i,\pi_i)$ defines a complex analytic orbispace   if the following properties hold. 
For any point $x\in \pi_i(X_i) \cap \pi_j(X_j)$, there is  an open neighborhood $U$ of $x$ in $\pi_i(X_i) \cap \pi_j(X_j)$, such that every connected component of  $\pi_i^{-1}(U)$  is isomorphic to every connected component of $\pi_j^{-1}(U)$, as finite covers over $U$. 
\end{lemma}

\begin{proof}  
Let $W$ be a connected component of   $\pi_i^{-1}(U)$.  
Then $(W,H)$  is an orbifold chart of $U$, where $H\subseteq G_i$ is the subgroup which leaves $U$ invariant. 
Then there is a natural injection of orbifold chart from $W$ to $X_i$. 
The assumption implies that there is also an inclusion from $W$ to $X_j$. 
Hence the condition in (3) of Definition \ref{def:orbifold} is satisfied. 
This completes the proof of the lemma. 
\end{proof}

A main objective that we use the notion of  complex analytic orbispaces is to reduce the study of klt singularities to the one of canonical singularities.   
 
\begin{example}
\label{example: complex analytic orbispace}
Assume that $X$ is a complex analytic variety with klt singularities. 
Then for every point $x\in X$, there is an open neighborhood $U(x)$, and a smallest positive integer $k(x)$, such that $(\omega_X^{\otimes k(x)})^{**}|_{U(x)} \cong \mathcal{O}_{U(x)}$, where $\omega_X$ is the canonical sheaf of $X$. 
Let $V(x)\to U(x)$ be the cyclic cover induced by this isomorphism (see \cite[Definition 2.52]{KollarMori1998}), 
and let $G(x)$ be the Galois group. 
Then $V(x)$ has canonical singularities. 
Furthermore,  the collection of the $(V(x),G(x))$'s induces a complex analytic orbispace with quotient space $X$, see \cite[Lemma 2.3]{DasOu2023}. 
\end{example}

We will need  the following lemma, which performs cyclic covers on   complex analytic orbispaces.

\begin{lemma}
\label{2-construction-2}
Let $\mathfrak{X}=(X_i,G_i)_{i\in I}$ be a complex analytic orbispace with quotient space $X$. 
Let $\delta_i$ be a holomorphic function on $X_i$ such that the ideal sheaf $\mathcal{J}_i$ generated by $\delta_i$ is compatible along the overlaps. 
In other words, if $p_i\colon V\to X_i$ and $p_j\colon V\to X_j$ are inclusions of orbifold charts, 
then $p_i^* \mathcal{J}_i = p_j^{*}\mathcal{J}_j$ are the same ideal sheaf on $V$.   
Let $S\subseteq X$ be a compact subset and let $S_i=\pi_i^{-1}(S)$.

Then there is a finite family $(U_k,G'_k)_{k\in K}$ of orbifold charts satisfying the following properties. 
There is an application $\sigma \colon K\to I$ and there are inclusions $\iota_k\colon U_k \to X_{\sigma(k)}$ of orbifold charts. 
Let $n>0$ be  an integer and let $W_k=U_k[\sqrt[n]{\delta_{\sigma(k)}}]$.  
Then the natural morphism $W_k \to X$ is Galois over its image with Galois group $H_k$, 
and the family $(W_k,H_k)_{k\in K}$ induces a complex analytic orbispace, whose quotient space is an open neighborhood of $S$. 
\end{lemma}

\begin{proof}
For any $i\in I$,  the assumption implies that there is a unit function $u_g$ such that $g^*\delta_i = u_g \cdot \delta_i$ for any $g\in G_i$. 
For every point $x\in S_i$, there is an open neighborhood $U_x$ of $x$ in $X_i$, 
such that $u_g|_{U_x}$ admits $n$-th roots. 
Let $G_x\subseteq G_i$ be the stabilizer of $x$. 
By shrinking $U_x$, we may assume that $(U_x,G_x)$ is an orbifold chart compatible with $\mathfrak{X}$.  
By the same argument of Lemma \ref{lemma:cyclic-cover-unit}, 
the automorphisms of $U_x$ over $U_x/G_x$ can be lifted to automorphisms of $U_x[\sqrt[n]{\delta_i}]$. 
This implies that the natural morphism $U_x[\sqrt[n]{\delta_i}] \to X$ is Galois over its image.  

By considering all points $x\in S_i$ and all $i\in I$, we obtain a family of orbifold charts $(U_x,G_x)$. 
Since $S$ is compact, we can extract a finite family of points $(x_k)_{k\in K}$, with an application $\sigma\colon K\to I$, such that $x_k\in X_{\sigma(k)}$, and that the images of $U_{x_k}$ in $X$ covers $S$. 
We let $U_k = U_{x_k}$ and $G'_k = G_{x_k}$. 
Then the previous paragraph implies that $W_k\to X$ is Galois over its image, with Galois group $H_k$. 
By Lemma \ref{lemma:cyclic-cover-unit} and Lemma \ref{lemma:orbispace-def}, we deduce that the family $(W_k,H_k)_{k\in K}$ induces a complex analytic orbispace.
\end{proof}

From  \cite[Th\'eor\`eme 4]{Cartan1957}, we see that the 
quotient space $X$ of a complex analytic orbispace $\mathfrak{X} = (X_i,G_i,\pi_i)$ has a natural structure of complex analytic space. 
The   holomorphic functions on any open subset $U\subseteq \pi_i(X_i)$ are exactly the $G_i$-invariant holomorphic functions on $\pi_i^{-1}(U)\subseteq X_i$.
Furthermore, if every   $X_i$ is normal, then so is $X$.  
In this case, we say that $\mathfrak{X}$ is a normal  complex analytic orbispace.

We can consider proper bimeromorphic morphisms between  complex analytic orbispaces. 
They are closely related to proper bimeromorphic morphisms between the quotient spaces, as we explain as follows. 
Let  $\mathfrak{X} = (X_i,G_i,\pi_i)$  be a complex analytic orbispace with quotient space $X$. 
First we assume that there are $G_i$-equivariant proper bimeromorphic morphisms $f_i\colon Y_i \to X_i$ which are compatible along the overlap in the following sense. 
For any inclusions $\varphi_i \colon U \to X_i$ and $\varphi_j\colon U\to X_j$, the fiber products $U\times_{X_i} Y_i$ and $U {\times}_{X_j} Y_j$ are isomorphic over $U$.  
We note that such an isomorphism must be unique for  $f_i$ and $f_j$ are bimeromorphic. 
Then the collection of the $(Y_i,G_i)$'s induces a complex analytic orbispace $\mathfrak{Y}$ with quotient space $Y$. 
Furthermore, there is an induced proper bimeromorphic morphism $f\colon Y\to X$, so that $X_i \times_{X} Y \cong Y_i$ and the corresponding basechange of $f$ coincides with $f_i\colon Y_i \to X_i$.

Conversely, we assume that there is a proper bimeromorphic morphism $f\colon Y \to X$ with $Y$ normal. 
Let  $Y_i$ be the normalization of the main component of  $X_i\times_X Y$. 
Then there is a natural action of $G_i$ on $Y_i$ and the natural projection $f_i\colon Y_i \to X_i$ is $G_i$-equivariant. 
Moreover, the collection of the $(Y_i,G_i)$'s defines a complex analytic orbispace with quotient space $Y$. 

We   now introduce the following notion of (reduced) orbi-divisors, and prove a result on resolving the singularities of an orbi-divisor.

\begin{defn}
\label{defn:orbi-divisor} 
Let $ \mathfrak{X}=(X_i,G_i,\pi_i)_{i\in I} $ be a normal  complex analytic orbispace with quotient space $X$. 
An  orbi-divisor $\mathfrak{D}$ on $\mathfrak{X}$ is a collection of divisors $(D_i)_{i\in I}$ such that each $D_i$ is a reduced $G_i$-invariant  divisor in  $X_i$,  and that they are compatible along the overlaps. 
The images of the $D_i$'s in $X$  then define  a reduced divisor $E$ in $X$. 
The orbi-divisor $\mathfrak{D}$ is called a snc orbi-divisor if for any stratum $V$  of $E$,  including the case when $V=X$,  
the preiamge $\pi_i^{-1}(V)$ is  smooth in $X_i$.   
\end{defn}

We note that $\mathfrak{D}$ is a snc orbi-divisor
if and only if  $(X_i,\pi_i^{-1}(E))$ is a snc pair for all $i\in I$,  
and for any stratum $V$ of $E$, 
the preiamge $\pi_i^{-1}(V)$ is a disjoint union of strata of $\pi_i^{-1}(E)$. 
However, the condition that $(X_i,\pi_i^{-1}(E))$ is snc is not sufficient for our definition of snc orbi-divisor, as shown in the following example.

\begin{example}
\label{example:orbi-divisor}
Let $W = \mathbb{C}^2$ with coordinates $(a,b)$ and let $G=\mathbb{Z}/4\mathbb{Z} = <g>$.  
Then there is an action of $G$ on $W$ by setting $g.(a,b) = (-b,a)$.  
Then $(W,G,\pi)$ induces an orbifold with quotient space $X=W/G$. 
Let $D$ be the sum of the coordinates axes in $W$. 
Then $D$ is $G$-invariant and its image $E$ in $X$ is irreducible. 
We remark that the pair $(W,D)$ is snc, but $D$ does not induce a  snc orbi-divisor, for $D=\pi^{-1}(E)$ is not smooth. 
\end{example}

\begin{lemma} 
    \label{lemma:snc-orbidivisor}
Let $ \mathfrak{X}=(X_i,G_i,\pi_i) $ be a normal   complex analytic orbispace with quotient space $X$.  
Assume that    $\mathfrak{D}$ is an orbi-divisor in $\mathfrak{X}$.  
Then there are  $G_i$-invariant  projective bimeromorphic morphisms $f_i\colon  Y_i\to X_i$, 
such that the collection of the
$ (Y_i,G_i)$'s induces a complex orbifold $\mathfrak{Y}$ with quotient space $Y$, 
and that the preimage of $\mathfrak{D}$ in $\mathfrak{Y}$ is a snc orbi-divisor.  
\end{lemma}

\begin{proof}
Let $E$ be the image of $\mathfrak{D}$ in $X$ and let  $D_i = \pi_i^{-1}(E)$. 
Let $p_i\colon X'_i\to X_i$ be the functorial principalization of the ideal sheaf of $D_i$, 
see Theorem \ref{thm:func-principalization} in the next section. 
Then up to replacing $X_i$ by $X'_i$, we can assume that $X_i$ is smooth and $D_i$ is a snc divisor.

We note that, if $K$ and $K'$ are strata  of $D_i$ of minimal dimension, 
then   either $K'=K$ or $K'\cap K = \emptyset$. 
Let $r$ be the minimum of the dimensions of the strata  of $D_i$ for all $i$.  
Let $p_i\colon X'_i \to X_i$ be the blowup at the union of  all strata of $D_i$ of dimension $r$.  
Then the center of the blowup $p_i$ is either empty or smooth. 
Moreover, there is a natural action of $G_i$ on $X'_i$ so that $p_i$ is $G_i$-invariant.  
The preimage  $p_i^{-1}(D_i)$ is a $G_i$-invariant snc divisor in $X_i'$.  
Furthermore, the collection of the $(X'_i,G_i)$'s induces a   complex orbifold $\mathfrak{X}'=(X'_i,G_i,\mu_i)$ with quotient space $X'$. 
There is a natural projective bimeromorphic morphism $p\colon X' \to X$, 
with a $p$-ample and $p$-exceptional divisor.  
If $\Delta\subseteq X'$ is the exceptional set of $p$, 
then $\Delta$ is pure of codimension 1 and $\mu_i^{-1}(\Delta)$ is the exceptional set of $p_i$.
In particular, $\mu_i^{-1}(\Delta)$ is smooth and is the disjoint union of its irreducible components.

We continue this procedure by blowing up $X'_i$ at the union of all $(r+1)$-dimensional strata of the strict transform $(p_i^{-1})_*D_i$. 
By induction on the dimension of strata, we obtain projective bimeromorphic morphisms $f_i\colon Y_i \to X_i$ such that the following properties hold. 
\begin{enumerate}
    \item There is a natural action of $G_i$ on $Y_i$ so that $f_i$ is $G_i$-equivariant. 
    \item We have a complex  orbifold $ \mathfrak{Y} = (Y_i,G_i,\nu_i)$ with quotient space $Y$.  
    \item There is an induced bimeromorphic morphism $f\colon Y\to X$, whose exceptional locus is pure of codimension 1.  
    \item For any  prime $f$-exceptional $\Gamma$, the preimage $\nu_i^{-1}(\Gamma)$ is smooth. 
    \item The strict transform $(f_i^{-1})_*D_i$ in $Y_i$ is smooth. 
    \item $f_i^{-1}(D_i)$ is a snc divisor. 
\end{enumerate}
For the preimage of $\mathfrak{D}$ in $\mathfrak{Y}$, if $E'\subseteq Y$ is its image, 
then  the items (4) and (5) above imply  that, for every irreducible component of $\Lambda$ of $E'$, 
 the preimage $\nu_i^{-1}(\Lambda)$ is smooth,  hence is the disjoint union of its irreducible components. 
We then deduce that for any stratum $V$ of $E'$, the preimage $\nu_i^{-1}(V)$ is a disjoint union of strata of $f_i^{-1}(D_i)$.  
Thus  the preimage of $\mathfrak{D}$ in $\mathfrak{Y}$ is a snc orbi-divisor.
We conclude that morphisms $f_i$ satisfy the conditions of the lemma. 
\end{proof}

\subsection{Minimal Model Programs for complex analytic spaces}
 
Let $f\colon X\to Y$ be a projective surjective morphism  between   complex analytic spaces. 
We assume that either $Y$ is compact or $Y$ is a germ around a point $y\in Y$. 
Assume that $(X,\Delta)$ is a dlt pair and $X$ is $\mathbb{Q}$-factorial. 
We recall that  $X$ is $\mathbb{Q}$-factorial if every Weil divisor in $X$ is $\mathbb{Q}$-Cartier, 
and some reflexive power  $(\omega_X^{\otimes k})^{**}$ with $k \in \mathbb{Z}_{>0}$ of the canonical sheaf $\omega_X$ is locally free. 
In this case, we write $K_X$ for the class of $\frac{1}{k}c_1((\omega_X^{\otimes k})^{**})$.  
Then the $f$-relative cone theorem  and contraction theorem 
hold, see \cite[Section 4]{Nakayama1987},  \cite[Section 2.6]{DasHaconPaun2022} and \cite[Section 7]{Fujino2022}.   
Furthermore, if $H$ is a $f$-ample $\mathbb{Q}$-divisor, 
we can run a $f$-relative MMP for $(X,\Delta)$, with the scaling of $H$, see  \cite[Theorem 1.4]{DasHaconPaun2022} or \cite[Theorem 1.7]{Fujino2022}.  
The following lemma is essentially proved in \cite[Lemma 4.2]{Birkar2012}. 

\begin{lemma}
\label{lemma:MMP-scaling-birkar}
With the notation above, assume that 
\[
(X,\Delta) = (X_0,\Delta_0) \dashrightarrow (X_1,\Delta_1) \dashrightarrow \cdots 
\]
is the sequence of $f$-relative MMP, with the scaling of $H$. 
Assume that $\lambda_0 \ge \lambda_1 \ge\cdots$ is the sequence of the nef threshold for $H_i$ with respect to $K_{X_i}+\Delta_i$. 
In other words, 
\[
\lambda_i = \inf \{s\ge 0 \ | \  K_{X_i} + \Delta_i + sH_i \mbox{ is nef}\}.
\]
Then either $ \lim_{i} \lambda_i = 0$, or the previous sequence of MMP terminates.  
\end{lemma}

\begin{proof}
We set $\lambda = \lim_{i} \lambda_i$ and  assume that $\lambda > 0$.  
Let $\mu = \frac{1}{2}\lambda$. 
We first consider the case when $(y\in Y)$  is a germ of complex analytic space.  
In particular, we may assume that $Y$ is Stein.  
By \cite[Theorem 2.12]{DasHaconPaun2022}, there is some rational number $\varepsilon >0$ such that 
$\mu H+ \varepsilon \Delta$ is $f$-ample. 
Then, by using \cite[Lemma 2.21]{DasHaconPaun2022},   
we can obtain an effective $\mathbb{Q}$-divisor $D$ such that $D\sim_{\mathbb{Q}}   \mu H +\varepsilon \Delta$ and that $(X, (1-\varepsilon)\Delta + D)$ is klt.   
We set $\Delta' = (1-\varepsilon)\Delta + D$. Then 
\[
K_X+\Delta' \sim_{\mathbb{Q}} K_X+\Delta + \mu H. 
\]
Therefore, the sequence $(X_i)$ is also a sequence of $f$-relative MMP for $K_X+\Delta'$, with the scaling of $(\lambda_0-\mu)H$. 
Since $(X,\Delta')$ is klt, such a MMP must terminate  by \cite[Theorem 1.4]{DasHaconPaun2022} or \cite[Theorem 1.7]{Fujino2022}. 

In the case when $Y$ is a compact complex analytic variety, the previous argument shows that the sequence $(X_i)$  must be stationary over a neighborhood of any point of $Y$. 
Since $Y$ is compact, it follows that the MMP must terminate. 
\end{proof}

We use the MMP to construct certain bimeromorphic models in the following two lemmas.

\begin{lemma}
    \label{lemma:QcDV-modification} 
Let $X$ be a   compact  complex analytic variety.    
Let $\{S_1,...,S_r\}$ be a collection of  irreducible components of $X_{\sing}$ of codimension 2, such that 
$X$ has quotient  singularities around general points of $S_i$ for all $i=1,...,r$.    
Then there is a  projective bimeromorphic morphism $f\colon Y\to X$ such that the following properties hold. 
\begin{enumerate}
    \item $f$ is an isomorphism over $X_{\sm}$. 
    \item $f$ is an isomorphism over  general points of $S_i$ for all $i=1,...,r$. 
    \item $Y$ has $\mathbb{Q}$-factorial klt singularities.  
    \item The codimension 2 part of  $Y_{\sing} $ is  equal to $  \bigcup_{i=1}^r  f^{-1}_*S_i$.

\end{enumerate}
\end{lemma} 

\begin{proof}
    Let $r\colon \widetilde{X} \to X$ be a  projective log resolution which is an isomorphism over $X_{\sm}$. 
    Let $E$ be the reduced sum of all  $r$-exceptional  prime divisors whose centers in  $X$ are one of the $S_i$'s.   
    Let $\varepsilon>0$ be a small enough rational number such that $-(1-\varepsilon)$ is smaller than the discrepancies for $K_X$ around general points of the $S_i$'s.  
    Then, by \cite[Theorem 1.4]{DasHaconPaun2022}, we can run a $r$-relative MMP for the klt pair $(\widetilde{X}, (1-\epsilon)E)$, with the scaling of certain $r$-ample divisor,  
    and obtain a projective bimeromorphic morphism $f\colon Y \to X$.    
    Then  $Y$ has $\mathbb{Q}$-factorial klt singularities.  
    The negativity lemma shows that $E$ is contracted by $\widetilde{X} \dashrightarrow Y$. 
    This implies that  $f$ is an isomorphism over general points of every $S_i$. 

    It remains to prove the item (4). 
    Up to replacing $X$ by $Y$ in the previous paragraph, we can assume that $X$  has $\mathbb{Q}$-factorial klt singularities.   
    By \cite[Proposition 2.36]{KollarMori1998},  we may assume further that $\widetilde{X}$ contains all divisors over $X$, whose  discrepancies  with respect to $K_X$  are at most 0.  
    Let $\Gamma$ be a prime divisor over $Y$ which has discrepancy at most 0 for $K_Y$.  
    On the one hand, since $K_Y$ is nef over $X$, the negativity lemma 
    implies that $K_Y + B \sim_{\mathbb{Q}} f^*K_X $ for some $f$-exceptional $\mathbb{Q}$-divisor $B\ge 0$. 
    Thus the discrepancy of $\Gamma$ for $K_X$ is at most 0. 
    Hence the center of $\Gamma$ in $\widetilde{X}$ is a divisor. 
    On the other hand, since the MMP $\widetilde{X} \dashrightarrow Y$ is  $(K_{\widetilde{X}} + (1-\varepsilon )E)$-negative, 
    the discrepancy of $\Gamma$ for $K_{\widetilde{X}} + (1-\varepsilon )E$ is at most 0 as well.  
    This implies that the center of $\Gamma$ in $\widetilde{X}$ is contained in $E$.  
    Therefore, we conclude that $\Gamma$ is a component of $E$. 
    Hence $Y$ has terminal singularities outside $\bigcup_{i=1}^r  f^{-1}_*S_i$. 
    It follows that  the codimension dimension 2  part of $Y_{\sing}$ is equal to $ \bigcup_{i=1}^r  f^{-1}_*S_i$. 
\end{proof}

\begin{lemma}
    \label{lemma:dlt-modification} 
Let $(X,D)$ be a reduced  pair, where $X$ is a  normal compact  complex analytic variety.    
Let $\{S_1,...,S_r\}$ be the collection of all codimension 2 irreducible components of $X_{\sing}$, which are not contained in $D$.   
Assume that $X$ has quotient singularities around general points of  $S_i$ for all $i=1,...,r$. 
Then there exists a  projective bimeromorphic morphism $f\colon Y\to X$ such that the following properties hold. 
\begin{enumerate} 
    \item $f$ is an isomorphism over  the snc locus of $(X,D)$ and over general points of every $S_i$.  
    \item $(Y,\Delta+f^{-1}_*D)$ is a $\mathbb{Q}$-factorial  dlt pair, where $\Delta$ is the reduced sum of all $f$-exceptional prime divisors.   
    \item If  $D = \Gamma_{\red}$ for some Cartier divisor $\Gamma$, and  if $X\setminus  D$  has quotient singularities, 
    then the exceptional locus of $f$ is contained in the support of $\Delta+f^{-1}_*D$.
\end{enumerate}
\end{lemma} 

\begin{proof} 
Let $\rho\colon Z \to X$ be a  projective log resolution of $(X,D)$,  which is an isomorphism over the snc locus of $(X,D)$. 
Let $E$ be the whole reduced  $\rho$-exceptional divisor and let $\Gamma = E + \rho^{-1}_*D$.  
Then, thanks to \cite[Theorem 2.45]{DasHaconPaun2022}, 
we can run a $\rho$-relative MMP for the   dlt pair $(Z,  \Gamma )$,  
with the scaling of some $\rho$-ample divisor $H$ as follows,  
\[(Z,\Gamma ) = (Z_0,\Gamma_0) \dashrightarrow (Z_1,\Gamma_1) \dashrightarrow \cdots.\]

If the MMP terminates, then we let $Y$ be the output of the MMP and let $f\colon Y \to X$ be the natural morphism. 
From the negativity lemma, 
we deduce that $f$ is an isomorphism over general points of each $S_i$. 

Otherwise, by Lemma \ref{lemma:MMP-scaling-birkar}, the nef threshold sequence $(\lambda_j)$ of the MMP  tends to 0. 
Let $V_i$ be a local surface intersecting $S_i$ transversally at a very general point $x_i\in S_i$. 
Then the MMP induces  a surface MMP over $V_i$.  
Since $X$ has quotient singularities around general points of $S_i$, 
we see that $V_i$ is a klt surface singularity by Lemma \ref{lemma:singular-type-well-defined}. 
Since $D$ does not contain $S_i$, we can pick  $x_i\not\in D$. 
Then we deduce that there is some integer $N(i)\ge 0$ such that  
$Z_j$ is isomorphic to $X$ at $x_i$ if $j\ge N(i)$.  
Let $k$  be the maximum of $\{N(1),...,N(r)\}$. 
Then    $Z_k\to X$ is an isomorphism over general points of every $S_i$. 
In this case, we set $Y=Z_k$ and denote by $f\colon Y\to X$  the natural morphism.

In both cases,   $(Y,\Delta+f^{-1}_*D)$ is a $\mathbb{Q}$-factorial  dlt pair, 
where $\Delta$ is the  reduced sum of all $f$-exceptional prime divisors. 
Furthermore,  $f$ is an isomorphism over the snc locus of $(X,D)$ and over general points of every $S_i$.    

It remains to prove the item (3).  
Assume that $V$ is an irreducible component of the $f$-exceptional locus which has codimension at least 2. 
Since $X\setminus D$ has quotient singularities,  
the exceptional locus of $f$ over $X\setminus D$ is pure of codimension 1, and hence $V$ is contained in $f^{-1}(\mathrm{Supp}\, D)$. 
We note that $f^{-1}(\mathrm{Supp}\, D) = \mathrm{Supp}\, f^*\Gamma$, which is pure of codimension 1. 
It follows that $f^{-1}(\mathrm{Supp}\, D)$ is contained in the support of $\Delta+f^{-1}_*D$.  
This completes the proof of the lemma. 
\end{proof}

\section{Bimeromorphic transforms}

In this section, we collect some tools of bimeromorphic transforms that we need for the proof of Theorem \ref{thm:main-thm}

\subsection{Functorial resolution of singularities}

Throughout this paper, a construction is called  functorial if it commutes with local analytic isomorphisms.
The functorial resolution of singularities implies equivariant resolution of singularities. 
It also allow us to glue local resolutions along the overlaps. 
We recall their statements  from \cite{Wlodarczyk2009} as follows.

 \begin{thm}
\label{thm:func-resol}
Let $X$ be a complex analytic variety. 
Then $X$ admits a functorial (or canonical) desingularization $r\colon  \widetilde{X} \to X$. 
More precisely,  $\widetilde{X}$ is a smooth complex analytic variety, $r$ is a projective bimeromorphic morphism,
which is a composition of blowups at smooth centers contained in $X_{\sing}$, the $r$-exceptional locus is a snc divisor, 
and  the construction of $r$ commutes with local analytic isomorphisms. 
\end{thm}

\begin{thm}
\label{thm:func-principalization}
Let $X$ be a complex analytic variety, and let $\mathcal{I}$ be a coherent ideal sheaf on $X$. 
Then $\mathcal{I}$ admits a functorial (or canonical) principalization $r\colon  \widetilde{X}  \to X$. 
More precisely,  $\widetilde{X}$ is a smooth complex analytic variety, $r$ is a projective bimeromorphic morphism,
which is a composition of blowups at smooth centers contained in the cosupport of $\mathcal{I}$, 
$r^{-1}\mathcal{I} \cdot \mathcal{O}_{\widetilde{X}}$ is an invertible sheaf which defines a  divisor of snc support, and the construction of $r$ commutes with local analytic isomorphisms. 
\end{thm}

We will need the following two results on partial desingularizations, which allow us to work locally around components of the singular locus.  
For example, we assume that $X$ is a complex analytic variety, and that  $S_1$  and $S_2$ are disjoint compact irreducible components of $X_{\sing}$. 
Let $X_1$ and $X_2$ be disjoint open subsets of $X$, which contain $S_1$ and $S_2$ respectively.
Suppose that, for $i=1,2$,  we have a proper bimeromorphic modification  $f_i\colon Y_i\to X_i$, which is an isomorphism over the smooth locus of $X_i$. 
We remark that there can be  an obstruction to glue the modifications $f_i$ into a modification on $X$. 
Indeed, there may be a component $S_3$ of $X_{\sing}$ which intersects both $S_1$ and $S_2$. 
If the modification of $f_1$ at general points of $S_3\cap X_1$  is different from the one of $f_2$ at general points of $S_3\cap X_2$, then we cannot glue $f_1$ and $f_2$. 
With the following two results, we can first obtain a  modification $\rho \colon X'\to X$ so that the singular loci of $\rho^{-1}(X_1)$ and $\rho^{-1}(X_2)$ are compact subsets. 
Then the modifications on  $\rho^{-1}(X_1)$ and $\rho^{-1}(X_2)$, which are isomorphic over the smooth loci,  
are independent and we can always glue them by using Lemma \ref{lemma:extend-morphism-1}.

\begin{prop}
\label{prop:partial-resolution} 
Let $X$ be a normal complex analytic variety and let $S$ be the union of some irreducible components of $X_{\sing}$.  
Let $\Delta$ be a reduced divisor in $X$, which does not contain any component of $S$. 
Then there is a projective bimeromorphic morphism 
$f\colon Y \to X$ satisfying the following properties. 
Let  $U = X\setminus S$ and $V= f^{-1}(U)$. 

\begin{enumerate}
 \item The $f$-exceptional locus is pure of codimension 1. 
       If we denote it by $E$ as a reduced divisor, then  $(V,(E + f^{-1}_*\Delta)|_V)$ is a snc pair.  
       In particular, $Y_{\sing}$ is contained in $f^{-1}(S)$. 
 \item $f$ is an isomorphism over general points of every irreducible component of $S$. 
 \item $f$ is an isomorphism over the  snc locus of $(X,\Delta)$.  
\end{enumerate}
\end{prop}

\begin{proof}
Let $r\colon Z\to X$ be a  log resolution of $(X,\Delta)$ such that it can be decomposed into a sequence
\[
Z= Z_m \to \cdots \to Z_0 = X,
\]
where each $Z_{j+1}\to Z_j$ is the blowup at a smooth irreducible center $C_j$ contained in the singular  locus (\textit{i.e.} the non-snc locus)  of $(X_j,\Delta_j)$, and $\Delta_j$ is the sum of the strict transform of $\Delta$ in $X_j$ and the whole reduced exceptional divisor of $X_j\to X$.   
We will construct by induction a sequence of projective bimeromorphic morphisms 
\[
Y=Y_m \to \cdots \to Y_0 = X 
\]
such that the following properties hold. 
The composition $f\colon Y\to X$ satisfies the properties of the proposition. 
In particular, we can define $S_j$ the strict transform of $S$ in $Y_j$, 
and $Y_{j+1}\to Y_j$ is an isomorphism over general points of every irreducible component of $S_j$. 
Let $V_j$ be the preimage of $U$ in $Y_j$ and let $W_j$ be the preimage of $U$ in $Z_j$. 
Then $V_j \cong  W_j$ under the natural bimeromorphic map $\psi_j\colon Z_j \dashrightarrow Y_j$ over $X$.

Assume that we have constructed $Y_j$. 
If $C_j\cap W_j = \emptyset$, then we let $Y_{j+1}\to Y_j$ be the identity map.   
We note that  $W_{j+1}\cong W_{j}$. 
Hence $V_{j+1}\cong W_{j+1}$ under $\psi_{j+1}$. 

If  $C_j\cap W_j \neq \emptyset$, then we can define the strict transform $D_j$  of $C_j$ \textit{via} $\psi_j$. 
Let $Y_{j+1}\to Y_j$ the blowup of $D_j$.   
It follows that $V_{j+1}\cong W_{j+1}$ under $\psi_{j+1}$. 
It remains to prove that $Y_{j+1}\to Y_j$ is an isomorphism over general points of every irreducible component of $S_j$. 
It is equivalent to show that $D_j$ does not contain any of these irreducible components.  
Assume by contradiction that $D_j$ contains a component $F_j$ of $S_j$. 
Then the image $D$ of $D_j$ in $X$ contains the one $F$ of $F_j$. 
We note that $D\cap U \neq \emptyset$. 
On the one hand, by assumption, $F$ is an irreducible component of $S$.  
On the other hand, $D$ is also the image of $C_j$ in $X$. 
By assumption on $Z$, $D$ is contained in singular locus of $(X, \Delta)$.  
It follows that there is an irreducible component $P$ of the singular locus of  $(X, \Delta)$ such that 
\[
F\subseteq D \subseteq P. 
\]
Since $\Delta$ does not contain $F$, we deduce that $P$ is an irreducible component of $X_{\sing}$.  
Then we must have $F=P$. 
This is a contradiction as  we have 
\[
 \emptyset \neq  D\cap U \subseteq  P\cap U = F\cap U = \emptyset.   
\]
In conclusion, $Y_{j+1}\to Y_j$ is an isomorphism over general points of every irreducible component of $S_j$. 
This completes the proof of the proposition. 
\end{proof}

\begin{cor}
\label{cor:partial-resolution} 
With the notation in Proposition \ref{prop:partial-resolution}, we assume that $X$ has quotient singularities around $S\setminus \Delta$.  
Let $g\colon X'\to Y$ be the normalization of the blowup of $Y$ at the ideal sheaf $\mathcal{I}$ of $E+f^{-1}_*\Delta$, 
and let $\rho\colon X'\to X$ be the natural morphism. 
Then the following properties hold.
\begin{enumerate} 
  \item $g^*\mathcal{I}$ is an invertible sheaf and it  defines a Cartier divisor $\Gamma$.
  \item If $\Delta'= \Gamma_{\red}$, then   $(X',\Delta')$ is a   snc pair on $\rho^{-1}(X\setminus S)$.  
       In particular, $X'_{\sing}$ is contained in $\rho^{-1}(S)$. 
  \item The $\rho$-exceptional locus is contained in $\Delta'$, and $\Delta'$ is equal to the sum of $\rho^{-1}_*\Delta$ and the whole reduced $\rho$-exceptional divisor.  
  \item $\rho$ is an isomorphism over general points of every irreducible component of $S$. 
  \item $\rho$ is an isomorphism over the  snc locus of $(X,\Delta)$.  
  \item Every component $C$ of $X'_{\sing}$ which is not contained in $\Delta'$ is the strict transform of a component of $S$.  
        Moreover, $X'\setminus \Delta'$ has quotient singularities.  
\end{enumerate}
\end{cor}

\begin{proof}
The item (1) follows from the property of blowups. 
Since $(Y, E+f^{-1}_*\Delta)$ is a snc pair on $f^{-1}(X\setminus S)$, we deduce that $g$ is an isomorphism over $f^{-1}(X\setminus S)$. 
This implies the item (2).  
Since $Y\setminus E$ is normal, the $g$-exceptional locus is contained in $\Delta'$. This shows the items (3) and (4).   
The item (5) follows from the fact that $f$ is an isomorphism over the  snc locus of $(X,\Delta)$.   
For the item (6), we first  note that $\rho(C\setminus \Delta') \subseteq S\setminus \Delta$ by the items (2) and (3). 
Moreover, $\rho$ is an isomorphism from a neighborhood of  $C\setminus \Delta'$ to its image in $X$.  
This implies that $C$ is the strict transform of a component of $S$. 
Since $X$ has quotient quotient singularities around $S\setminus \Delta$, 
we deduce  that $X'$ has quotient singularities around $C\setminus \Delta'$. 
This shows the item (6), and completes the proof of the corollary. 
\end{proof}

\subsection{Blowups at an ideal sheaf}
\label{subsection:blowup}

Let $X$ be a complex analytic space and let $S\subseteq X$ be a closed subspace. 
If $g\colon S' \to S$ is a blowup at a coherent ideal sheaf $\mathcal{I}$, 
then there is a natural blowup $f\colon X'\to X$, 
such that the restriction of $f$ on the strict transform of $S$ in $X'$ is isomorphic to   $g$.    
Indeed, the closed analytic subspace $Z\subseteq S$ defined by $\mathcal{I}$ is naturally a closed analytic subspace of $X$. 
If we denote the ideal sheaf of $Z$ in $X$ by $\mathcal{I}_1$, 
then we set  $f\colon X'\to X$ as the blowup at the ideal sheaf $\mathcal{I}_1$.  
As a consequence, if $g\colon S' \to S$ is now a composition of a sequence of functorial blowups, 
then there is a natural functorial morphism $f\colon X' \to X$ 
such that the restriction of $f$ on the strict transform of $S$ in $X'$ is isomorphic to   $g$.   

For an effective  Cartier divisor $H\subseteq S$, the blowup of $S$ at (the ideal sheaf of) $H$ is just an isomorphism. 
However, we can blowup $X$ at $H$ and get a non trivial morphism $f\colon X'\to X$. 
In this case, $f$ induces an isomorphism from the strict transform $S'$ of $S$ to $S$.  
We will focus on this operation in the next subsection.

In the remainder of this subsection, we let $o\in S\subseteq X$ be a point. 
We will shrink $S$ and $X$ around $o$ freely.  
Let $W$ be an open neighborhood of $\{\mathbf{0}_N \}\times S$  in $\mathbb{C}^N \times S$.  
We assume that  there is a closed immersion from $X$ to   $W$
such that the composite inclusion  $S\subseteq X \subseteq W$ coincides  with $\{\mathbf{0}_N \}\times S \subseteq W$. 
We denote by $\mathcal{J}_X$ the ideal sheaf of $X$ inside $W$. 
Let $F(x_1,...,x_N,s)$ be a holomorphic function on $W$, where $(x_1,...,x_N)$ are
the coordinates of $\mathbb{C}^{N}$ and $s\in S$ is a point. 
We assume that $F\in \mathcal{J}_X$. 
In particular, we have $F(0,...,0,s) =0$ for all $s \in S$. 
Then, up to shrinking $W$,  we have a natural  decomposition  
\[F = F_1+ F_2 + \cdots,\]
where $F_d(x_1,...,x_N,s)$ is a homogeneous polynomial of degree $d$ in $(x_1,...,x_N)$, with coefficients as holomorphic functions on $S$. 
We will blowup $W$ at  closed subspaces contained in $S=\mathbf{0}_N\times S$ and then investigate the defining functions of the strict transform $X'$ of $X$, in a neighborhood of the strict transform $S'$ of $S$.

More precisely, let $\mathcal{I}$ be a coherent ideal sheaf on $S$ whose cosupport is proper closed subset.  
Up to shrinking $X$, we assume that $\mathcal{I}$ is globally generated. 
Let $h\colon W'\to W$ be the blowup at the ideal   $(x_1,...,x_N, \mathcal{I})$.  
Since  the cosupport of $\mathcal{I}$ is strictly contained in $S$, 
we can denote by $X'$ and $S'$ the strict transforms of $X$ and $S$ respectively in $W'$.  
Let $\mathcal{J}_{X'}$ be the ideal sheaf of $X'$ inside $W'$. 

Then for every point $o'\in S'$ lying over $o$,   
there is a   neighborhood $U'\subseteq W'$ of $o'$, such that  
$U'$ is isomorphic to an neighborhood of $\mathbf{0}_N\times (S'\cap U')$ in  $\mathbb{C}^N \times  (S'\cap U')$,  
and that 
$h|_{U'}\colon U'\to W$ can be written in coordinates as 
\[
(x'_1,...,x'_N,s') \to (tx'_1,...,tx'_N, h(s') ) = (x_1,...,x_N,s), 
\]
where  $t$ is a holomorphic function on $S'\cap U$, which is a generator of the invertible ideal sheaf $(h|_{S'})^*\mathcal{I}$  on $S'\cap U'$.  
Furthermore, regarded as a function on $U'$, \textit{i.e.} $t(x'_1,...,x'_N,s') = t(s')$,  $t$ is a generator of the invertible ideal sheaf $h^*(x_1,...,x_N,\mathcal{I})$ on $U'$. 
Indeed,  we can let $U'  \subseteq  W'\setminus h^{-1}_*H$ and $t=(h|_{S'})^*\xi$, 
where $\xi$ belongs to some  generating set of $\mathcal{I}$, and  $H\subseteq W$ is the zero locus  of the function $\xi$ on $W$ defined by $\xi(x_1,...,x_N,s)=\xi(s)$.

If we pullback the function $F_i$ by $h|_{U'}$, then it is of the shape $ t^i \cdot  G_i(x'_1,...,x'_N,s')$, 
where $G_i$ is a homogeneous polynomial of degree $i$ in $(x'_1,...,x'_N)$, with coefficients as holomorphic functions in $s'$. 
Thus if $m$ is the smallest integer such that $F_m\neq 0$, then the pullback of $F$ on $U'$ is of the shape
\[
(h^*F)(x'_1,...,x'_N,s')  = t^m( G_m + t G_{m+1} + \cdots + t^{k-m}G_k+\cdots).
\]
When the coefficients of $F_m$ are not all contained in $\mathcal{I}$, then $t$ does not divides $G_m$, 
and the function 
\[
  G_m + t G_{m+1} + \cdots + t^{k-m}G_k+\cdots
\]
is in  the ideal $\mathcal{J}_{X'}$  of $X'$ in $W'$.  
If moreover, $F_m(x_1,...,x_N,o)$ is a nonzero polynomial in $(x_1,...,x_N)$, 
then so is $G_m(x'_1,...,x'_N,o')$.

We assume that every coefficient of $F_m$ belongs to $\mathcal{I}$.  More precisely, we can write  
\[
F_m(x_1,...,x_N, s) = \sum_I a_I(s)x_I, 
\]
where the sum runs over all multi-indices $I$ with $|I|=m$,
and  $a_I \in \mathcal{I}$ for all $I$. 
If we set $a_I'(s') = a_I(h(s'))$, then $t$ divides $a_I'$. 
In particular, we can write $a_I'(s')=t\cdot b_I'(s')$. 
Then  the pullback of $F_m$ can be written as 
\[
(h^*F_m)(x'_1,...,x'_N,s') = \sum a_I'(s') \cdot t^m\cdot x'_I = t^{m+1}( \sum  b_I'(s') x'_I).
\]
This implies that $G_m$ is divisible by $t$. 

In addition, if $\mathcal{I}$ is exactly the ideal generated by the coefficients $a_I$ of $F_m$, 
then, up to shrinking $U'$, 
there is some $b_I'$ which is nowhere vanishing on $U'\cap S'$. 
As a consequence, we can write $(f^*F_m)(x'_1,...,x'_N,s') = t^{m+1} H_m(x'_1,...,x'_N,s')$ such that 
$H_m$ is a nonzero homogeneous polynomial  in $(x'_1,...,x'_N)$ for any $s'\in U'\cap S'$.
Therefore, the function 
\[
  H_m + G_{m+1} +  t G_{m+2} +  \cdots + t^{k-m-1} G_{k} + \cdots 
\]
is in $\mathcal{J}_{X'}$.

\subsection{Blowups at a Cartier divisor}
\label{subsection:blowup-divisor}

Let $X$ be a   complex analytic space and let $S\subseteq X$ be a closed subspace. 
In the previous subsection, we have seen the blowup of $X$ at a center strictly contained in $S$. 
In this subsection, we will specify   the case when the center is a Cartier divisor in $S$.  
We will perform this kind of bimeromorphic transforms in Section \ref{section:equation} and Section \ref{section:local-cover}.

In the remainder of this subsection,  
we let $S$ be a smooth complex  analytic variety  and let $H=H_1+ \cdots H_n$ be a reduced snc divisor on $S$. 
We assume that $S$ is isomorphic to a bounded neighborhood of the origin in $\mathbb{C}^N$ with $N\ge n$,  
and that  $H_i$  is isomorphic to the coordinate hyperplane $ \{T_i=0\}$. 
We consider  $W =   \mathbb{D}^3 \times S$.  
Let $F(x,y,z,s)$ be a holomorphic function on $W$, where $(x,y,z)$ are the coordinates of $\mathbb{D}^{3}$ and $s\in S$ is a point. 
We assume that    
\[F(x,y,z,s) = x^2 +  F_2(y,z,s) + F_3(y,z,s)  + \cdots + F_k(y,z,s) + \cdots,\]
where $F_k(y,z,s)$ is a homogeneous polynomial of degree $k$ in $(y,z)$, with coefficients as holomorphic functions on $S$. 
Let $X\subseteq W$ be the   subspace defined by $F=0$. 

Let $f\colon W'\to W$ be the blowup at the ideal $\mathcal{I} = (x,y,z,T_i)$, where $i\in \{1,...,n\}$. 
Let $X'$ and $S'$ be the strict transforms of $X$ and $S$ respectively in $W'$.  
Our interest  is  a defining function of $X'$ in a neighborhood of $S'$. 
We note that $f$ induces an isomorphism from $S'$ to $S$. 
There is an  open subset $U' \subseteq W'\setminus f^{-1}_*(D_i)$, where $D_i$ is the divisor in $W$ defined by $T_i=0$, such that   $U'\cong  \mathbb{D}^3 \times S$, and that  
$f|_{U'}$ can be written with coordinates as 
\[
(x',y',z',s) \to (T_i x,T_i y,T_i z,s). 
\]
Moreover, the $f$-exceptional divisor $E$ is defined by $\{T_i=0\}$ on $U'$.  
We note that the pullback $f^*F$ is divisible by   $T_i^2$. 
Thus, a  defining function of $X'\cap U'$ in $U'$ is of the shape
\[
F'(x',y',z', s) = x'^2 +  F_2(y',z',s) + T_i F_3(y',z',s)  + \cdots + T_i^{k-2} F_k(y',z',s) + \cdots.
\]
We will need to perform this operation for several times in order to raise the exponent of $T_i$ in the degree $k$ part of the defining function $F$ for all $k$ large.

More precisely,  let $f\colon W' \to W$ be the  
composition of the following sequence  
\[
W'=W^e \to W^{e-1}  \to \cdots \to W^0 = W, 
\]
such that each $\varphi^k \colon  W^{k+1}\to W^{k}$ is the blowup of $W^k$  at the center $H_i^k$ for some $i \in \{1,...,n\}$, 
where $H_i^k$ is the preimage of $H_i$ in $S^k$, and $S^k$ is the strict transform of $S$ in $W^k$.  
We say that there are $m$ blowups whose centers are  $H_i$  if there are $\varphi^{k_1},...,\varphi^{k_m}$ in the previous sequence, 
such that $\varphi^{k_l}$ is the blowup of $W^{k_l}$ at the center $H_i^{k_l}$ for all $l=1,...,m$.  
Let $X'$ and $S'$ be the strict transforms of $X$ and $S$ respectively in $W'$. 
Then $f$ induces an isomorphism from $S'$ to $S$.  
Let $T' = \prod_{i=1}^n  T_i^{m_i}$, where $m_i$ is the number of blowups  whose centers are $H_i$. 
Then there is an open neighborhood $U'$ of  $S'$ in $W'$, 
such that  $U'\cong \mathbb{D}^3 \times S'$ 
and that $f|_{U'}$ can be written with coordinates as 
\begin{equation}\label{eqn:blowup}
(x',y',z',s) \mapsto (T' x, T' y, T' z,s). 
\end{equation}
Moreover, the $f$-exceptional divisor $E$ is defined by $\{T'=0\}$ on $U'$  as a closed subset.  
A  defining function for $X'\cap U'$ in $U'$ is 
\[
F'(x',y',z', s) = x'^2 +  F_2(y',z',s) + T' F_3(y',z',s) + \cdots  + (T')^{k-2} F_k(y',z',s) + \cdots.
\]
As a consequence, the homogeneous part  of $F'$  in $(x',y',z')$ of degree $k$ is divisible by $(T')^{k-2}$.

\begin{remark}
\label{rmk:function-exc-div} 
Assume that $m_i>0$ for some $i\in \{1,...,n\}$. 
If we  regard $T_i$ as a function on $U'$ by setting   $T_i(x',y',z',s)=T_i(s)$, 
then we observe   from \eqref{eqn:blowup}   that the divisor in $U'$ defined by $T_i=0$ is the $f$-exceptional divisor defined by the ideal sheaf $f^{-1}\mathcal{I}(H_i) \cdot \mathcal{O}_{W'}$, 
where $\mathcal{I}(H_i) = (x,y,z, T_i)$ is the ideal sheaf of $H_i$ in  $W$.   
\end{remark}

By abuse of notation, we  will write $X$ and $S$ for $X'\cap U'$ and $S'$ respectively, and  say that, after the blowups, the defining equation of $X$ becomes 
\begin{eqnarray*}
  F(x,y,z, s) &=& x^2 +  F_2(y,z,s) + F_3(y,z,s) + \cdots + F_k(y,z,s) + \cdots \\
        & = &  x^2 +  F_2(y,z,s) + T' G_3(y,z,s) + \cdots + (T')^{k-2}  G_k(y,z,s) + \cdots
\end{eqnarray*}

\begin{remark}
    \label{rmk:blowup-order} 
We will call the previous operation a blowup of $X$ at the $H_i$'s.   
As mentioned earlier, the objective of this  operation is to raise the exponent of $T_i$ in $F_k$. 
We observe that the number  we raise  for the exponent of $T_i$  depends only on the number of blowups whose centers are $H_i$. 
In particular,  the order that we  blow up at the $H_i$'s does not matter for our purpose. 
\end{remark}

\begin{remark}
    \label{rmk:blowup-divisor-1} 
Assume that $F_j(y,z,s) = a(s)\cdot E_j(y,z,s)$, where the zero locus of $a(s)$ is contained in $H$. 
Let $b(s)$ be a holomorphic function whose zero locus is contained in $H$. 
By blowing up $X$ at the $H_i$'s for several times, the defining equation of $X$ becomes the following shape
\begin{eqnarray*}
   &&F(x,y,z, s) \\
&=& x^2 + F_2(y,z,s) + \cdots + (T')^{j-3} G_{j-1}  \\
           && +  (T')^{j-2} \Big( a(s)E'_j(y,z,s)  
             + T' \cdot \big(G_{j+1}( y,z,s) + T'G_{j+2}(y,z,s) + \cdots \big) \Big).  
\end{eqnarray*}
By choosing $m_i$ greater  than the vanishing order of $a(s)b(s)$ along $H_i$ for all $i=1,...,n$,  
the function $a(s)b(s)$ divides $T'$. 
As a result, the defining equation of $X$ can be written as 
\begin{eqnarray*}
F(x,y,z,s) &=& x^2 + F_2(y,z,s) + \cdots + (T')^{j-3} G_{j-1}  \\ 
           &&  +   (T')^{j-2} a(s)  \Big( E'_j(y,z,s) + \\ 
           && b(s)E'_{j+1}(y,z,s) +  b(s)E'_{j+2}(y,z,s) +  \cdots   \Big), 
\end{eqnarray*}
where each $E'_{l}$ is a holomorphic function, homogeneous in $(x,y,z)$ of degree $l$. 
In  other words, after the blowups, we  can write $F_j(y,z,s) = a'(s) \cdot E'_j(y,z,s)$, 
where $ a'(s) =  (T')^{j-2} a(s)$. 
Furthermore, the zero locus of $a'(s)$ is contained in $H$, 
and $a'(s)b(s)$ divides 
\[F_{j+1} + F_{j+2} + \cdots =  a'(s)b(s)E'_{j+1} +  a'(s)b(s)E'_{j+2} + \cdots .  \]
In conclusion, by blowing up $X$ and by denoting $a'(s)$ by  $a(s)$, 
we can make $a(s)b(s)$ divides $F_{j+1} + \cdots$.   
Equivalently, the defining function $F$ becomes 
\begin{eqnarray*}
F(x,y,z,s) &=& x^2 + F_2(y,z,s) + \cdots +  F_{j-1}(y,z,s) \\
           &&+ F_j(y,z,s)+ F_{j+1}(y,z,s) + \cdots \\ 
           &=&  x^2 + F_2(y,z,s) + \cdots +  F_{j-1}(y,z,s)  \\  
           && + a(s) E_j(y,z,s) + a(s)b(s)R(y,z,s),   
\end{eqnarray*}
where $R(y,z,s)$ is a holomorphic function with $R(y,z,s) = 0 \mod\, (y,z)^{j+1}$.
\end{remark}

\begin{remark}
\label{rmk:blowup-divisor-2}
Assume that $F(x,y,z,s) = x^2 + G(y,z,s)$ with 
\[G(y,z,s)  = \sum_{i=0}^3 R_i(y,z,s) \cdot y^i\cdot (q(s)z)^{3-i} +   P(y,s), \]
where the zero locus of $q(s)$ is contained in $H$, 
and $P(y,s) = 0 \mod \,  (y)^4$. 
We blow up $X$  at  each $H_i$ for $m_i$ times,  
and  we choose $m_i$  greater than the vanishing order of $q(s)$ along $H_i$. 
Particularly, there is a holomorphic function $b(s)$ such that $T' = q(s)b(s)$.  
Up to shrinking $S$ around the origin, we may assume that $b(s)$ and $q(s)$ are bounded on $S$. 
The defining equation of $X$ becomes 
\begin{eqnarray*}
F(x,y,z,s) &=& x^2 +  (T')^{-2} \cdot G(T'y,T'z,s) \\ 
           &=& x^2 + (T')^{-2} \cdot \Big(  \sum_{i=0}^3 R_i(T'y,T'z,s) \cdot (T')^3   y^i\cdot (q(s)z)^{3-i}    + P(T'y,s) \Big) \\
           &=& x^2 + T' \cdot   \sum_{i=0}^3 R_i(T'y,T'z,s) \cdot    y^i\cdot (q(s)z)^{3-i}  
           +   (T')^{-2} \cdot P(T'y,s) .  
\end{eqnarray*}
We note that $P'(y,s) := (T')^{-2} \cdot P(T'y,s)$ is a holomorphic function in $(y,s)$.  
We set $\zeta = q(s)z$ and  define
\[R_i'(y,\zeta,s)  := T'R_i(T'y,b(s) \cdot \zeta,s) = T'R_i(T'y,b(s) \cdot q(s)z,s)  = T'R_i(T'y,T'z,s),   \]
for $i=0,2,3$. 
Then each $R'_i(y,\zeta,s)$ is a holomorphic function on $\mathbb{D}^2\times S$, for $T'$ and $b(s)$ are bounded on $S$. 
Moreover, we have 
\[
F(x,y,z,s) = x^2 + \sum_{i=0}^3 R'_i(y,\zeta,s) \cdot    y^i\cdot \zeta^{3-i}    +  P'(y,s),
\]
which now can be viewed as a function in $(x,y,\zeta,s)$.  
In conclusion, by blowing up $X$ at the $H_i$'s,  
we can make $F(x,y,z,s)$ a  holomorphic function  
in $(x,y,\zeta,s)$. 
\end{remark}

\section{Modification to double-point singularities} 
\label{section:double-point}

Let $(S,H)$ be a reduced snc pair, and let $o\in S$ be a point. 
We consider a holomorphic function $F$ defined in a neighborhood of $(\mathbf{0}_3,o) \in \mathbb{C}^3\times S$ of the shape 
\[
F(x,y,z,s) = x^2 + F_2(y,z,s)+F_3(y,z,s)+R(y,z,s),
\]
where $(x,y,z)$ are coordinates of $\mathbb{C}^3$ and $s\in S$, the functions 
$F_2$ and $F_3$ are homogeneous polynomials in $(y,z)$, of degrees $2$ and $3$ respectively,  with coefficients as holomorphic functions on $S$, 
and   $R(y,z,s) = 0 \mod \, (y,z)^4$. 
We say that $F$ is of \textit{standard form} with respect to $(S,H)$ if the following property holds. 
Either $F_2(y,z,s) = a(s) \cdot G_2(y,z,s)$, where $G_2(y,z,s)$ is a nonzero polynomials in $(y,z)$ for any point $s\in S$ , and the zero locus of $a(s)$ is contained in $H$; 
or $F_2(y,z,s) = 0$ and $F_3(y,z,s) = b(s) \cdot G_3(y,z,s)$, where $G_3(y,z,s)$ is a nonzero polynomials in $(y,z)$ for any point $s\in S$, and the zero locus of $b(s)$ is contained in $H$.  
The objective of this section is to prove the following proposition.

\begin{prop}
\label{prop-first-local-reduction}
Let $X$ be a   complex analytic variety. 
Let $S$ be a  codimension 2 irreducible component   of $X_{\sing}$.  
Assume that there is a proper Zariski closed  subset $C\subseteq S$ such that 
$X$ has the same type of  canonical singularities at  points   of $S\setminus C$.  
Then there is a   projective bimeromorphic morphism 
$f\colon Y\to X$ 
such that the following properties hold.   
\begin{enumerate}
    \item $f$ is the composition of blowups with centers strictly contained in $S$. 
    \item If $S_Y$ is  the strict transform of $S$ in $Y$, then $(S_Y,H_Y)$ is a   snc pair, where $H_Y$  is the reduced divisor, 
    whose support is the union of the preimage of  $C$ in $S_Y$ and the intersection of $S_Y$ with the $f$-exceptional locus. 
    \item For every point $o\in S_Y$, there is a neighborhood $U\subseteq Y$ of $o$, such that 
    $U$ is isomorphic to a hypersurface in $\mathbb{D}^3\times (S_Y\cap U)$, defined by an equation of standard form with respect to $(S_Y\cap U,H_Y\cap U)$. 
    \item  $Y$ has the same type of canonical singularities at   points of $S_Y\setminus H_Y$. 
    \item The construction of $f$ is functorial. 
\end{enumerate}
\end{prop}

We remark that the items (1)-(4) still hold if we blow up $Y$ at a stratum of $H_Y$.   
The proof of the proposition is divided into three steps. 
In the first one, we reduce to the case when $X$ has hypersurface singularities. 
In the second step, we reduce to the case when $X$ has double-point singularities. 
We arrange the defining equation into standard form in the last step.

\subsection{Hypersurface singularities}
\label{subsection:hypersurface}

In this subsection, we consider the following situation. 
Let $X$ be a complex analytic variety and let $S \subseteq X $ be a closed  subspace. 
Assume that  $S$ is  irreducible and smooth of dimension $n$.  
Let $W =  \mathbb{D}^N \times S$.   
Assume that $X$ is isomorphic to the closed subspace in $W$ defined by 
holomorphic functions $f_j(x_1,...,x_N,s)$, where  $(x_1,...,x_N)$ are the coordinates of $\mathbb{C}^N$, $s\in S$ is a point,  and $j=1,...,k$ for some positive integer $k$. 
Assume that the composite inclusion $S\subseteq X \subseteq W$ coincides with $\{\mathbf{0}_{N}\} \times S \subseteq W$.  
In particular, $f_j(\mathbf{0}_N, s) = 0$ for all $s\in S$ and for all $j=1,...,k$.
We set  
\[
\mathrm{d}_x f_j := \sum_{i=1}^N \frac{\partial f_j}{\partial x_i} \mathrm{d}x_i  
\]
and we let  $r(s)$ be the rank of $ (\mathrm{d}_x f_j (\mathbf{0}_N, s))_{j=1,...,k}$, in the vector space generated by $\mathrm{d}x_1,..., \mathrm{d}x_N$.    
We remark that such an embedding $X\subseteq W$ always exists locally around every point of $S$, as $S$ is smooth.

\begin{lemma}
\label{lemma:embedded-dim}
With the notation above, the  embedding dimension  of $X$ at a point $s\in S \subseteq X$ is equal to   $ N + n - r(s)$. 
\end{lemma}

\begin{proof}
Since  $S$ is smooth, by  Jacobi criterion (see for example \cite[Section 6.1]{GR84}),   
the  embedding dimension  of $X$ is equal to $N+n-R(s)$, 
where $R(s)$ is the rank of $(\mathrm{d} f_j (\mathbf{0}_N, s) )_{j=1,...,k}$, in the vector space $\Omega^1_{W, (\mathbf{0}_N, s)}$. 
Since $f_j(\mathbf{0}_N, s) = 0$ for all $s\in S$ and for all $j=1,...,k$, 
the partial derivatives of $f_j$, with respect to the variables of $ S$, 
are all zero along   $\{\mathbf{0}_N\} \times S$, for all $j=1,...,k$.  
It follows that $R(s)$ is equal to $r(s)$. 
This completes the proof of the lemma. 
\end{proof}

\begin{remark}
With the notation above, 
we remark that $\mathrm{d}_x f_j( \mathbf{0}_N, s)$ represents the linear part of $f_j$ in $(x_1,...,x_N)$. 
Indeed, such a linear part is equal to 
\[
 \sum_{i=1}^N \frac{\partial f_j}{\partial x_i} ( \mathbf{0}_N, s)  \cdot x_i.
\]
\end{remark}

\begin{lemma}
\label{lemma:embeded-dimension-reduction} 
With the notation  above, let $r = \min \{r(s) \ | \  s\in S\}$.  
In  other words, $N+n-r$ is the maximum of the  embedding dimensions  of $X$ at the points of  $S$. 
Assume that $X$ has  embedding dimension  smaller than $N+n-r$ at general points of $S$. 
Then there is a projective bimeromorphic morphism  $f\colon Y\to X$ obtained by blowing up $X$ at ideal sheaves whose cosupports are strictly contained in $S$, 
such that the maximum of the  embedding dimensions of $Y$ at the points of  $S_Y$ is at most $N+n-r-1$, where $S_Y$ is the strict transform of $S$ in $Y$. 
Moreover, $S_Y$ is smooth and    the construction of 
$f$ is functorial,  which depends only on the isomorphism class of $X$ around $S$.  
In particular, it is independent of $N$. 
\end{lemma}

\begin{proof}
Let $o\in S$ be a point at which the  embedding dimension  of $X$ is $N+n-r$. 
We will first work locally around $o$. 
In particular, we are allowed to shrink $S$ around $o$. 
Then, by showing that the operation in the lemma is functorial, we can glue the local constructions  in the global setting. 

Up to a linear change of the coordinates $(x_1,...,x_N)$ of $\mathbb{C}^N$, 
and up to permuting the functions $f_j$, 
we may assume that 
\[
\mathrm{d}_x f_j (\mathbf{0}_N, o) = \mathrm{d}x_j 
\]
for $j=1,...,r.$ 
Then, by subtracting each $f_j$, with $j>r$, some linear combination, with coefficients as holomorphic functions on $S$, of the $f_j$'s with $j=1,...,r$, we may assume that 
\[
\frac{\partial f_j}{\partial x_i} (\mathbf{0}_N, s)=0 
\]
for $i=1,...,r$, $j>r$ and all $s\in S$ contained in a neighborhood of $o$.   
In particular, the matrix $\Big[ \frac{\partial f_j}{\partial x_i} (\mathbf{0}_N, s)\Big]_{1\le i \le N, 1\le j \le k }$, whose entries are holomorphic functions on $S$,  is of the shape 
\[
\begin{bmatrix}
A(s) & \mathbf{0}\\
* & B(s) &
\end{bmatrix},	
\]
where $A(s)$ is a $r\times r$ invertible matrix with $A(o)=\mathrm{Id}_r$.
Let $\mathcal{I} \subseteq \mathcal{O}(S)$ be the ideal generated by 
\[\frac{\partial f_j}{\partial x_i} (\mathbf{0}_N,s)  \mbox{ for all }  i, j>r,  \]
equivalently, the entries of the matrix $B$ above. 
Then, up to shrinking $S$, the closed subset contained in $S$, at which $X$ has  embedding dimension  $N+n-r$,  is the cosupport of $\mathcal{I}$.  
By assumption, this cosupport is a proper subset of $S$.

Let $h\colon Z\to X$  be the projective bimeromorphic morphism induced  by blowing up the ideal $\mathcal{I}$ on $S$, see Subsection \ref{subsection:blowup}.  
That is, $h$ is the blowup of $X$ at the ideal $(x_1,...,x_N, \mathcal{I})$. 
Let $S_Z$ be the strict transform of $S$ in $Z$ and $o_Z\in S_Z$ a point lying over $o$. 
Then there is a neighborhood $U$ of $o_Z$ in $Z$,  which  can be viewed as a closed subset of 
$ \mathbb{D}^N\times (U\cap S_Z). $ 
Moreover, $h$ can be written in coordinates as 
\[
h\colon (x_1',...,x_N',s') \mapsto (tx_1',...,tx_N',h(s')), 
\]
where $t$ is a holomorphic function on $U\cap S_Z$ which defines the exceptional divisor of $S_Z \to S$ on $U\cap S_Z$. 
We define 
\[
f'_j(x_1',...,x_N',s') = \frac{1}{t} (h^*f_j)(x_1,...,x_N,s) = \frac{1}{t} f_j( tx_1',...,tx_N', h(s') )
\]
for $j=1,...,r$, and 
\[
f'_j(x_1',...,x_N',s') = \frac{1}{t^2} (h^*f_j)(x_1,...,x_N,s) = \frac{1}{t^2} f_j( tx_1',...,tx_N', h(s') )
\]
for $j=r+1,...,k$. 
Then as shown in Subsection \ref{subsection:blowup}, 
all of the $f'_j$ are holomorphic functions and vanish on $Z$. 
We see that 
\[
\mathrm{d}_{x'} f'_j (\mathbf{0}_N, o_Z) = \mathrm{d}x'_j 
\]
for $j=1,...,r. $
Furthermore, since the ideal $\mathcal{I}$ is generated by $\frac{\partial f_j}{\partial x_i} (\mathbf{0}_N,s) $, where $i, j>r$, 
there are some $i,j>r$ such that 
\[
\frac{\partial f'_j}{\partial x'_i} (\mathbf{0}_N,o_Z) \neq 0.
 \]
As a consequence, after the blowup, the rank of $(\mathrm{d}_{x'} f'_j (\mathbf{0}_N, o_Z))_{j=1,...,k}$ is at least $r(o)+1$. 
In particular, if $S_Z$ is smooth at $o_Z$, then the  embedding dimension  of $Z$ at $o_Z$ is at most $N+n-r-1$.

Let $f\colon Y \to X$ be the projective bimeromorphic morphism induced  by the functorial principalization of $(g|_{S_Z})^{-1}\mathcal{I} \cdot \mathcal{O}_{S_Z}$ on $S_Z$.  
In particular, the strict transform $S_Y$ of $S$ in $Y$ is smooth.
Since $Y\to Z$ is the composition of blowups with centers strictly contained in  $S_Z$,  
the similar computation as above shows that 
$Y$ has  embedding dimension   at most $N+n-r-1$  at  a point  $o_Y\in S_Y$ lying over $o$. 
\\

It remains to show that the construction is functorial. 
This is equivalent to show that, around $o\in S$,  
the ideal $\mathcal{I}$ is independent of the choice of  coordinates. 
We first note that, in the construction above, the ideal $\mathcal{I}$ is also the ideal generated by all determinants of $(r+1)$-minors of the matrix 
\[
 \Big[\frac{\partial f_j}{\partial x_i} (\mathbf{0}_N, s) \Big]_{1\le i \le N, 1\le j \le k }. 
\]
Therefore, if $(g_e)_{e=1,...,t}$ is another collection of functions defining $X$, then they define the same ideal $\mathcal{I}$. 
Moreover, we assume that  $\psi \colon  (x'_1,...,x'_N,s) \to (x_1,...,x_N,s)$ is a change of  coordinates   of $\mathbb{C}^N\times S$.  
Then the $N\times N$ matrix 
\[\Big[\frac{\partial}{\partial x'_i} x_l (\mathbf{0}_N,s) \Big]_{ 1 \le i \le N, 1\le l \le N}\] is invertible around $o\in S$.  
Hence, by the chain rule, we deduce that the functions $f_j\circ \psi$, which define $X$ inside $\mathbb{C}^N\times S$ with the coordinates system  $(x'_1,...,x'_N,s)$, 
determine the same ideal $\mathcal{I}$. 

Next we will   reduce to the case when $r=0$, that is, the case when $N+n$ is the embedding dimension of $X$ at $o$.  
Indeed, from the previous discussions, thanks to the implicit function theorem, 
up to a change of coordinate of  $\mathbb{C}^N\times S$, 
we may assume that $(f_1,...,f_r)=(x_1,...,x_r)$. 
For any $j=r+1,...,k$, we set 
\[g_j(x_{r+1},...,x_N)= f_{j}(0,...,0, x_{r+1},...,x_N).\]
It follows that there is a closed embedding $X\subseteq \mathbb{C}^{N-r}\times S$, 
such that, with the coordinate system $(x_{r+1},...,x_N,s)$, the subspace $X$ is defined by $g_{r+1},...,g_k$.  
The ideal sheaf $\mathcal{I}$ is defined by the entries of the matrix  
\[
 \Big[\frac{\partial g_j}{\partial x_i} \Big]_{r+1\le i \le N, r+1\le j \le k }. 
\]

Therefore, we may always assume that $r=0$ and that $N+n$ is the embedding dimension of $X$ at $o$, 
for the construction of $\mathcal{I}$ around $o$.   
Assume that there is another  coordinates system $(x_1',...,x_N',s')$ of $\mathbb{C}^N\times S$, 
such that the subspace $X\subseteq \mathbb{C}^N\times S$  is defined by a family  $(g_e(x_1',...,x_N',s'))_{e=1,...,t}$ of functions. 
Then, locally around $o$,  there is a holomorphic map 
\[\varphi \colon  (x_1,...,x_N,s) \mapsto (x_1',...,x_N',s'), \]
where we regard $s'$ and  $x_1',...,x_N'$  as holomorphic functions in $(x_1,...,x_N, s)$.   
We note that $s'(\mathbf{0}_N,s)=s$ for any $s\in S$.  
We define the following  $N\times N$ matrix 
\[
\Theta(s) = \Big[\frac{\partial}{\partial x_i} x_l' (\mathbf{0}_N,s) \Big]_{ 1 \le i \le N, 1\le l \le N}. 
\]
As shown in Lemma \ref{lemma:embedded-dim}, the partial derivatives of $g_e$ with respect to the variables of $S$  are all 0 on $\{\mathbf{0}_N \} \times S$. 
Hence   for any $s\in S$,  we have 
\[
 \Big[\frac{\partial}{\partial x_i} (g_e\circ \varphi)  (\mathbf{0}_N,s) \Big]_{ 1 \le i \le N, 1 \le e \le t} = \Theta(s) \cdot \Big[(\frac{\partial g_e}{\partial x'_l}  (\mathbf{0}_N,s'(\mathbf{0}_N,s) ) \Big]_{ 1 \le l \le N, 1 \le e \le t}
\]
As we have assume that $r=0$ and that $N+n$ is the embedding dimension of $X$ at $o$, 
by considering the entries in the both sides above,
we deduce  that the ideal $\mathcal{I}$ constructed \textit{via} $(g_e\circ \varphi)_{e=1,...,t}$ on the coordinates system $(x_1,...,x_N,s) $ is contained in the one constructed \textit{via} $(g_e )_{e=1,...,t}$ on the coordinates system $(x_1',...,x_N',s')$. 
Since $(g_e\circ \varphi)_{e=1,...,t}$ also defines the inclusion $X\subseteq \mathbb{C}^N\times S$ with respect to the coordinates system $(x_1,...,x_N,s)$, 
it follows that the ideal constructed  \textit{via} $(f_j)_{j=1,...,k}$  on the coordinates system $(x_1,...,x_N,s) $ is  contained in  the one constructed \textit{via} $(g_e )_{e=1,...,t}$ on the coordinates system $(x_1',...,x_N',s')$.   
Since the roles of the two coordinates systems are symmetric, 
we conclude that the ideal $\mathcal{I}$ we have constructed is independent on the choice of coordinates. 
This completes the proof of the lemma. 
\end{proof}

\begin{example}
\label{example:embedded}
Assume that  $S=\mathbb{C}^2$ with coordinates $(a,b)$, and that $X\subseteq  \mathbb{C}^4\times S$ is defined by the ideal generated by the following functions,
\begin{eqnarray*}
f_1(x_1,x_2,x_3,x_4,a,b) &=& x_3  + x_3^2     \\
f_2(x_1,x_2,x_3,x_4,a,b) &=& x_3 + ax_1 +  bx_1x_2   \\
f_3(x_1,x_2,x_3,x_4,a,b) &=& x_3 + b x_1 +  x_2^2+x_4^3. \\
\end{eqnarray*}
Then when $a\neq 0$ or $b\neq 0$, and $x_1=x_2=x_3=x_4=0$,  
the rank of the Jacobi matrix of $(f_1,f_2,f_3)$ is  2. 
At $\mathbf{0}_6$, the rank of the Jacobi matrix is 1. 
Then $X$ has  embedding dimension  5 at this point. 
Now we blow up $\mathbb{C}^6$ at the ideal $(x_1,x_2,x_3,x_4,a,b )$ and obtain $h\colon W\to \mathbb{C}^6$. 
Then on  the open set $U= W \setminus  h^{-1}_*(\{ a =  0 \})$, $h$ can be written in coordinate as 
\[
h\colon (x_1',x_2',x_3',x'_4, \alpha , \beta ) \mapsto (\alpha x_1', \alpha x_2', \alpha x_3', \alpha x'_4,  \alpha, \alpha\beta ). 
\]
Then we have 
\begin{eqnarray*}
h^*f_1(x'_1,x'_2,x'_3,x'_4,\alpha , \beta) &=& \alpha (x'_3 +  \alpha (x'_3)^2)     \\
h^*f_2(x'_1,x'_2,x'_3,x'_4,\alpha , \beta) &=& \alpha (x'_3 + \alpha x'_1 +  \alpha^2\beta x'_1x'_2)   \\
h^*f_3(x'_1,x'_2,x'_3,x'_4,\alpha , \beta) &=& \alpha (x'_3  + \alpha \beta  x'_1 +  \alpha (x'_2)^2 +\alpha^2 (x'_4)^3). \\
\end{eqnarray*}
We set $g_j(x'_1,x'_2,x'_3,x'_4,\alpha,\beta)  = \frac{1}{\alpha}h^*f_j(x'_1,x'_2,x'_3,x'_4,\alpha , \beta) $ 
for $j=1,2,3$. 
We notice that $g_2-g_1$ and  $g_3-g_1$ are divisible by $\alpha$. 
We set 
\[p_2 = \frac{1}{\alpha}(g_2-g_1)  \mbox{ and }p_3 = \frac{1}{\alpha}(g_3-g_1).\] 
Then  the strict transform $X'$ of $X$ in $W$ is defined, on $U$, by the ideal generated by the following functions  
\begin{eqnarray*}
g_1(x'_1,x'_2,x'_3,x'_4,\alpha , \beta) &=&   x'_3 +   \alpha (x'_3)^2    \\
p_2(x'_1,x'_2,x'_3,x'_4,\alpha , \beta) &=&     x'_1  +  \alpha\beta x'_1x'_2   - (x'_3)^2   \\
p_3(x'_1,x'_2,x'_3,x'_4,\alpha , \beta) &=&    \beta  x'_1 +   (x'_2)^2  + \alpha (x'_4)^3 - (x'_3)^2.   \\
\end{eqnarray*}
Then the rank of the Jacobi matrix of $(g_1,p_1,p_2)$ is at least 2 on $U$. 
We can perform the same calculation on $U'= W \setminus  h^{-1}_*(\{ b =  0 \})$. 
In addition, the strict transform $S'$ of $S$ in $W$ is contained in $U\cup U'$. 
This shows that $X'$ has  embedding dimension  at most 4 around $S'$. 
\end{example}

\begin{cor}
\label{cor:blowup-hypersurface-sing}
With the notation in Lemma \ref{lemma:embeded-dimension-reduction}, we assume that   $X$ has hypersurface singularities (respectively is smooth) around general points of $S$. 
Then there is a projective bimeromorphic morphism  $f\colon X'\to X$ such that the following properties hold. 
\begin{enumerate}
    \item $f$ is obtained by a sequence  of blowups at   centers strictly contained  in $S$. 
    \item The strict transform $S'$ of $S$ in $X'$ is smooth. 
    \item $X'$ has hypersurface singularities (respectively is smooth)  around $S'$. 
    \item The construction of $f$ is functorial.
\end{enumerate}
\end{cor}

\begin{proof}
We will only treat the case when  $X$ has hypersurface singularities around general points of $S$, the other case can be proved by the same argument. 
Assume that the maximum of  embedding dimensions of $X$ at points of  $S$ is $\dim X +r$. 
If $r=1$, then $X$ has hypersurface singularities around $S$. 
Otherwise, there is a  sequence
of projective bimeromorphic morphisms 
\[ X_e  \to   \cdots \to X_1 =   X,\]
such that  
each $X_i \to X_{i+1}$ is the  functorial operation in Lemma \ref{lemma:embeded-dimension-reduction}, 
and  that the   maximum of  embedding dimensions of $X_e$ at points of $S_e$ is $\dim X +1$, 
where $S_e$ is the strict transform of $S$ in $X_e$. 
We let $X'=X_e$ and let $f\colon X'\to X$   be the composition of the sequence. 
Then it satisfies all the properties of the corollary.  
\end{proof}

\subsection{Double-point singularities}

We will prove the following statement in this subsection.

\begin{lemma}
\label{lemma:blowup-double-point}
Let $X$ be a complex analytic variety of dimension $n$ and let $S\subseteq X_{\sing}$ be an irreducible component of codimension $2$. 
Assume that $S$ is smooth, 
that $X$ has hypersurface singularities around  $S$, 
and that $X$ has double-point singularities around general points of $S$. 
Then there is a projective bimeromorphic morphism  $f\colon X'\to X$ such that the following properties hold. 
\begin{enumerate}
    \item $f$ is obtained by a sequence of blowups at centers strictly contained in  $S$. 
    \item The strict transform  $S'$ of $S$ in $X'$ is smooth. 
    \item $X'$ has double-point singularities around $S'$.  
    \item The construction of  $f$ is functorial. 
\end{enumerate}
\end{lemma}

\begin{proof}
It is enough to study the problem locally. 
Indeed, once we prove the functoriality,  we can glue our local constructions in the global setting.  
Let $o\in S \subseteq X$ be a point.  
We may assume that $S$ is a polydisc in $\mathbb{C}^n$ and $o\in S$ is the origin. 
Then locally around $o$, 
we may assume that $X$ is the hypersurface  in   $ \mathbb{D}^3\times S$, defined by the analytic function 
\[
F(x,y,z,s) = \sum_{i\ge 0} F_i(x,y,z,s),
\]
where $(x,y,z)$ are the coordinates of $\mathbb{D}^3$, $s \in S$,  and    $F_i$ is a homogeneous polynomial in $(x,y,z)$ of degree $i$ with coefficients as holomorphic functions on $S$.  
Furthermore, the composite inclusion $S\subseteq X\subseteq \mathbb{D}^3\times S$ coincides with $\{\mathbf{0}_3\} \times S \subseteq \mathbb{D}^3\times S$.  
In particular, we have $F_0=0$. 
Since $S\subseteq X_{\sing}$,   it follows that    $F_1=0$. 
Since $X$ has double-point singularities around general points of $S$,  
we have  $F_2 \neq 0$.

Let $\mathcal{I} \subseteq \mathcal{O}(S)$ be the   ideal  generated by the coefficients of $F_2$. 
If $D\subseteq S$ is the closed subset defined by $\mathcal{I}$, 
then  at every point in $S\setminus D$, 
$X$ has double-point   singularities.   

Let $h\colon X''\to X$ be the  blowup  up of  $X$ at the ideal $(x,y,z,\mathcal{I})$. 
As shown in Subsection \ref{subsection:blowup}, at every point $o''$ in the strict transform $S''$ of $S$ in $X''$, lying over $o$,  
there is an open neighborhood  $U''\subseteq X''$ of $o''$ such that 
$U''$ is isomorphic to the hypersurface in $\mathbb{D}^3\times (S''\cap U'')$ defined by the holomorphic function
\[
G(x'',y'',z'',s'') = G_2(x'',y'',z'',s'') + R(x'',y'',z'',s'')
\]
where $G_2$ is a homogeneous quadratic polynomial in $(x'',y'',z'')$ with coefficients as holomorphic functions on $S''$, and the terms of $R$ have degrees at least $3$ in $(x'',y'',z'')$. 
Furthermore, $G_2(x'',y'',z'',o'')$ is a nonzero polynomial in $(x'',y'',z'')$.
In particular, if $S''$ is smooth at $o''$, then $X''$ has double-point singularities at $o''$.

Let $X' \to X''$ be the projective bimeromorphic morphism induced by the functorial principalization of the ideal $(h|_{S''})^{-1}\mathcal{I}\cdot \mathcal{O}_{S''}$ on $S''$,   see Subsection \ref{subsection:blowup}.  
In particular, the strict transform $S'$ of $S$ in $X'$ is smooth. 
Since $X'\to X''$ is the composition of blowups with center strictly contained in  $S''$,   
the computations of Subsection \ref{subsection:blowup} implies  that 
$X'$ has double-point singularities along  $S'$.

To  prove that  construction is functorial, 
it is enough to show that the ideal $\mathcal{I} \subseteq \mathcal{O}(S)$  
is independent of the choice of  coordinates.
Let $ (\overline{x},\overline{y},\overline{z},\overline{s})$ be some other choice of coordinates of $\mathbb{D}^3\times S$ 
such that $X$ is defined by $\overline{F}(\overline{x},\overline{y},\overline{z},\overline{s}) = 0$. 
Let $\varphi\colon (x,y,z,s) \mapsto (\overline{x},\overline{y},\overline{z},\overline{s})$ be the change of coordinates, 
and we regard $\overline{x},\overline{y},\overline{z},\overline{s}$  as vector-valued holomorphic functions in $(x,y,z,s)$.  
Then we have 
\[
\overline{s}(0,0,0,s) = s  
\]
for all $s\in S$.   
Hence $\overline{s}(x,y,z,s) = s  \mod \, (x,y,z)$. 
There is a $3\times 3$ matrix $\Theta(s)$, 
whose entries are holomorphic functions on $S$, such that 
\[
(\overline{x},\overline{y},\overline{z}) =  (x,y,z) \cdot  \Theta(s)  \mod \, (x,y,z)^2. 
\]
We decompose $\overline{F}  = \sum_{i\ge 2} \overline{F}_i$ into the sum of its homogeneous parts in $(\overline{x},\overline{y},\overline{z})$. 
Then we have 
\begin{eqnarray*}
(\varphi^*\overline{F})(x,y,z,s) &=& \overline{F}(\overline{x},\overline{y},\overline{z},\overline{s})  \\
                                 &=& \overline{F}_2( (x,y,z) \cdot  \Theta(s), s)    \mod \, (x,y,z)^3. 
\end{eqnarray*} 
Since both the equations $\varphi^*\overline{F}=0$ and $F=0$ define  the hypersurface $X$, 
there is a unit function $U(x,y,z,s)$ such that 
\[F= U\cdot \varphi^*\overline{F}.\] 
We deduce that 
 the ideal $\mathcal{I}$ generated by the coefficients of  $F_2(x,y,z,s)$ 
is the same    as the one generated by the coefficients of $\overline{F}_2( (x,y,z) \cdot  \Theta(s), s)$.   
This implies that $\mathcal{I}$ is contained in the ideal generated by the coefficients of $\overline{F}_2(\overline{x},\overline{y},\overline{z},\overline{s})$.  
Since the roles of the coordinates systems are symmetric, 
we deduce that the ideal $\mathcal{I}$ we have constructed is independent of the choice of coordinates. 
This completes the proof of the lemma. 
\end{proof}

We  have the following observation for double-point singularities. 

\begin{lemma}
\label{lemma:local-irr} 
Let $X$ be a complex analytic variety  and let $S\subseteq X_{\sing}$ be an irreducible component of codimension at least $2$. 
Assume  that $X$ has double-point singularities around $S$. 
Then there is an open neighborhood $U$ of $S$ such that $U$ is locally irreducible. 
Equivalently, for every point $x\in U$, the stalk $\mathcal{O}_{X,x}$ is an integral domain. 
\end{lemma}

\begin{proof}
Let $f\colon Y\to X$ be the normalization of $X$, and let $x\in X$ be a point. 
If $\mathcal{O}_{X,x}$ is an integral domain, then $f$ is a homeomorphism over an open neighborhood of $x$ by \cite[Corollary on page 163]{GR84}. 
Conversely, if $\mathcal{O}_{X,x}$ is not an integral domain, then there is an open neighborhood of $x$ which is not irreducible by the local decomposition lemma (see \cite[page 79]{GR84}). 
In particular, $f$ is not a  homeomorphism over $x$.

Let $s\in S$ be a point. 
We claim that $\mathcal{O}_{X,s}$ is an integral domain. 
Assume the contrary. 
Then there is an open neighborhood $Z$ of $s$ in $X$ such that $Z$ is not irreducible.  
We may assume that $Z\subseteq \mathbb{D}^{n+1}$, and that $Z$ is a double cover of $\mathbb{D}^n$, where $n$ is the dimension of $X$. 
Then $Z$ has two irreducible components $Z_1$, $Z_2$,  and each of them is a smooth divisor in $\mathbb{D}^{n+1}$. 
It follows that $Z_{\sing} = Z_1 \cap Z_2$, which is pure of  dimension $n-1$.  
However, an irreducible component of  $S\cap Z$ is an irreducible component of $Z_{\sing}$, and is of dimension at most $n-2$. 
This is a contradiction. 

As a consequence, for every point $s\in S$, there is an open neighborhood $V_s$ of $s$ in $X$, such that the normalization morphism $f$ is an homeomorphism over $V_s$. 
It follows that $V_s$ is locally irreducible.  
In the end, we let $U= \bigcup_{s\in S} V_s$. 
Then $U$ is a locally irreducible neighborhood of $S$. 
This completes the proof of the lemma.
\end{proof}

\subsection{Equations of standard form}

We will complete the proof of Proposition \ref{prop-first-local-reduction} in this subsection. 

\begin{lemma}
\label{lemma:ideal-secondary}
Let $S$ be a polydisc  in  $\mathbb{C}^n$. 
Let $X\subseteq \mathbb{D}^3\times S$ be the hypersurface   defined by the equation 
\[
F(x,y,z,s) = x^2+ F_2(y,z,s)+F_3(y,z,s) + R(y,z,s) = 0, 
\]
where $(x,y,z)$ are the coordinates of $\mathbb{D}^3$ and $s\in S$. 
Moreover, for $i=2,3$,  $F_i(y,z,s)$ is a homogeneous polynomial of degree $i$ in $(y,z)$, with coefficients as holomorphic functions on $S$. 
The function $R(y,z,s)$ is holomorphic and its terms have degree at least 4 in $(y,z)$. 
We  assume that either $F_2$ or $F_3$ is not identically zero, 
and  define an   ideal  $\mathcal{I}$  of $\mathcal{O}(S)$
as follows. 
If $F_2 \neq 0$, then $\mathcal{I}$ is the ideal generated by the coefficients of $F_2$. 
Otherwise $\mathcal{I}$ is the ideal generated by  the coefficients of $F_3$. 
Then  the property that $F_2\neq 0$, and  the ideal $\mathcal{I}$, depend only on the isomorphism classes of $X$, 
and is independent of the choice of  coordinates $(x,y,z,s)$. 
\end{lemma}

\begin{proof}  
Assume that we have another choice of coordinates $(a,b,c,s')$ such that $X$ is defined as 
\[
G(a,b,c,s') = a^2 + G_2(b,c,s') + G_3(b,c,s') + P(b,c,s') = 0.  
\]
We denote the change of coordinates  by
\[
\varphi\colon  (x,y,z,s) \mapsto (a,b,c,s'),
\]
and we can regards $a,b,c,s'$ as vector-valued holomorphic functions in $(x,y,z,s)$.
Then $s'(0,0,0,s) = s$ for all $s\in S$. 
Thus $s'(x,y,z,s) = s  \mod \, (x,y,z)$.
If we set $(\varphi^*G)(x,y,z,s) = G(a,b,c,s')$, then 
\begin{equation}\label{eqn:change-cooridante}
    (\varphi^*G)(x,y,z,s) = (\mathrm{unit}) \cdot F(x,y,z,s). 
\end{equation}

First we assume that $G_2 =0$. 
Then, by comparing the coefficient before the term $x^2$ in \eqref{eqn:change-cooridante},   
we can  write  
\[a(x,y,z,s) = l(x,y,z,s)  \mod\, (y,z)^2,\] 
where $l =(\mathrm{unit})\cdot x +  l_1(y,z,s)$, and $l_1(y,z,s)$ is linear in $(y,z)$. 
Then we have 
\[(\varphi^*G)(x,y,z,s) = l(x,y,z,s)^2  \mod\, (x,y,z)^3.\]
Thus,  up to a nonzero multiplicative constant, 
the term $x  \cdot l_1(y,z,s)$ belongs to 
\[(\varphi^*G)(x,y,z,s)  \mod\, (x,y,z)^3. \]
Comparing with \eqref{eqn:change-cooridante}, we deduce that $l_1(y,z,s) =0$. 
Hence $l(x,y,z,s) = (\mathrm{unit})\cdot x$  and  $F_2 = 0$.  

We note that 
\begin{equation}\label{eqn:change-cooridante-1}
  (\mathrm{unit}) \cdot F(x,y,z,s) = u(s)\cdot F_3(y,z,s)   \mod\, (x,(y,z)^4), 
\end{equation}
where $u(s)$ is a unit. 
We write $b(x,y,z,s) = b_1(y,z,s) + b_1'(s)x + b_{>1}(x,y,z,s)$, where $b_1$ is linear in $(y,z)$ and $b_{>1}(x,y,z,s) = 0  \mod\, (x,y,z)^2$. 
Similarly, we write  $c(x,y,z,s) = c_1(y,z,s) + c_1'(s)x + c_{>1}(x,y,z,s)$.  
Then 
\begin{equation}
\label{eqn:change-coordinate-3}
\begin{split}
(\varphi^*G)(x,y,z,s) &=      G(a,b,c,s') \\ 
                &=  a^2 + G_3(b,c,s') +P(b,c,s') \\ 
               &=  G_3( b_1(y,z,s), c_1(y,z,s), s)    \mod\, (x,(y,z)^4). 
\end{split}
\end{equation}
Here, for the notation $G_3( b_1(y,z,s), c_1(y,z,s), s)$,  if  
\[G_3(b,c,s') =  \sum_{i=1}^3 \alpha_i(s')b^ic^{3-i}, \]
then  
\[G_3( b_1(y,z,s), c_1(y,z,s), s)  = \sum_{i=1}^3 \alpha_i(s)b_1(y,z,s)^ic(y,z,s)^{3-i}.\]
This is well-defined since the functions $\alpha_i$ are functions on $S$.   
By comparing \eqref{eqn:change-coordinate-3} with \eqref{eqn:change-cooridante} and \eqref{eqn:change-cooridante-1}, 
we deduce that 
\[
G_3( b_1(y,z,s), c_1(y,z,s), s)  =  u(s)\cdot F_3(y,z,s).
\]
As a consequence, the ideal generated by the coefficients of $F_3(y,z,s)$, regarded as a homogeneous polynomials in $(y,z)$, is contained in the ideal generated by the coefficients of $G_3(b,c,s')$.  
Hence, by symmetry, the ideal   $\mathcal{I}$ we have constructed is independent of the choice of coordinates in this case. 
\\

Next we assume that $G_2\neq 0$. 
In this case, we have 
\begin{equation}
\label{eqn:change-cooridante-2}
(\mathrm{unit}) \cdot F(x,y,z,s) = u(s) \cdot  (x^2 + F_2(y,z,s))   \mod\, (x,y,z)^3, 
\end{equation}
where $u(s)$ is a unit.  
There is a $3\times 3$ matrix $\Theta(s)$, with coefficients as holomorphic functions in $S$ such that 
\[(a,b,c) = (x,y,z)\cdot \Theta(s) \mod\, (x,y,z)^2. \] 
Let $(a_1,b_1,c_1) = (x,y,z)\cdot \Theta(s)$. 
Then we have 
\[
(\varphi^*G)(x,y,z,s) =   a_1^2 + G_2(b_1,c_1,s)    \mod \, (x,y,z)^3.  
\]
Comparing the previous equation with \eqref{eqn:change-cooridante} and $\eqref{eqn:change-cooridante-2}$, 
we deduce that 
\[
a_1^2+ G_2(b_1,c_1,s) = u(s) (x^2+ F_2(y,z,s)). 
\]
Both sides of the equation above can be viewed as  quadratic forms in three variables, with coefficients as holomorphic functions on $S$. 
We denote the LHS by $L(a_1,b_1,c_1)$ and the RHS by $Q(x,y,z)$.   
The ideal $\mathcal{J}$ (respectively of $\mathcal{I}$) generated by the coefficients of  $G_2$ (respectively of $F_2$) is exactly the idea generated by the determinants of all 2-dimensional minors of the symmetric matrix of  $L$ (respectively of $Q$).  
Since $ L((x,y,z)\cdot \Theta(s)) = Q(x,y,z) $,  
we have  $\mathcal{I} \subseteq \mathcal{J}$. 
By symmetry, we conclude that $\mathcal{I} = \mathcal{J}$.    
This completes the proof of the lemma.
\end{proof}

Now we are ready to prove Proposition \ref{prop-first-local-reduction}

\begin{proof}
Let $\tau\colon X' \to X$ be the projective bimeromorphic morphism induced by the functorial desingularization of $S$, as shown at the beginning of Subsection \ref{subsection:blowup}.  
Then, by replacing $X$ with  $X'$, $S$ with its strict transform $S'$ in $X'$,  and  $C$ with the union of its preimage in $S'$ and the intersection of $S'$ with the $\tau$-exceptional locus,   
we may assume that $S$ is smooth. 
By applying Corollary \ref{cor:blowup-hypersurface-sing} and then Lemma \ref{lemma:blowup-double-point}, we may assume that $X$ has double-point singularities along $S$.  

It particular, locally around a point $o_X\in S$, $X$ is the hypersurface in   $\mathbb{D}^3\times S$, defined by the equation 
\[
F(x,y,z,s) = x^2+ F_2(y,z,s)+F_3(y,z,s) + R(y,z,s) = 0, 
\]
where $(x,y,z)$ are the coordinates of $\mathbb{D}^3$ and $s\in S$. 
Moreover, for $i=2,3$,  $F_i(y,z,s)$ is a homogeneous polynomial of degree $i$ in $(y,z)$, with coefficients as holomorphic functions on $S$. 
The function $R(y,z,s)$ is holomorphic and its terms have degrees at least 4 in $(y,z)$. 

Since $X$ has canonical singularities around general points of $S$, by Lemma \ref{lemma:ADE} 
we see that either $F_2$ or $F_3$ is not identically zero. 
Hence we can define the ideal sheaf $\mathcal{I}$ as in Lemma \ref{lemma:ideal-secondary}. 

We only treat the case  when  $F_2\neq 0$.  
The other case can be proved by the same method. 
Let $p\colon X'\to X$ be the projective bimeromorphic morphism induced by the blowup of $S$ at $\mathcal{I}$. 
Let $S'$ be the strict transform of $S$ in $X'$, and let $E$ be the  divisor in $ S'$ defined by $(p|_{S'})^*\mathcal{I}$.  
Then, as shown in Subsection \ref{subsection:blowup},  for any point $o'$ in $S'$, 
there is a neighborhood $U'\subseteq X'$ of $o'$, 
such that $U'$ is the hypersurface in   $ \mathbb{D}^3\times (S'\cap U')$, defined by the function
\[
G(x',y',z',s') = x'^2 + a(s') \cdot G_2(y',z',s') + W(y',z',s'), 
\]
where $a(s')$ is a generator of $\mathcal{O}_{S'}(-E)$ on $S'\cap U'$, $G_2(y',z',s')$ is a nonzero polynomial in $(y',z')$ at any points $s' \in S'\cap U'$, 
and  $W(y',z',s') = 0 \mod \, (y',z')^3$.  

We denote by $\mathcal{I}'$  the ideal sheaf of the closed subset $E \cup (p|_{S'})^{-1}(C)$ in $S'$.  
Let $Y\to X'$ be the projective bimeromorphic morphism induced by the functorial principalization of $\mathcal{I}'$ on $S'$, 
and let $f\colon Y\to X$ be the natural morphism. 
We denote by $S_Y$ the strict transform of $S'$ in $Y$, 
and by $H_Y\subseteq S_Y$   the preimage of $E \cup (p|_{S'})^{-1}(C)$.  
Then $H_Y$ can be seen as a reduced divisor in $S_Y$ and $(S_Y,H_Y)$ is a snc pair. 
Furthermore,  for any point $o\in S_Y$, there is a neighborhood $U\subseteq Y$ of $o$, such that $U$ is a hypersurface in $\mathbb{D}^3\times (S_Y\cap U)$ defined by an equation of standard form with respect to $(S_Y\cap U,H_Y\cap U)$.  
In addition, $Y$ has the same type of canonical singularities at every point of $S_Y\setminus H_Y$, 
and $H_Y$ contains the intersection of $S_Y$ with the $f$-exceptional locus.  
This completes the proof of the proposition. 
\end{proof}

\section{Further improvement of  defining equations}
\label{section:equation}

Let $X$ be a complex analytic variety and let $S\subseteq X_{\sing}$ be an irreducible component of codimension 2 in $X$. 
In the previous section, we reduce the local situation to the case when $X$ is isomorphic to a hypersurface of $\mathbb{D}^3\times S$, and the hypersurface is defined by an equation of standard form.  
In this section, we will perform some bimeromorphic transforms on $X$, 
so that we can improve further the shape of the equation.  
After we prove some elementary results in Subsection \ref{subsection:Hensel}, we will work under the following setup.

\begin{setup}
\label{setup-local} 
Let $X$ be a complex analytic variety of dimension $M+2$, 
and let $S\subseteq X_{\sing}$ be an irreducible component of dimension $M$.  
Let $o\in S$ be a point and we  assume that $S$ is an open neighborhood  of   the origin $o=\mathbf{0}_M$ in $\mathbb{C}^M$.  
We denote by $(T_1,...,T_M)$  the coordinates of $S$ and we define  $H_i=\{T_i=0\}$. 
Let $H=H_1+\cdots +H_n$, where $n\ge 0$ is an integer. 
Assume that $X$ is    a neighborhood of  $\{\mathbf{0}_3\}\times S$,  in  the  hypersurface  in  $\mathbb{D}^3\times S$, defined by an equation of standard form with respect to $(S,H)$, see the beginning Section \ref{section:double-point}.  
Furthermore, the composite  inclusion $S\subseteq X\subseteq \mathbb{D}^3\times S$ is identified with  $\{\mathbf{0}_3\}\times S \subseteq \mathbb{D}^3\times S$.
Assume that $X$ has the same type of canonical singularities at points of $S\setminus H$. 
We will shrink $X$ and $S$ freely around $o\in S$. 
In particular, for every unit holomorphic function on $\mathbb{D}^3 \times S$, we may assume that it admits logarithms. 
In addition, for a holomorphic function $p(x,y,z,s)$ on $\mathbb{D}^3 \times S$ of the shape 
\[
p (x,y,z,s) =  p_0 + q(x,y,z,s),
\]
where  $p_0 \neq 0$ is a complex number and  $q(0,0,0,s) = 0$, 
we can assume that $p$ is a unit function.  
Throughout this section, 
we will make these assumptions without the specification of shrinking $S$. 
\end{setup}

The following proposition is the main objective of in this section.

\begin{prop}
\label{prop:good-shape}
With the notation of  Setup \ref{setup-local},  up to shrinking $S$,  there is an integer $N>0$ such that the following properties hold.   
\begin{enumerate}
\item Assume   that $X$ has $A$-type or $E$-type  singularities at   points of $S\setminus H$. 
Let $f\colon Y\to X$ be the composition of a sequence of blowups, such that each  center is a component $H_i$ of $H$, see Subsection \ref{subsection:blowup-divisor}.  
In particular, we can identify $S$ with its strict transform in $Y$. 
Assume that for each $i=1,...,n$, there are at least $N$ blowups in the sequence whose centers are $H_i$.  
Then   there is a neighborhood $U\subseteq Y$ of $S\subseteq Y$, such that $U$ is isomorphic to a neighborhood of $\{\mathbf{0}_3\} \times S$ in  the hypersurface in  $\mathbb{D}^3\times S$ defined  by a  function of one of  the following shapes
\begin{eqnarray*}
    F(x,y,z,s) &=& x^2 + \alpha(s)y^2 + \beta(s) z^l,   \ l \ge 2,    \\
    F(x,y,z,s) &=& x^2 + \alpha(s)y^3 +  \beta(s) \cdot y(\gamma(s)z +u(y,s) y)^3, \\
    F(x,y,z,s) &=& x^2 + \alpha(s)y^3 +  \beta(s) \cdot  (\gamma(s)z +u(y,s) y)^l, \  l =4,5. 
\end{eqnarray*} 
\item Assume   that $X$ has $D$-type  singularities at  points of $S\setminus H$.
Let $f_1\colon X_1\to X$ be the composition of  blowups at the $H_i$'s as in Subsection \ref{subsection:blowup-divisor}, 
such that   for each $i=1,...,n$, there is exactly one blowup whose  center is $H_i$. 
Let $f_2\colon X_2 \to X_1$ be the basechange over $S[\sqrt[6]{T_1},...,\sqrt[6]{T_n}]$. 
We define $\overline{S} = f_2^{-1}((f_1)^{-1}_*S)$ and define $\overline{H}_i$ as the preimage of $H_i$ in $\overline{S}$. 
Let $f_3\colon Y\to X_2$ be the composition of a sequence of blowups at the   $\overline{H}_i$'s as in Subsection \ref{subsection:blowup-divisor}. 
Assume that for each $i=1,...,n$, there are at least $N$ blowups in the sequence whose centers are $\overline{H}_i$.
Let $f\colon Y\to X$ be the natural morphism, 
let $\overline{S}_Y$ be the  strict transform of $\overline{S}$ in $Y$. 
Then   there is a neighborhood $U\subseteq Y$ of $\overline{S}_Y$, 
such that $U$ is isomorphic to a neighborhood of $\{\mathbf{0}_3\} \times \overline{S}_Y$ in  the hypersurface in  $\mathbb{D}^3\times \overline{S}_Y$ defined  by a function of the shape  
\begin{eqnarray*}
     F(x,y,z,s) &=& x^2 + \alpha(s) y(\gamma(s)z +u(y,s) y )^2  +  \beta(s) y^l,   \ l\ge 3.
\end{eqnarray*}
\end{enumerate}
Here $(x,y,z)$ is the coordinates of $\mathbb{D}^3$,   $l$ is an integer,  $\alpha(s)$, $\beta(s) $ and $\gamma(s)$ are holomorphic functions  on $S$ (or on $\overline{S}$) whose zero loci are contained in $H$ (or in $\sum\overline{H}_i$),  and $u(y,s)$ is a holomorphic function.   

In addition, regarded as a holomorphic function  on $U$, the Cartier divisor defined by $T_i=0$ is equal to the $f$-exceptional divisor  defined by the ideal sheaf $f^{-1}\mathcal{I}(H_i) \cdot \mathcal{O}_Y$, 
where $\mathcal{I}(H_i)$ is the ideal sheaf of $H_i$ in $X$.  
\end{prop}

We note that $U$ in the previous proposition is not assumed to be the whole hypersurface in  $\mathbb{D}^3\times S$ (or in $\mathbb{D}^3\times \overline{S}_Y$) for the following reason.  
In application, we shrink $S$ to get an open subset $S'\subseteq S$ so that the proposition holds on $S'$. 
In particular, for example for the case (1),  
we can get an open neighborhood $W$ of the strict transform $S''$ of $S'$ in $Y$ so that $W$ is isomorphic to a hypersurface  in  $\mathbb{D}^3\times S''$. 
If we need to shrink $X$ to get an open neighborhood $X'$ of $S'$ in $X$, 
then we need to set $U=f^{-1}(X') \cap W$ in order that $f(U)$ is contained in $X'$.

We will prove the proposition in the remainder of the section, by discussing according to the singularity type of $X$ at   points of $S\setminus H$. 

\begin{proof}[{Proof of Proposition \ref{prop:good-shape}}] 
The last paragraph in the statement follows from Remark \ref{rmk:function-exc-div}. 
We can deduce the remainder part of  proposition by combining the results from Lemma \ref{lemma:blowup-equation-A} to  Lemma \ref{lemma:blowup-equation-D}. 
\end{proof}

\subsection{Factorization lemmas}
\label{subsection:Hensel}

We  prove two statements  on holomorphic functions in several variables, which enable us to factorize certain functions. 

\begin{lemma}
\label{lemma:Hensel}
Let $F(y,z,s) = y(y^2+a(y,z,s)yz +  b(y,z,s) z^2 ) +  b(y,z,s)^2  R(y,z,s)$ be a holomorphic function defined in a neighborhood of $\mathbf{0}_{2+N}\in \mathbb{C}^{2+N}$, where $N\ge 0$ is an integer,  $(y,z)$ is the first two coordinates of $ \mathbb{C}^{2+N}$, and  $ s \in \mathbb{C}^N$ is a point which represents the last $N$ coordinates of $ \mathbb{C}^{2+N}$. 
Assume that the terms of $R$ have degrees at least $4$ in $(y,z)$. 
Then, in a neighborhood of $\mathbf{0}_{2+N}$, 
there is a factorization $F =  G_1G_2$ with $G_1=y+ b(y,z,s)c(y,z,s)$ 
and 
\[G_2 = y^2 + a (y,z,s)yz + b(y,z,s) z^2 +  b(y,z,s) \cdot d(y,z,s), \] 
such that   $ c(y,z,s)$ is divisible by $z^2$ 
and that  $d(y,z,s) = 0 \mod \, (y,z)^3$. 
\end{lemma}

\begin{proof} 
We can write 
\[R(y,z,s) = y^3 R_0(y,z,s) +  y^2z^2 \cdot R_1(z,s)  + yz^3 \cdot R_2(z,s) + z^4 \cdot R_3(y,z,s), \]
where  $R_0(y,z,s) = 0 \mod \, (y,z)$. 
Then we have 
\begin{eqnarray*}
F(y,z,s) &=& y \Big( \big( 1+ b(y,z,s)^2 R_0(y,z,s) \big) y^2  \\ 
          &&  +  \big(a(y,z,s) + z\cdot b(y,z,s)^2 \cdot R_1(z,s)\big)yz  \\ 
          &&  + b(y,z,s) \big(1 +   z \cdot b(y,z,s) \cdot R_2(z,s)\big)z^2 \Big) \\
          &&  + b(y,z,s)^2z^4\cdot R_3(y,z,s).  
\end{eqnarray*}
We note that $ 1+b(y,z,s)^2R_0(y,z,s) = 1 \mod \, (y,z) $. 
In particular, it is a unit in a neighborhood of  $\mathbf{0}_{N+2}$.  
Hence, up to replacing 
  $( 1+ b(y,z,s)^2 R_0(y,z,s) )^{-1} F(y,z,s)$ by $F(y,z,s)$,  
\[
\big( 1+ b(y,z,s)^2 R_0(y,z,s) \big)^{-1} \cdot \big(a(y,z,s) + z\cdot b(y,z,s)^2 \cdot R_1(z,s)\big) \mbox{ by } a(y,z,s),
\] 
\[
\big( 1+ b(y,z,s)^2 R_0(y,z,s) \big)^{-1} \cdot \big(1 +   z \cdot b(y,z,s) \cdot R_2(z,s)\big) \cdot b(y,z,s) \mbox{ by }  b(y,z,s), 
\] 
we may assume that $R(y,z,s)$ is divisible by $z^4$. 
In particular, we can write $R(y,z,s) = z^4R_3(y,z,s).$

We define the following holomorphic functions in $(y,z,s)$, 
\[h(y,z,s) =  y^2+a(y,z,s)yz +  b(y,z,s)  z^2 \mbox{ and }  p(y,z,s)= -y - a(y,z,s)z.\]
Then we have 
\[p(y,z,s)y+   h(y,z,s) = b(y,z,s)  z^2.\]  
Up to shrinking the domain of $F$, we may assume that 
$1 + 4p(y,z,s) R_3(y,z,s)$ admits a square root $\delta(y,z,s)$ of the shape 
\[
\delta(y,z,s) = 1 + \sum_{i=1}^{+\infty} \binom{\frac{1}{2}}{i} \cdot \Big(4p(y,z,s) R_3(y,z,s)\Big)^i.
\]
Then $-1+\delta(y,z,s)$ is divisible by $p(y,z,s)$. 
We set 
\[
\gamma(y,z,s) = \frac{1}{2} \cdot p(y,z,s)^{-1}  \cdot (-1+\delta(y,z,s)). 
\] 
Then $\gamma$ is a holomorphic function and satisfies the equation 
\[
p\gamma^2 + \gamma -  R_3 = 0. 
\]
Let $c(y,z,s) = z^2\cdot \gamma(y,z,s)$ and $d(y,z,s)= z^2\cdot \gamma(y,z,s) \cdot  p(y,z,s)$. 
Then we have 
\begin{eqnarray*}
(y+bc)(h+bd) &=& yh + b(yd+hc) + b^2dc \\
             &=&  yh+ bz^2\gamma\cdot (py + h) + b^2z^4  \gamma^2 p \\  
             &=&  yh+ b^2z^4 \gamma + b^2z^4  \gamma^2 p \\ 
             &=&  yh + b^2z^4R_3  + b^2z^4\cdot (-R_3   + \gamma + p\gamma^2 ) \\
             &=&  F
\end{eqnarray*}
This completes the proof of the lemma. 
\end{proof}

\begin{cor}
    \label{cor:Hensel-2}
    Let $F(y,z,s) = y(y^2  +   z^3 ) +   z^5R(y,z,s)$ be  a holomorphic function defined in a neighborhood of $\mathbf{0}_{2+N}\in \mathbb{C}^{2+N}$, where $N\ge 0$ is an integer,  $(y,z)$ is the first two coordinates of $ \mathbb{C}^{2+N}$,
    and  $s\in \mathbb{C}^N$ is a point which represents the last $N$ coordinates of $ \mathbb{C}^{2+N}$.  
    Then, in a neighborhood of $\mathbf{0}_{2+N}$,  
    there is a factorization $F = G_1G_2$ with $G_1=y+  c(y,z,s )$   and $G_2 = y^2+   z^3 +  d(y,z,s ) +e(s)yz^2$, 
    such that  $ c(y,z,s)$ is divisible by $z^2$, and that terms of $d(y,z,s)$ have degrees at least 4 in $(y,z)$.  
\end{cor}

\begin{proof}   
We can write  $R(y,z,s) = C(s) + R_1(y,z,s)$, where   the terms of   $R_1(y,z,s )$ have degrees at least $1$ in $(y,z)$.  
It follows that 
\begin{eqnarray*}
       H(y,z,s)&:=& (y+C(s) \cdot  z^2) (y^2+z^3 -C(s)\cdot z^2y + C(s)^2\cdot z^4) \\ 
       & \ = &y(y^2+z^3) + C(s)\cdot z^5 + C(s)^3\cdot z^6 \\
       &\ =& F(y,z,s) + z^5 \cdot (C(s)^3z - R_1(y,z,s))
\end{eqnarray*}
We set   ${y'}= y+C(s)z^2$.  
Then we have 
\begin{eqnarray*}
H(y,z,s)& =& {y'}\Big( ({y'}-Cz^2)^2 + z^3 - Cz^2({y'}-Cz^2) + C^2z^4 \Big)  \\ 
       & = & {y'}\Big(  {y'}^2 - 3C{y'}z^2 +  z^3(1+3C^2z) \Big).
\end{eqnarray*}  
Thus we can write  
\[
F(y,z,s)={y'}\left(  {y'}^2 - 3C{y'}z^2 +  z^3(1+3C^2z)\right) + z^5R_2({y'},z,s) 
\]
for some holomorphic function $R_2$, whose terms have degrees at least 1 in $(y',z)$. 
Let $u=z$, then we have 
\[
    F(y,z,s)=\overline{F}({y'},z,s,u) =  {y'}({y'}^2  + a(u,s)\cdot {y'}z +  n(u,s)  u \cdot z^2 ) +   u^2 \cdot z^3R_2({y'},z,s ),  
\]
where $a(u,s)= -3C(s)u$,  and $n(u,s) = (1+3C(s)^2u)$ is a unit.  
Hence we can apply  Lemma \ref{lemma:Hensel} to $\overline{F}(y',z,s,u)$,   
and deduce a factorization $\overline{F}= \overline{G}_1\overline{G}_2$ with $\overline{G}_1={y'} +  \overline{c}({y'},z, s )$ such that $\overline{c}$ is divisible by $z^2$.   
Furthermore, we have 
\[\overline{G}_2 =  {y'}^2 - 3C(s){y'}z^2 +  z^3(1+3C(s)^2z) +  \overline{d}({y'},z,s ), \]
where the terms of $\overline{d}$ have degrees at least $4$ in $(y',z)$.
Since $y'= y+ C(s)z^2$, we obtain a factorization $ {F}= {G}_1 {G}_2$ with $G_1={y}+  c(y,z,s)$ and $G_2 = y^2+ z^3 + d(y,z,s)+e(s)yz^2$,  
such that $c$ is divisible by $z^2$ and  that the terms of $d$  have degree at least $4$ in $(y,z)$.
This completes the proof of the corollary.
\end{proof}

\subsection{Case of $A$-type singularities}
In this subsection, we  treat Proposition \ref{prop:good-shape} in the case when $X$ has $A$-type singularities at points of $S\setminus H$.

\begin{lemma}
  \label{lemma:blowup-equation-A} 
  Proposition \ref{prop:good-shape} holds if $X$ has $A$-type singularities at points of $S\setminus H$. 
\end{lemma}

\begin{proof} 
We may assume that $X\subseteq \mathbb{D}^3 \times S$ is defined by an equation of standard form with respect to $(S,H)$, which is of  the shape   
\[ x^2 + F_2(y,z,s) + F_3(y,z,s)+R(y,z,s) = 0, \]
where $F_2$ and $F_3$ are homogeneous in $(y,z)$ of degree 2 and 3 respectively, and $R$ is a holomorphic function such that $R(y,z,s)  = 0 \mod \, (y,z)^3$.  
By assumption, when  $s\not\in H$, the term $F_2(y,z,s)$ is not zero (see Lemma \ref{lemma:ADE}).  
Since the defining equation is of standard form with respect to $(S,H)$, 
up to a linear change of the coordinates $(y,z)$, we may assume that 
$F_2(y,z,s) = a(s)(y^2 + b(s)yz + c(s)z^2)$, where the zero locus of $a$ is contained in $H$.  
By blowing up $X$  at the   $H_i$'s for several times,  
we may assume that $a(s)$ divides $F_3$ and $R$ (see Remark  \ref{rmk:blowup-divisor-1}).  
Hence $X$ is defined by  
\[
x^2 + a(s) \Big( y^2 + b(s)yz + c(s)z^2 +R'(y,z,s) \Big) = 0,
\]
where  $R' (y,z.s) = 0 \mod \,  (y,z)^3$. 
By applying Weierstrass preparation theorem to the function in the parenthesis after $a(s)$ above,  
with respect to the variable $y$, we see that $X$ is defined by 
\[
x^2 + a(s)\cdot  ( \mathrm{{unit}}) \cdot  ( y^2  + y \cdot R_1(z,s) + R_2(z,s) ) = 0,
\]
for some holomorphic functions $R_1$ and  $R_2$.   
Replacing $y$ by $y-\frac{1}{2}R_1(z,s)$,  the defining equation becomes 
\[
x^2 + a(s) \cdot ( \mathrm{{unit}}) \cdot  \left(   y^2 +     z^e \cdot R_3(z,s)  \right) =0, 
\]
where $e\ge 2$ is an integer and  $R_3$ is not divisible by $z$. 

By assumption, there is some integer $l\ge 1$ such that $X$ has $A_l$-type singularities  at any point $s \in S\setminus H$.  
Hence $e=l+1$,  and 
if $b(s)$ is the term in $R_3$ of degree $0$ in $z$, then  the zero locus of $b(s)$ is contained in $H$. 
Thus, by blowing up $X$  at the   $H_i$'s,   
we may assume that $b(s)$ divides $R_3(z,s)$, see Remark \ref{rmk:blowup-divisor-1}.  
Then the defining equation is of  the shape 
\[
x^2+ a(s) \cdot (\mathrm{unit})\cdot  y^2 + b(s) \cdot (\mathrm{unit})\cdot z^{l+1} = 0, 
\]
where the zero locus of $b(s)$ is contained in $H$. 
Hence, up to a change of coordinates, the defining equation becomes 
\[
x^2 + \alpha(s) y^2 + \beta(s) z^{l+1} = 0,  
\]
as in the statement of Proposition \ref{prop:good-shape}.
In the end, we note that if we continue blowing up $X$ at some $H_i$, then the defining equation remains the same shape. 
This completes the proof of the lemma. 
\end{proof}

\subsection{Case of $E$-type singularities}

In this subsection, we  treat Proposition \ref{prop:good-shape} in the case when $X$ has $E$-type singularities at points of $S\setminus H$.  
We will discuss the cases of $E_6$, $E_7$ and $E_8$ separately in three lemmas.

\begin{lemma}
\label{lemma:blowup-equation-E6} 
Proposition \ref{prop:good-shape} holds if $X$ has $E_6$-type singularities at points of $S\setminus H$. 
\end{lemma}

\begin{proof}
We may assume that $X\subseteq \mathbb{D}^3 \times S$ is defined by an equation of standard form with respect to $(S,H)$, which is of the shape
    \[ x^2 + F_2(y,z,s) + F_3(y,z,s)+R(y,z,s) =0, \]
where $F_2$ and $F_3$ are homogeneous in $(y,z)$ of degree 2 and 3 respectively, and $R(y,z,s) = 0 \mod \, (y,z)^4$.   
By Lemma \ref{lemma:ADE}, when  $s\not\in H$, the term $F_2$ is  zero and $F_3$ is a cube.  
It follows that $F_2=0$. 
Moreover, since the equation is of standard form with respect to $(S,H)$, 
up to shrinking $X$ and  up to a linear change of the coordinates $(y,z)$, 
we can assume that   
$F_3$ is  of the shape $F_3(y,z,s) = a(s)(y-u(s)z)^3$ for some holomorphic function  $a(s)$ whose zero locus is contained in $H$, 
and for some holomorphic function $u(s)$.   
Replacing $y$ by $y+u(s)z$,  we may assume that  
\[ F_3(y,z,s)=a(s)y^3.\]  

By blowing up $X$ for several times at the $H_i$'s as in Subsection \ref{subsection:blowup-divisor}, we may assume that $a(s)$ divides $R(y,z,s)$, see Remark \ref{rmk:blowup-divisor-1}. 
Then we can write  
\begin{eqnarray*}
    F_3(y,z,s) + R(y,z,s)  
  = a(s) \cdot  (y^3 + \cdots  ), 
\end{eqnarray*}
By applying  Weierstrass preparation theorem to the series in the parenthesis after $a(s)$ in the equation above, with respect to $y$, 
we can write
\begin{eqnarray*}
     F_3(y,z,s) + R(y,z,s) &=& (\mathrm{unit})\cdot  a(s) \cdot  \Big( y^3   + \omega_0(z,s)z^2y^2 \\ 
                            &&  +  \omega_1(z,s)z^{3+p}y + \omega_2(z,s)z^{4+q} \Big), 
\end{eqnarray*}
where $q,  p  \ge 0$ are integers,   $\omega_1$ and  $\omega_2$ are either zero or non divisible by $z$. 
Up to replacing $y$ by $y- \frac{1}{3}\omega_0(z,s)z^2$,   we can assume that  
\[
F_3(y,z,s) + R(y,z,s) = (\mathrm{unit})\cdot  a(s) \cdot  \Big( y^3   + \omega_1(z,s)z^{3+p}y + \omega_2(z,s)z^{4+q} \Big).  
\]    
By merging the terms of $\omega_1(z,s)z^{3+p}y$ divisible by $z^{4+q}$ into $ \omega_2(z,s)z^{4+q}$,  we can assume that $q\ge p$ and write 
\[
F_3(y,z,s) + R(y,z,s) = (\mathrm{unit})\cdot  a(s) \cdot  \Big( y^3   + v(z,s)z^{3+p}y + w(y,z,s)z^{4+q} \Big),  
\]
such that either $v$ is zero or $z$ does not divide $v$. 
Moreover,  the terms of  $v(z,s)$ have degrees at most $q-p$ in $z$. 
In addition, either $w(y,z,s)$ is  zero, or  $w(0,0,s)\neq 0$ for a general point $s\in S$.   
\\

Since $X$ has $E_6$-type singularities at points of $S\setminus H$, by Lemma \ref{lemma:ADE}, 
we deduce that  $q = p = 0$, and  $v(z,s) = v(s)$ is a holomorphic function on $S$. 
Lemma \ref{lemma:ADE} also implies that $w(y,z,s) = b(s) +      b'(y, z,s)$,  where
the zero locus of $b(s)$ is contained in $H$ and   $b'(y,z,s) = 0 \mod \, (y,z)$.  
By blowing up $X$ at the  $H_i$'s for several times as in Subsection \ref{subsection:blowup-divisor}, 
we may assume that $w(y,z,s)= \eta(y,z,s) \cdot b(s)$, where $\eta$ is a unit, see Remark \ref{rmk:blowup-divisor-1}. 

Let  $T=T_1\cdots T_n$.
Since the zero locus of   $b(s)$ is contained in $H=\{T=0\}$, there is an integer $e'\ge 0$ such that $T^{e'}$ is divisible by  $b(s)$. 
Let  $e=3e'$. 
By blowing up $X$ at each $H_i$ for  $2e$ times as in Subsection \ref{subsection:blowup-divisor},  
the defining function $F(x,y,z,s)$ becomes   
  \[
  F(x,y,z,s) = x^2 + (\mathrm{unit})\cdot  a(s) \cdot  \Big(   T^{2e}y^3 +  T^{4e} v(s) z^{3}y +  \eta(y,z,s) \cdot b(s) \cdot  T^{4e} z^{4}
  \Big), 
  \] 
where, by abuse of notation, we still denote $\eta(T^{2e}y,T^{2e}z,s)$ by $\eta(y,z,s)$.    
Since  $e-2e' = e'$, by the choice of $e'$,  there is a function $m(s)$   such that 
\[b(s) \cdot  m(s) \cdot T^{2e'}  = v(s) \cdot T^{e}.
\] 
Hence we can write  
\begin{eqnarray*}
F(x,y,z,s) &=& x^2 + (\mathrm{unit})\cdot  a(s) \cdot   \Big( T^{2e}y^3 \\ 
           &&  +   b(s) \cdot  m(s) \cdot   (T^{e} z)^{3} \cdot  T^{2e'}y  \\ 
           &&  +  \eta(y,z,s)\cdot b(s) \cdot  (T^{e} z)^{4}
  \Big).       
\end{eqnarray*} 
We set \[\zeta = T^e z   +  \frac{1}{4} \eta(y,z,s)^{-1} m(s) \cdot T^{2e'}y,\]
so that 
\begin{eqnarray*}
   F(x,y,z,s) &=& x^2 +   (\mathrm{unit})\cdot  a(s) \cdot  \Big( T^{2e}y^3 + \nu_1(y,z,s)\cdot b(s)  (T^{2e'}y)^4  \\ 
    && + \nu_2(y,z,s)\cdot b(s) \zeta \cdot (T^{2e'}y)^3 \\
   && +    \nu_3(y,z,s)\cdot b(s) \zeta^2\cdot (T^{2e'}y)^2+    \eta(y,z,s)\cdot  b(s)\zeta^{4} \Big), 
\end{eqnarray*}
for some holomorphic functions $\nu_1,\nu_2, \nu_3$. 
We remark that $T^{2e}y^3 = (T^{2e'}y)^3$ and we  consider the following part in the last expression on $F(x,y,z,s)$ above, 
\begin{eqnarray*}
    &&     (T^{2e'}y)^3  + \nu_1(y,z,s)\cdot b(s)  (T^{2e'}y)^4 \\ 
    &&  +\nu_2(y,z,s)\cdot b(s) \zeta \cdot (T^{2e'}y)^3     \\
    && +\nu_3(y,z,s)\cdot b(s) \zeta^2\cdot (T^{2e'}y)^2 \\
    &=&  ( \mathrm{unit}) \cdot ( T^{2e'}y)^3  +\nu_3(y,z,s)\cdot b(s) \zeta^2\cdot (T^{2e'}y)^2. 
\end{eqnarray*}
It follows that there are  some holomorphic functions $\mu(y,z,s)$ and $\sigma(y,z,s)$ such that  the function $F$ can be written as 
\begin{eqnarray*}
    F(x,y,z,s) &=& x^2 +  (\mathrm{unit})\cdot  a(s) \cdot  \Big( (\mathrm{unit})\cdot  (T^{2e'}y + \mu(y,z,s) \zeta^2 )^3 \\ 
               && +    \sigma(y,z,s) \cdot b(s) \zeta^4    +   \eta(y,z,s)\cdot b(s)\zeta^4 \Big).
\end{eqnarray*}  
Furthermore, $\sigma(y,z,s) = 0 \mod \, (y,z)$.  
Hence the function $F$ can be written as 
\[
 F(x,y,z,s) = x^2 +  a(s)  \cdot (\mathrm{unit}) \cdot  \Big( (\mathrm{unit})\cdot  (T^{2e'}y + \mu(y,z,s) \zeta^2 )^3   +   (\mathrm{unit})\cdot b(s)\zeta^4 \Big).
\]
We recall  that $\zeta =  T^e z   +  \frac{1}{4} \eta(y,z,s)^{-1} m(s) \cdot T^{2e'}y$.  
Thus  there is some holomorphic function $\mu'(y,z,s)$ such that 
\begin{eqnarray*}
T^{2e'}y + \mu(y,z,s) \zeta^2 &=&  (\mathrm{unit}) \cdot  T^{2e'} y + \mu(y,z,s)T^{2e}z^2\\
                              & & (\mathrm{unit}) \cdot  T^{2e'} \Big(y + \mu'(y,z,s) T^{2e-2e'}z^2\Big). 
\end{eqnarray*}
If we set $\overline{y}= y + \mu'(y,z,s) T^{2e-2e'}z^2$, 
then $y-\overline{y}$ is divisible by $T^{2e-2e'}z^2$.
We deduce that  
\begin{eqnarray*}
\zeta &=& T^e z   +  \frac{1}{4} \eta(y,z,s)^{-1} m(s) \cdot T^{2e'}y \\
      &=&     T^e z   +   \frac{1}{4} \eta(y,z,s)^{-1} m(s) \cdot T^{2e'}\overline{y} +   \frac{1}{4} \eta(y,z,s)^{-1} m(s) \cdot T^{2e'}(y-\overline{y})     \\ 
       &=&   (\mathrm{unit}) \cdot  T^e z   +   \frac{1}{4} \eta(y,z,s)^{-1} m(s) \cdot T^{2e'}\overline{y}   \\
      &=& (\mathrm{unit}) \cdot  \Big( (\mathrm{unit}) \cdot T^e z  +  m(s) \cdot T^{2e'}\overline{y}\Big).
\end{eqnarray*}
Therefore, up to replacing  $\overline{y}$ by  $y$, the function $F$ can be written as 
\begin{eqnarray*}
    F(x,y,z,s) &=& x^2 +  R_1(y,z,s)\cdot a(s)  \cdot T^{2e} y^3 \\
      && +   R_2(y,z,s)\cdot a(s)b(s) \cdot  ( R_3(y,z,s)  T^e z   +  m(s) \cdot T^{2e'}y)^4,  
\end{eqnarray*}
where $R_1,R_2,R_3$ are units. 
We let $R_4=R_2R_1^{-1}$  and   $R_5 = R_1R_4^{-3}=R_2R_4^{-4}$, so that  
\begin{eqnarray*}
    F(x,y,z,s) &=& x^2 +  R_1R_4^{-3}\cdot a(s) \cdot  T^{2e} (R_4 y)^3 \\
      && +   R_2R_4^{-4}\cdot a(s)b(s) \cdot ( R_3R_4  T^e z   +  m(s) \cdot T^{2e'}R_4y)^4.  
\end{eqnarray*} 
Up to replacing $y$ by $R_4^{-1}y$, we deduce that $F$ is of the shape
\begin{eqnarray*}
    F(x,y,z,s) &=& x^2 +  R_5\cdot a(s) \cdot T^{2e} y^3 \\
      && +   R_5\cdot a(s)b(s) \cdot  ( R_3R_4  T^e z   +  m(s) \cdot T^{2e'}y)^4.   
\end{eqnarray*}
Then $X$ is defined by 
\begin{eqnarray*}
   R_5^{-1} x^2 +   a(s) \cdot T^{2e} y^3   +    a(s)b(s)  \cdot ( R_3R_4 T^e z   +    m(s) \cdot T^{2e'}y)^4 =0.   
\end{eqnarray*} 
Hence    there is a change of coordinates  so that $X$ can be defined by 
\[
x^2 +   \alpha(s)  y^3 +   \beta(s) (\gamma(s) z   +  u(y,s)y)^4 =0, 
\]
where the zero loci of $\alpha(s)$, $\beta(s)$ and $\gamma(s)$ are contained in $H$. 
In the end, we note that if we continue blowing up $X$ at some $H_i$, then the defining equation remains the same shape. 
This completes the proof of the lemma.
 \end{proof}

In the following lemma, we treat the case of $E_8$-type  singularity. 
Its proof proceeds in the same way as the one for the case of $E_6$-type singularity. 
 
\begin{lemma}
\label{lemma:blowup-equation-E8} 
Proposition \ref{prop:good-shape} holds if $X$ has $E_8$-type singularities at points of $S\setminus H$. 
\end{lemma}

\begin{proof}
As in the case of Lemma \ref{lemma:blowup-equation-E6}, we can reduce to the situation when  
$X$ is a hypersurface in $\mathbb{D}^3\times S$ defined by the equation 
\[
F(x,y,z,s)=x^2 + F_3(y,z,s) + R(y,z,s) =0,  
\]
with 
\[
F_3(y,z,s) + R(y,z,s) =  (\mathrm{unit})\cdot  a(s)  \cdot   \Big( y^3   + v(z,s)z^{3+p}y + w(y,z,s)z^{4+q} \Big), 
\]
where $q \ge p  \ge 0$ are integers. 
Either $v$ is zero or $z$ does not divide $v$. Furthermore, 
the terms of  $v(z,s)$ have degrees at most $q-p$ in $z$. 
In addition, either $w(y,z,s)$ is  zero, or  $w(0,0,s)\neq 0$ for a general point $s\in S$.   

Since $X$ has $E_8$ singularities at points of  $S\setminus H$, by Lemma \ref{lemma:ADE}, 
we deduce that  $q = p = 1$, and  $v(z,s) = v(s)$ is a holomorphic function on $S$. 
Lemma \ref{lemma:ADE} also implies that $w(y,z,s) = b(s) +   b'(y,z,s)$ such that 
the zero locus of $b(s)$ is contained in $H$ and that $b'(y,z,s) = 0 \mod \, (y,z)$. 
By blowing up $X$ at the  $H_i$'s for several times as in  Subsection \ref{subsection:blowup-divisor}, 
we may assume that $w(y,z,s)= \eta(y,z,s) \cdot b(s)$, where $\eta$ is a unit, see Remark \ref{rmk:blowup-divisor-1}.

Let  $T=T_1\cdots T_n$.
Since the zero locus of the function $b(s)$ is contained in $H=\{T=0\}$, there is an integer $e'\ge 0$ such that $T^{e'}$ is divisible by  $b(s)$. 
We define
\[e = 5d =15e'.  \]
By blowing up $X$ at each $H_i$ for $e$ times,   as in  Subsection \ref{subsection:blowup-divisor},  
the function $F$ becomes   
\begin{eqnarray*}
F(x,y,z,s) &=& x^2 +   (\mathrm{unit})\cdot   a(s) \cdot\Big(   T^{e}y^3  \\  
           &&  +  T^{3e} v(s) z^{4}y +  \eta(y,z,s) \cdot b(s) \cdot  T^{3e} z^{5} \Big),    
\end{eqnarray*}
where, by abuse of notation, we still denote $\eta(T^ey,T^ez,s)$ by $\eta(y,z,s)$.  
Since $e=5d$,  we remark that, in the last line of the previous equation, 
\[
 T^{3e} v(s) z^{4}y +  \eta(y,z,s) \cdot b(s) \cdot  T^{3e} z^{5} = (T^{3d} z)^4 \cdot v(s) \cdot T^{3d}y + \eta(y,z,s) \cdot b(s) \cdot (T^{3d}z)^5. 
\]
Since  $3d - 5e' = 4e'$, by the choice of $e'$, there is a function $m(s)$   such that 
\[
b(s) \cdot  m(s) \cdot T^{5e'}  = v(s) \cdot T^{3d}.
\] 
Hence we can write  
\begin{eqnarray*}
F(x,y,z,s) &=& x^2 +  (\mathrm{unit})\cdot   a(s) \cdot  \Big(    T^{e}y^3 \\ 
           &&  +   b(s) \cdot  m(s) \cdot   (T^{3d} z)^{4}  \cdot T^{5e'}y  \\ 
           &&  +  \eta(y,z,s)\cdot b(s) \cdot  (T^{3d} z)^{5}
  \Big).       
\end{eqnarray*}
We set  \[\zeta = T^{3d} z   + \frac{1}{5} \eta(y,z,s)^{-1}  m(s) \cdot T^{5e'}y,\] 
so that 
\begin{eqnarray*}
   F(x,y,z,s) &=& x^2 + (\mathrm{unit})\cdot   a(s) \cdot \Big( T^{e}y^3 + \nu_1(y,z,s)\cdot b(s)  \cdot  (T^{5e'}y)^5  \\ 
    && + \nu_2(y,z,s)\cdot b(s) \zeta \cdot (T^{5e'}y)^4  + \nu_3(y,z,s)\cdot b(s) \zeta^2 \cdot (T^{5e'}y)^3 \\
   && +    \nu_4(y,z,s)\cdot b(s) \zeta^3\cdot (T^{5e'}y)^2+    \eta(y,z,s)\cdot  b(s)\zeta^{5} \Big), 
\end{eqnarray*}
for some holomorphic functions $\nu_1,\nu_2, \nu_3, \nu_4$. 
We remark that $T^{e}y^3 = (T^{5e'}y)^3$ and we  consider the following part in the last expression of $F(x,y,z,s)$ above, 
\begin{eqnarray*}
    &&    T^{e}y^3 + \nu_1(y,z,s)\cdot b(s)  (T^{5e'}y)^5  + \nu_2(y,z,s)\cdot b(s) \zeta \cdot (T^{5e'}y)^4 \\
    && + \nu_3(y,z,s)\cdot b(s) \zeta^2 \cdot (T^{5e'}y)^3   +    \nu_4(y,z,s)\cdot b(s) \zeta^3\cdot (T^{5e'}y)^2 \\
    &=&  ( \mathrm{unit}) \cdot  (T^{5e'}y)^3  +\nu_4(y,z,s)\cdot b(s) \zeta^3\cdot (T^{5e'}y)^2. 
\end{eqnarray*}
Then there are  some holomorphic functions $\mu(y,z,s)$ and $\sigma(y,z,s)$ such that  the function $F$ can be written as 
\begin{eqnarray*}
    F(x,y,z,s) &=& x^2 + (\mathrm{unit})\cdot   a(s) \cdot  \Big( (\mathrm{unit})\cdot  (T^{5e'}y + \mu(y,z,s) \zeta^3 )^3 \\ 
               && +    \sigma(y,z,s) \cdot b(s) \zeta^6    +   \eta(y,z,s)\cdot b(s)\zeta^5 \Big).
\end{eqnarray*}   
Hence the function $F$ can be written as 
\[
 F(x,y,z,s) = x^2 +  a(s) \Big( (\mathrm{unit})\cdot  (T^{5e'}y + \mu(y,z,s) \zeta^3 )^3   +   (\mathrm{unit})\cdot b(s)\zeta^5 \Big).
\]
We recall that $\zeta = T^{3d} z   + \frac{1}{5} \eta(y,z,s)^{-1}  m(s) \cdot  T^{5e'}y$. 
Hence there is some holomorphic function $\mu'(y,z,s)$ such that 
\begin{eqnarray*}
    T^{5e'}y + \mu(y,z,s) \zeta^3 &=& (\mathrm{unit}) \cdot  T^{5e'}y + \mu(y,z,s) \cdot T^{9d}z^3 \\
       &=& (\mathrm{unit}) \cdot  T^{5e'} \Big(y + \mu'(y,z,s)\cdot T^{9d-5e'}z^3 \Big).  
\end{eqnarray*}
Let $\overline{y}= y + \mu'(y,z,s)\cdot T^{9d-5e'}z^3$. 
Then $y-\overline{y}$ is divisible by  $ T^{9d-5e'}z^3$, and we have 
\begin{eqnarray*}
    \zeta &=& T^{3d} z   + \frac{1}{5} \eta(y,z,s)^{-1}  m(s) \cdot  T^{5e'}y \\
          &=&  T^{3d} z   + \frac{1}{5} \eta(y,z,s)^{-1}  m(s) \cdot  T^{5e'}\overline{y} 
                 + \frac{1}{5} \eta(y,z,s)^{-1}  m(s) \cdot  T^{5e'}(y-\overline{y})\\
          &=& (\mathrm{unit})\cdot T^{3d} z +    \frac{1}{5} \eta(y,z,s)^{-1}  m(s) \cdot  T^{5e'}\overline{y} \\
          &=& (\mathrm{unit})\cdot  \Big(  (\mathrm{unit})\cdot T^{3d} z +  m(s) \cdot  T^{5e'}\overline{y}  \Big). 
\end{eqnarray*}
Up to replacing $\overline{y}$ with $y$, we can write 
\begin{eqnarray*}
    F(x,y,z,s) &=& x^2 +  R_1(y,z,s)\cdot a(s)T^{e} y^3 \\
       && + R_2(y,z,s)\cdot a(s)b(s) \cdot (R_3(y,z,s)T^{3d} z   +    m(s) \cdot T^{5e'}y)^5,   
\end{eqnarray*}
where $R_1,R_2,R_3$ are units. 
There is a holomorphic function $R_4$ such that $R_4^2=R_2R_1^{-1}$. 
Then we can set $R_5 = R_2R_4^{-5}=R_1R_4^{-3}$, and we have 
\begin{eqnarray*}
    F(x,y,z,s) &=&   x^2 +  R_5\cdot a(s)T^{e} (R_4y)^3 \\
               &  & +    R_5 a(s)b(s)  (R_3R_4 T^{3d} z  + m(s) T^{5e'}R_4 y)^5.   
\end{eqnarray*}
Up to replacing $y$ by $R_4^{-1}y$, we see that  $X$ is defined by 
\[
  R_5^{-1}x^2 +  a(s)T^{e} y^3 +     a(s)b(s)  (R_3R_4T^{3d} z   + m(s) T^{5e'} y)^5  =0 .
\] 
Hence   there is a change of coordinates  so that the defining equation becomes
\[
 x^2 +   \alpha(s)  y^3 +   \beta(s) (\gamma(s) z   +  u(y,s)y)^5 = 0, 
\]
where the zero locus of $\alpha(s)$, $\beta(s)$ and $\gamma(s)$ are contained in $H$.  
In the end, we note that if we continue blowing up $X$ at some $H_i$, then the defining equation remains the same shape. 
This completes the proof of the lemma.
\end{proof}

We will now treat the case of $E_7$-type singularity. 
Its proof is slightly different from the ones of $E_6$-type or $E_8$-type singularities.

\begin{lemma}
\label{lemma:blowup-equation-E7} 
Proposition \ref{prop:good-shape} holds if $X$ has $E_7$-type singularities at points of $S\setminus H$. 
\end{lemma}

\begin{proof}
As in the case of Lemma \ref{lemma:blowup-equation-E6}, we can reduce to the situation when  
$X$ is a hypersurface in $\mathbb{D}^3\times S$ defined by the equation 
\[
F(x,y,z,s)=x^2 + F_3(y,z,s) + R(y,z,s) =0,  
\]
with 
\[
F_3(y,z,s) + R(y,z,s) =  (\mathrm{unit})\cdot  a(s)  \cdot  \Big( y^3   + v(z,s)z^{3+p}y + w(y,z,s)z^{4+q} \Big), 
\]
where $q \ge p  \ge 0$ are integers. 
Either $v$ is zero or $z$ does not divide $v$. Furthermore, 
the terms of  $v(z,s)$ have degrees at most $q-p$ in $z$. 
In addition, either $w(y,z,s)$ is  zero, or  $w(0,0,s)\neq 0$ for a general point $s\in S$.    

Since $X$ has $E_7$ singularities at points of $S\setminus H$, by Lemma \ref{lemma:ADE}, 
we deduce that  $p = 0$ and $q\ge 1$.  
Lemma \ref{lemma:ADE} also implies that $v(z,s) = b(s) + z  \cdot \omega(z,s)$ such that 
the zero locus of $b(s)$ is contained in $H$ and that $\omega(z,s)$ is a holomorphic function. 
By blowing up $X$ at the $H_i$'s for several times as in Subsection \ref{subsection:blowup-divisor},  
we may assume that $v(z,s)= \eta(z,s) \cdot b(s)$, where $\eta$ is a unit, see Remark \ref{rmk:blowup-divisor-1}. 
Up to multiplying $\eta$ by a unit function, 
we can assume that $b(s) = \prod_i T_i^{l_i}$ for some integers $l_1,...,l_n\ge 0$. 
Then we have 
\[
F_3(y,z,s) + R(y,z,s) = (\mathrm{unit})\cdot a(s) \cdot  \Big(   y^3   + \eta(z,s)  b(s) yz^{3} + w(y,z,s)z^{4+q} \Big).   
\]
By blow up $X$ at  each $H_i$ for $e_i$ times as in Subsection \ref{subsection:blowup-divisor},   
we can write 
\begin{eqnarray*}
    F(x,y,z,s) &=& x^2 + (\mathrm{unit})\cdot a(s) \cdot  \prod_i T_i^{e_i} \cdot  \Big(  y^3 \\
               && + \prod_i T_i^{e_i+l_i} \eta(z,s)  yz^{3} + \prod_i T_i^{qe_i+e_i} w(y,z,s)z^{4+q} \Big), 
\end{eqnarray*}
where by abuse of notation, we still write $\eta(z,s)$ for $\eta( \prod_i T_i^{e_i} \cdot z,s)$ and $w(y,z,s) $ for $w(\prod_i T_i^{e_i} \cdot y,\prod_i T_i^{e_i} \cdot z,s)$. 
We choose $e_i$ so that $e_i+l_i = 3e_i'$ for some integer $e_i'>0$. 
We note that 
\begin{eqnarray*}
qe_i +e_i  - (4+q)e_i' =  (q+1)e_i - \frac{1}{3}(4+q)(e_i+l_i) = 
   \frac{1}{3}(2q-1) e_i  - \frac{1}{3}(4+q)l_i.  
\end{eqnarray*}
Since $q \ge 1$, the RHS above tends to $+\infty$ if $e_i$ tends to $+\infty$.  
Hence we can pick $e_i$ sufficiently large so that 
\[
\prod_i T_i^{(4+q)e_i'}   \mbox{ divides }   \prod_i T_i^{qe_i+e_i} w(y, z,s).
\]
Then there is some holomorphic function $m(y,z,s)$ such that 
\[
   \prod_i T_i^{qe_i+e_i} w(y, z,s)  \cdot z^{4+q}  = m(y,z,s) \cdot  \prod_i T_i^{(4+q)e_i'} \cdot  z^{4+q}.
\]
Let $\zeta = z\cdot \prod_i T_i^{e_i'}$. 
Then  we have  
\begin{eqnarray*}
    F(x,y,z,s) &=& 
                   x^2 + (\mathrm{unit})\cdot a(s) \cdot \prod_i T_i^{e_i} \cdot \Big(   y^3   
                    +  \eta(z,s)  y\zeta^{3} + m(y,z,s)\zeta^{4+q} \Big) \\
               &=& x^2 + (\mathrm{unit})\cdot a(s) \cdot \prod_i T_i^{e_i}  \\
                && \cdot  \Big( y(\eta(z,s)^{-1}  y^2 + \zeta^3  ) + \eta(z,s)^{-1} m(y,z,s)\zeta^{4+q}
                   \Big).
\end{eqnarray*}
Replacing $y$ by $y\cdot \eta(z,s)^{\frac{1}{2}}$,   
we can write 
\[
F(x,y,z,s)=x^2+ (\mathrm{unit})\cdot a(s) \cdot \prod_i T_i^{e_i} \cdot \Big(y(y^2+\zeta^3) + m'(y,z,s)\zeta^{4+q} \Big)
\]
for some holomorphic function $m'(y,z,s)$.  
By applying Corollary \ref{cor:Hensel-2} 
to the function 
\[
y(y^2 +   \zeta^3  ) +    m'(y,z,s)\zeta^{4+q}, 
\]
it can be factorized into 
\[   \Big(y  +  c(y,\zeta,z,s)\Big) 
   \Big( y^2 + \zeta^3 + \sigma(y,\zeta,z,s) + \nu(s)\cdot \zeta^2y
   \Big),
\]
where $\zeta^2$ divides $c(y,\zeta,z,s)$ and   $\sigma(y,\zeta,z,s) = 0 \mod \, (y,\zeta)^4$.
By replacing   $y + c(y,\zeta,z,s)$ with $y$, 
the previous term becomes 
\begin{eqnarray*}
&&   y \Big(y^2 + \nu_1(y,\zeta,z,s) y\zeta^2  + (\mathrm{unit})\cdot \zeta^3 +  \nu_2(y,\zeta,z,s)\Big) \\
&=&   y\Big((\mathrm{unit}) \cdot y^2  + \nu_3(z,s)y\zeta^2  +   (\mathrm{unit})\cdot \zeta^3 \Big), 
\end{eqnarray*} 
where $\nu_1,\nu_2$ and $\nu_3$ are holomorphic functions, and   $\nu_2(y,\zeta,z,s) = 0 \mod \, (y,\zeta)^4$.  
There is some  $\zeta'$   of the shape
\[
\zeta' = \zeta +   (\mathrm{unit})\cdot \nu_3(z,s) y =  z\cdot \prod_i T_i^{e_i'} +  (\mathrm{unit})\cdot \nu_3(z,s) y, 
\]
so that the previous function becomes  
\begin{eqnarray*}
 &&   y \Big((\mathrm{unit}) \cdot y^2 + \nu_4(y, \zeta',z,s) y^2  + (\mathrm{unit})\cdot (\zeta')^3  \Big) \\ 
 &=&   y\Big((\mathrm{unit}) \cdot y^2 + (\mathrm{unit})\cdot (\zeta')^3  \Big), 
\end{eqnarray*} 
where  $\nu_4(y, \zeta',z,s) = 0 \mod \, (y,\zeta')$.  

Therefore, we deduce that 
\begin{eqnarray*}
F(x,y,z,s)   &=& x^2 + a(s) \prod_i T_i^{e_i} \cdot  R_1(y,z,s) \cdot  y^3 \\
           && +  a(s) \prod_i T_i^{e_i}  \cdot R_2(y,z,s) \cdot y\Big(   R_3(y,z,s) z\cdot \prod_i T_i^{e_i'} +  \nu_3(z,s) y\Big)^3,  
   \end{eqnarray*}
where $R_1, R_2, R_3$ are units. 
We write $\nu_3(z,s) = \mu_1(s) + z\mu_2(z,s)$.   
Let $T' = \prod_i T_i^{e'_i}$. 
By blowing up  $X$ at  each $H_i$  for $e_i'$ times as in Subsection \ref{subsection:blowup-divisor}, 
the defining function $F$ becomes 
\begin{eqnarray*}
F(x,y,z,s)   &=& x^2 + a(s) T' \cdot  \prod_i T_i^{e_i} \cdot  R_1(T'y,T'z,s) \cdot  y^3 \\
           && +  a(s)(T')^2 \cdot \prod_i T_i^{e_i}  \cdot R_2(T'y,T'z,s)  \\
           && \cdot y \Big(   R_3(T'y,T'z,s) z\cdot T' +  \mu_1(s) y  +  T'z \cdot \mu_2(T'z,s)y\Big)^3.          
   \end{eqnarray*}
The last line in the previous equation is equal to 
\[
y \Big(  (\mathrm{unit})\cdot z\cdot T' +  \mu_1(s) y   \Big)^3.  
\]
Hence we can write 
\begin{eqnarray*}
F(x,y,z,s)   &=& x^2 + a_1(s)  \cdot  Q_1(y,z,s) \cdot  y^3 \\
           && +  a_2(s)  \cdot Q_2(y,z,s) \cdot y\Big(  Q_3(y,z,s) z\cdot T' +  \mu_1(s) y\Big)^3,       
   \end{eqnarray*}
where $Q_1,Q_2,Q_3$ are units and the zero loci of $a_1$ and $a_2$ are both contained in $H$. 
We let $Q_4 = Q_2 \cdot Q_1^{-1}$. 
Then there is a unit $Q_5$ such that $Q_4^3 Q_5 = Q_1$ and $Q_4^4 Q_5 = Q_2$.  
We can write 
\begin{eqnarray*}
F(x,y,z,s)   &=& x^2 + a_1(s) \cdot  Q_5\cdot  (Q_4 \cdot  y)^3 \\
           && + a_2(s) \cdot Q_5  \cdot (Q_4 \cdot y) \cdot (  Q_3Q_4\cdot z\cdot T' +  \mu_1(s) Q_4 \cdot y)^3  \\
            &=& Q_5 \cdot \Big( 
             Q_5^{-1} x^2 + a_1(s) \cdot   (Q_4 \cdot  y)^3 \\
           && +  a_2(s)  \cdot    (Q_4 \cdot y)\cdot (   Q_3Q_4\cdot z\cdot T' +  \mu_1(s) Q_4 \cdot y)^3 
            \Big). 
   \end{eqnarray*}
Hence, up to a change of coordinates,   $X$ can be defined by an equation of the shape 
\[
 x^2  + \alpha(s) y^3 +  \beta(s) y( \gamma(s) z + u(y,s)y)^3  = 0,
\]
where  $\alpha, \beta$ and $\gamma$ are holomorphic functions on $S$, whose zero loci are contained in $H$. 
In the end, we note that if we continue blowing up $X$ at the $H_i$'s, 
then the defining equation remains the same shape. 
This completes the proof of the lemma.
\end{proof}

\subsection{Case of $D$-type singularities}
In this subsection, we  treat Proposition \ref{prop:good-shape} in the case when $X$ has $D$-type singularities at points of $S\setminus H$.

\begin{lemma}
  \label{lemma:blowup-equation-D} 
Proposition \ref{prop:good-shape} holds if $X$ has $D$-type singularities at points of $S\setminus H$. 
\end{lemma}

\begin{proof} 
We can assume that $X$ is a hypersurface in $\mathbb{D}^3\times S$ defined by the following equation  of standard form with respect to $(S,H)$,   
    \[ F(x,y,z,s) = x^2 + F_2(y,z,s) + F_3(y,z,s)+R(y,z,s) = 0, \]
where $F_2$ and $F_3$ are homogeneous in $(y,z)$ of degree 2 and 3 respectively, 
and $R$ is a holomorphic function  such that  $R(y,z,s) = 0 \mod\, (y,z)^4$.   
By Lemma \ref{lemma:ADE}, we see that  $F_2 = 0$.  
Since $F$ is of standard form, up to shrinking $X$ and up to a linear change of the coordinates $(y,z)$, 
we may assume that 
\[F_3(y,z,s) = a(s)(y^3 + b(s)y^2z + c(s)yz^2 + d(s)z^3)\] 
such that  the zero locus of $a(s)$ is contained in $H$.  
We  blow up $X$ at each $H_i$ once (in arbitrary order), the shape  $F_3$ is unchanged.

By Lemma \ref{lemma:ADE},   $F_3(y,z,s)$ has either two or three distinct factors,  for any $s\in S\setminus H$. 
We  apply Lemma \ref{lemma:cubic}   to the following cubic polynomial in $\frac{y}{z}$, 
\[
a(s)^{-1} F_3(\frac{y}{z},1,s) = (\frac{y}{z})^3 + b(s)\cdot (\frac{y}{z})^2z + c(s)\cdot \frac{y}{z} + d(s). 
\]   
Let  $X_2\to X$ be the basechange  over $\overline{S}=S[\sqrt[6]{T_1}, ..., \sqrt[6]{T_n}] \to S$, let $\overline{S}$  be the preimage of $S$ in $X_2$, and let $\overline{H}_i$ be the preimage of $H_i$ for every $i=1,...,n$. 
In the remainder of the proof, by abuse of notation, we will write $s\in S$ for $\overline{s} \in \overline{S}$, $H=\sum_{i=1}^n H_i$ for $\overline{H}=\sum_{i=1}^n\overline{H}_i$,  $X$ for a neighborhood   of $\overline{S}$ in  $X_2$, 
and $F$ for a defining function of $X$. 
Then  Lemma \ref{lemma:cubic} implies  that 
\[
F_3(y,z,s)= a(s)(y-p(s)z)(y-q(s)z)(y-r(s)z)  
\]
for holomorphic functions $p,q,r$ on $S$. 
\\

\textit{Case 1.} 
We first assume that  $X$ has $D_{l}$-type singularities at points of $S\setminus H$ with $l\ge 5$. 
by Lemma \ref{lemma:ADE}, we may assume that $q(s)=r(s)$. 
Up to replacing $y$ by $ y+p(s)z$, we can write 
\[ F_3(y,z,s)=a(s)y(y-q(s)z)^2, \] 
such that the zero locus of $q$ is contained in $H$.  
In particular, it is of the shape 
\[
q(s)= T_1^{l_1}\cdots T_n^{ln} \cdot \overline{q}(s), 
\]
where $\overline{q}(s)$ is a unit. 
Up to replacing $z$ be $\overline{q}(s)^{-1} \cdot z$, we may assume that 
\[
q(s)= T_1^{l_1}\cdots T_n^{l_n}. 
\]

By blowing up  $X$ at the  $H_i$'s for several times as in Subsection \ref{subsection:blowup-divisor}, 
we may assume that $a(s)q(s)^4$ divides $R(y,z,s)$, 
see Remark \ref{rmk:blowup-divisor-1}.  
Hence we can write  $R( y,z,s) = a(s)q(s)^4Q(y,z,s)$ for some holomorphic function $Q$ with $Q(y,z,s) = 0 \mod\, (y,z)^4$.   
Then we can write 
\[
F(x,y,z,s) = x^2 + a(s)\Big( y(y^2-2q(s)zy + q(s)^2 z^2) + q(s)^4 \cdot Q(y,z,s)    \Big)
\]
By applying  Lemma \ref{lemma:Hensel} to the function in the parenthesis after $a(s)$ above,  
the defining function of $X$ can be written as 
\[ x^2 +   a(s)\cdot \Big(y +  q(s)^2  w(y,z,s)\Big)\Big( q(s)^2z^2 -2q(s)yz + y^2 + q(s)^2  v(y,z,s)\Big),\]
where   $w(y,z,s)$  is divisible by $z^2$ and   $v(y,z,s) =0 \mod \, (y,z)^3$. 
We replace  $y+q(s)^2 w(y,z,s)$ by $y$, then the second summand  above becomes 
\begin{eqnarray*}
   G(y,z,s) &=&   a(s) \cdot y \cdot \Big( R_1(y,z,s)  \cdot q(s)^2 z^2 \\ 
            &&  -2 R_2(y,s)\cdot  q(s)  zy    + y^2 + q(s)^2 R_3(y,s)  \Big), 
\end{eqnarray*}
where $R_1$ is a unit and  $R_3(y,z) = 0 \mod \,  (y)^3$. 
Therefore, by blowing up $X$ at the $H_i's$ for several times as in Subsection \ref{subsection:blowup-divisor},  we may assume that 
\begin{eqnarray*}
     G(y,z,s)  
  &=&   a(s)  y \cdot \Big( R_4(y,q(s) z,s) \cdot q(s)^2 z^2 \\ 
  && -2 R_2(y,s)\cdot  q(s)  zy    + y^2 + q(s)^2 R_3(y,s)  \Big), 
\end{eqnarray*}
where $R_4$ is a unit, see Remark \ref{rmk:blowup-divisor-2}. 

If we set $\zeta = q(s)z$, then by using Weierstrass preparation theorem with respect to the variable $\zeta$ for the function in the parenthesis after $a(s) y$ above, we can write 
\[
G(y,z,s) =  (\mathrm{unit}) \cdot  a(s)y  \cdot\left( \zeta^2+ 2R_5(y,s) \zeta y + R_6(y,s) \right), 
\]
where $R_6 = 0 \mod \, (y)^2$.  
It follows that we  can write 
\[
G(y,z,s) =  (\mathrm{unit}) \cdot a(s)  \cdot  \Big(  y \cdot (q(s)z + R_5(y,s) \cdot y )^2 +  R_7(y,s)\Big), 
\]
where   $R_7(y,s ) = 0 \mod \, (y)^3$. 

Since $X$ has  $D_{l}$-type singularities at points of $S\setminus H$ with $l\ge 5$, 
we deduce that 
\[ R_7(y,s) = b(s)y^{l-1} + y^l \cdot c(y,s), \]
where the zero locus of $b$ is contained in $H$. 
By blowing up $X$ at  the $H_i's$  for several times as in Subsection \ref{subsection:blowup-divisor}, 
the function $R_7(y,s)$ can be written as 
\[R_7(y,s)  = b(s) \cdot   R_8(y,s) \cdot y^{l-1}, \] 
where $R_8$ is a unit, see Remark \ref{rmk:blowup-divisor-1}.  
Hence the defining function of $X$ becomes 
\[
F(x,y,z,s)=x^2+ a(s) \cdot (\mathrm{unit})\cdot \Big( y (q(s)z + R_5(y,s)y)^2 +  b(s) \cdot R_8(y,s) \cdot y^{l-1} \Big)
\] 
Let  $R_{9}(y,z)$ be a  $(l-2)$-root of $R_{8}(y,s)$. Then we have 
\begin{eqnarray*}
F(x,y,z,s)& =& x^2+ a(s) \cdot (\mathrm{unit})\cdot \Big( y (  q(s)z + R_5 \cdot  y)^2 \\ 
            &&  +  b(s) \cdot R_9^{l-2}  \cdot y^{l-1} \Big)    \\
          & = &  x^2+ a(s) \cdot (\mathrm{unit})\cdot R_{9}^{-1} \cdot \Big( R_{9} \cdot  y  \cdot (  q(s)z + R_5 \cdot  y)^2  \\ 
             &&  +  b(s) \cdot (R_{9} \cdot  y)^{l-1} \Big).
\end{eqnarray*}
Up to replacing $y$ by $R_{9}(y,s)^{-1}y$, 
the equation $F(x,y,z,s)=0$ is equivalent to 
\[
(\mathrm{unit})\cdot  x^2+ a(s)  \Big( y ( q(s)z - R_{10}(y,s)y)^2 +  b(s) \cdot   y^{l-1} \Big) = 0.   
\]
Hence, up to a change of coordinates, $X$ is now defined by 
\[
 x^2 + \alpha(s) \cdot  y  (\gamma(s)z + u(y,s)y )^2  + \beta(s) y^{l-1} =0,   
\] 
as in the statement of  Proposition \ref{prop:good-shape}. 
\\

\textit{Case 2.}  
Now we assume that $X$ has $D_{4}$-type singularities at points of $S\setminus H$.  
By Lemma \ref{lemma:ADE}, and by 
replacing $y$ by $ y+p(s)z$, we can write 
\[ F_3(y,z,s)=a(s)y(y-q(s)z)(y-r(s)z), \]  
such that the zero locus of $q$ and $r$  are contained in $H$, and that the zero locus of $q(s)-r(s)$ is also contained in $H$.  
Hence, up to switching $q$ and $r$,  there are  unit functions $\overline{q}$  and  $\overline{r}$  such that.  
\[
q(s)= T_1^{l_1}\cdots T_n^{l_n} \cdot \overline{q}(s) \mbox{ and  } r(s)= T_1^{l_1+k_1}\cdots T_n^{l_n+k_n} \cdot \overline{r}(s)
\]
for some integers $l_1,...,l_n, k_1,...,k_n\ge 0$. 
Let $\overline{v}$ be a square root of  $\overline{q}\overline{r}$. 
Then, by  replacing $z$ by $\overline{v}^{-1}z$, we may assume that $\overline{q}\cdot  \overline{r} = 1$.  
We set $\mu(s) = T_1^{l_1}\cdots T_n^{l_n}$ and $\nu(s)=T_1^{k_1}\cdots T_n^{k_n}$. 
Then we have 
\[q(s)r(s) = T_1^{2l_1+k_1}\cdots T_n^{2l_n+k_n} = \mu(s)^2 \cdot \nu(s).  \]
We write $(y-q(s)z)(y-r(s)z) = y^2 + \xi(s)zy + \eta(s)z^2$. 
That is, 
\[\eta(s) = q(s)r(s) = \mu(s)^2 \cdot \nu(s) \mbox{ and }
\xi(s) = q(s) + r(s) = (\overline{q}(s) + \overline{r}(s)\nu(s)) \cdot \mu(s).
\]

By blowing up  $X$ at the  $H_i's$ for several times as Subsection \ref{subsection:blowup-divisor},  
we may assume that $a(s)\eta(s)^2$ divides $R(y,z,s)$, see Remark \ref{rmk:blowup-divisor-1}.  
Then  we have 
\[
F(x,y,z,s) = x^2 + a(s) \Big( y(y^2 + \xi(s)zy + \eta(s)z^2) + \eta(s)^2 Q(y,z,s)
\Big),
\]
where $Q(y,z,s) = 0 \mod \, (y,z)^4$. 
By applying Lemma \ref{lemma:Hensel} to the function in the parenthesis after $a(s)$ above, 
the defining function of $X$ can be written as 
\[ x^2 +  a(s) \cdot  \Big(y +  \eta(s)  w(y,z,s)\Big)\Big( \eta(s) z^2 -2\xi(s)yz + y^2 + \eta(s)  v(y,z,s)\Big),\]
where   $w(y,z,s)$ is divisible by $z^2$ and  $v(y,z,s) = 0 \mod \, (y,z)^3$. 
We replace $y$  by $y-\eta(s)  w(y,z,s)$,    then the second summand  above becomes 
\begin{eqnarray*}
 G(y,z,s) &=&  a(s)y \cdot  \Big( R_1(y,z,s) \cdot \eta(s) z^2 \\ 
        &&-2   \xi(s)  zy    + R_2(y,z,s) y^2   \Big),    
\end{eqnarray*}
where $R_1$ is a unit.  
Recall that $\eta(s)=\mu(s)^2\nu(s)$ and that $\xi(s)$ is divisible by $\mu(s)$. 
Then we have 
\begin{eqnarray*}
 G(y,z,s) &=&   \frac{1}{\nu(s)} \cdot a(s)y \cdot  \Big( R_1(y,z,s) \cdot \mu(s)^2\nu(s)^2 z^2\\ 
        && -2 \cdot   \frac{\xi(s)}{\mu(s)}  \cdot  \mu(s) \nu(s) z  \cdot y    + \nu(s) R_2(y,z,s) y^2   \Big).   
\end{eqnarray*}
By blowing up $X$ at the $H_i's$ for several times as in Subsection \ref{subsection:blowup-divisor},  
we may assume that the function in the parenthesis after $a(s)y$ above 
is a function of $(y,\mu(s)\nu(s) z,s)$,  see Remark \ref{rmk:blowup-divisor-2}.  
In other words, we can assume that  
\begin{eqnarray*}
    G(y,z,s)    &=&   \frac{1}{\nu(s)} \cdot a(s) y \cdot  \Big( R_3(y,\mu(s)\nu(s) z,s) \cdot \mu(s)^2\nu(s)^2 z^2  \\ 
      & &  +2R_4(s)\cdot \mu(s)\nu(s)   z \cdot y    +  R_5(y,\mu(s)\nu(s) z,s)\nu(s) \cdot y^2  \Big), 
\end{eqnarray*}
where $R_3$ is a unit.   

We remark that if we blow up $X$ at the $H_i'$s for sufficiently many times, we can assume that $a(s)$ is divisible by $\nu(s)$.  
Indeed, if we blow up $X$ at $H_i$, then $a(s)$ becomes $T_i\cdot a(s)$ but $\nu(s)$ remains the same. 
By abuse of notation, we   denote $\frac{1}{\nu(s)} \cdot a(s)$ by $a(s)$ in the following computations.

If we set $\zeta = \mu(s)\nu(s)z$, then by applying Weierstrass preparation theorem with respect to the variable $\zeta$ to the function in the parenthesis after $a(s)y$ above, we can write 
\[
G(y,z,s) =   (\mathrm{unit}) \cdot a(s)y \cdot \Big( \zeta^2+ 2R_6(y,s) \zeta y + R_7(y,s) \Big), 
\]
where $R_7  = 0 \mod \, (y)^2$. 
It follows that we can write 
\[
G(y,z,s) =  (\mathrm{unit}) \cdot a(s) \cdot  \Big( y \cdot  \big(\mu(s)\nu(s) z + R_6(y,s) y \big)  ^2 +    R_{8}(y,s) \Big), 
\]
where $R_{8}(y,z) = 0 \mod \,  (y)^3$.   
Since $X$ has  $D_{4}$-type singularities at points of $S\setminus H$, 
we deduce that 
\[ R_{8}(y,s) = b(s)y^{3} + y^4c(y,s), \]
where the zero locus of $b$ is contained in $H$. 
By blowing up $X$ at  the $H_i$'s for several times as in  Subsection \ref{subsection:blowup-divisor}, 
the function $R_{8}(y,s)$ can be written as 
\[R_{8}(y,s)  = b(s) \cdot    (\mathrm{unit}) \cdot y^{3}, \]  
see Remark \ref{rmk:blowup-divisor-1}.  
Then we can write 
\[
F(x,y,z,s)=x^2 +  (\mathrm{unit}) \cdot a(s) \cdot   y \cdot  \big(\mu(s)\nu(s) z + R_8(y,s) y \big)  ^2 +   (\mathrm{unit}) \cdot a(s) b(s) \cdot y^3. 
\]
Now as in the proof for the Case 1,  
we deduce that   $X$ can be  defined by 
\[
 x^2 + \alpha(s) \cdot  y  (\gamma(s)z + u(y,s)y )^2  + \beta(s) y^{3} =0,   
\] 
as in Proposition \ref{prop:good-shape}. 
In the end, we note that if we continue blowing up $X$ at some $H_i$, then the defining equation remains the same shape. 
This completes the proof of the lemma. 
\end{proof}

\begin{remark}
\label{rmk:finite-equation} 
If we let $\theta\colon X\to S$ be the projection induced by the canonical one from  $\mathbb{D}^3\times S$ to $S$, then we can regard $T_i$ as a holomorphic function on $X_1$  by pulling it back \textit{via} $\theta\circ f_1 \colon X_1 \to S$. 
Then in a neighborhood of the strict transform  $(f_1^{-1})_*S$ of $S$ in $X_1$, the Cartier divisor  defined by  $T_i=0$ is the $f_1$-exceptional divisor defined by $f_1^{-1} \mathcal{I}({H_i}) \cdot \mathcal{O}_{X_1}$, where  $\mathcal{I}({H_i})$ is the ideal sheaf of $H_i$ in $X$, see Remark \ref{rmk:function-exc-div}. 
Furthermore, we have $X_2 = X_1[\sqrt[6]{T_1},..., \sqrt[6]{T_n}]$. 
\end{remark}

\begin{remark}
\label{rmk:finite-blowup}
The objective that we first construct $f_1\colon X_1\to X$  is to make the finite cover $f_2\colon X_2\to X_1$ functorial.   
For  the holomorphic function  $T_i$ on $S$, there is no canonical way to make it a  holomorphic function  on $X$. 
And the Cartier divisors in $X$ defined by $T_i=0$ depend on the choice of coordinates in Setup \ref{setup-local}. 
However, in a neighborhood $V$ of the strict transform  $(f_1^{-1})_*S$ of $S$ in $X_1$, the Cartier divisor  defined by  $T_i=0$ is the $f_1$-exceptional divisor defined by $f_1^{-1}\mathcal{I}({H_i})\cdot \mathcal{O}_{X_1}$. 
In particular, this Cartier divisor is intrinsic. 
To see this more explicitly, we assume that 
\[
\varphi\colon (x,y,z,s) \mapsto (x',y',z',s')
\]
is a change of coordinates, and we regard $x',y',z',s'$ as holomorphic vector-valued functions on $(x,y,z,s)$. 
Then $s' =s + P$ with $ P = 0 \mod\, (x,y,z)$.  
Let $T=T_1\cdots T_n$. 
Then $f_1$ can be written in coordinates as 
\[
f_1\colon (x_1,y_1,z_1,s) \mapsto (T(s)x_1,T(s)y_1,T(s)z_1,s) = (x,y,z,s)
\]
and 
\[
f_1\colon (x_1',y_1',z_1',s') \mapsto (T(s')x_1',T(s')y_1',T(s')z_1',s') = (x',y',z',s')
\]
As a consequence, the change of coordinates $\varphi_1\colon  (x_1,y_1,z_1,s) \mapsto  (x_1',y_1',z_1',s') $ satisfies 
\[
s'(x_1,y_1,z_1,s) = s + P(T(s)x_1,T(s)y_1,T(s)z_1,s)
\]
Since $P(x,y,z,s) = 0 \mod \, (x,y,z)$, there exists some vector-valued holomorphic function $Q$ such that 
$ P(T(s)x_1,T(s)y_1,T(s)z_1,s) = T(s) Q(x_1,y_1,z_1,s)$ and that $Q(x_1,y_1,z_1,s) = 0 \mod \, (x_1,y_1,z_1)$.  
Then 
\[
T_i(x_1',y_1',z_1',s') =   T_i(s') = T_i(s+ T(s) Q(x_1,y_1,z_1,s)). 
\]
It follows that 
\begin{eqnarray*}
    T_i(x_1',y_1',z_1',s') - T_i(x_1,y_1,z_1,s) 
    &=& T_i(s+ T(s) Q(x_1,y_1,z_1,s)) - T_i(s) \\ 
    &=& T(s) \cdot R(x_1,y_1,z_1,s),
\end{eqnarray*} 
such that $R(x_1,y_1,z_1,s) = 0 \mod \, (x_1,y_1,z_1)$.  
Since $T_i(s)$ divides $T(s)$,   up to shrinking $V$,   we have 
\[ T_i(x_1',y_1',z_1',s') = (\mathrm{unit})\cdot  T_i(x_1,y_1,z_1,s).\] 
As a result, by Lemma \ref{lemma:cyclic-cover-unit}, 
the finite morphism $f_2\colon X_2\to X_1$ constructed \textit{via} two choices of coordinates 
in Setup \ref{setup-local} are the same  on $V$.  
\end{remark}

\section{Local construction of covering spaces}  
\label{section:local-cover}

In this section, we will construct a finite Galois cover for  a  hypersurface $X$ defined by some equation  as in Proposition \ref{prop:good-shape}, so that the covering space is smooth over $S\setminus H$.

We   fix the following convention for this section.  
Assume that  $M$ is a path-connected topological space and  $m\in M$ is a point. 
We recall that  $\pi_1(M,m)$, the fundamental group of $M$ with basepoint $m$, 
 is the group of homotopy classes of loops in $M$ with basepoint $m$. 
The fundamental group of $M$, $\pi_1(M)$,  is the Galois group of the universal cover of $M$. 
There is an  isomorphism from  $\pi_1(M,m)$ to  $\pi_1(M)$. 
We will work with   fundamental groups without basepoint in this section. 
By Galois theory, there is a correspondence between normal subgroups of  $\pi_1(M)$ and isomorphism classes of Galois covers over $M$. 

If $\overline{m}\in M$ is another point, then there are isomorphisms $\pi_1(M,m)\cong \pi_1(M,\overline{m})$ induced by the paths joining $m$ and $\overline{m}$. 
These isomorphisms are in bijection with the inner automorphisms of  $\pi_1(M,\overline{m}). $
Now we assume that $f\colon M \to M'$ is a continuous map between path-connected topological spaces. If $m'=f(m)$, then there is a natural morphism  
\[
f_{m,*}\colon \pi_1(M,m) \to \pi_1(M',m'),  
\]
which induces a morphism of fundamental groups $\pi_1(M) \to \pi_1(M')$. 
The conjugacy class of  $\pi_1(M) \to \pi_1(M')$ is independent of the choice of basepoints. 
By abuse of notation, we will call this conjugacy class the natural morphism induced by $f$, and denote it by 
\[f_*\colon  \pi_1(M) \to \pi_1(M').\]
By the image of $f_*$, we refer to the conjugacy class of subgroups in $\pi_1(M')$ induced by $f_{m,*}(\pi_1(M,m))$. 
When it is the conjugacy class of a normal subgroup, then the quotient of $\pi_1(M')$ by the image of $f_*$ is well-defined.  
We say that $f_*$ is surjective if its image   is the conjugacy class of $\pi_1(M')$. 
For any normal subgroup $H$ of $\pi_1(M')$, its preimage in $\pi_1(M)$ is a well-defined normal subgroup, and we denote it by $(f_*)^{-1}(H)$. 
In particular, we can define the kernel of $f_*$. 
We say that $f_*$ is  injective if its kernel is the trivial subgroup of $\pi_1(M)$, and that $f_*$ is   an isomorphism if it is both surjective and injective.

\subsection{Preparatory lemmas}

In this part, we will prove some properties on varieties of the shape $Y \times \mathbb{D}^n$, where $Y$ is a Du Val singularity.   
We start with two elementary results.

\begin{lemma}
\label{lemma:homotopy-map} 
Let $U\subseteq \mathbb{C}^m$ and $S\subseteq \mathbb{C}^n$ be open neighborhoods of the origins, 
and let  $f\colon U \times S \to S$ be a morphism  such that 
$f  (\mathbf{0}_m,\mathbf{0}_n) = \mathbf{0}_n.$  
Then there are open neighborhoods  $W\subseteq U$ and  $T\subseteq S$ of the origins, such that
$f|_{W\times T}$ is  a continuous deformation  equivalence to the projection $p$ onto $T$, with values contained in $S$.  
In other words, there is a continuous function 
\[H  (u,s,\mu) \colon W\times T \times [0,1]  \to S, \]
such that $H(u, s, 0) = s$ and $H(u, s, 1) = f(u,s)$.
\end{lemma}

\begin{proof} 
There is a real number $r>0$ such that 
$S$ contains  the ball in $\mathbb{C}^n$ of radius $2r>0$ centered at the origin. 
We can define the function $g(u,s) = f(u,s) - s$ on $U\times S$, 
which takes values in $\mathbb{C}^n$. 
Then the assumption implies that $g (\mathbf{0}_m,\mathbf{0}_n) = \mathbf{0}_n.$  
Thus there are open neighborhoods  $W\subseteq U$ and  $T\subseteq S$ of the origins, 
such that $|g(u,s)| <  r$ for $(u,s)\in W\times T$. 
We can choose $T$ so that it is contained in the ball in $\mathbb{C}^n$ of radius $r$ centered at the origin. 
We define the function   $H  (u,s,\mu) = s + \mu g(u,s)$ on $W\times T \times [0,1]$. 
Then $H$ has values in $S$, 
and it induces a deformation equivalence from $f|_{W\times T}$ to $p$. 
This completes the proof of the lemma. 
\end{proof}

\begin{lemma}
\label{lemma:same-image-fund} 
Let $M$ and $M'$ be  path-connected topological spaces, 
and let $H(m,\mu)$ be a continuous map from $M\times [0,1]$ to $M'$. 
If we denote by $\varphi_0$ and $\varphi_1$ the restriction of $H$ on $M\times \{0\}$ and 
on $M\times \{1\}$ respectively, 
then the morphisms of fundamental groups 
\[(\varphi_0)_*, (\varphi_1)_* \colon  \pi_1(M) \to \pi_1(M') \]
have  the same image.  
\end{lemma}

\begin{proof}
For $i=0,1$, 
we let $\iota_i \colon M \to M\times [0,1]$ be the continuous map 
sending a point $m\in M$ to $(m,i)$. 
Then both $\iota_0$ and $\iota_1$ induce an isomorphism from $\pi_1(M)$ to $\pi_1(M\times I)$. 
Since $\varphi_i = H\circ \iota_i$ for $i=0,1$, the images of $(\varphi_0)_*$ and $ (\varphi_1)_*$ are both equal to the image of $H_*\colon  \pi_1(M\times I) \to \pi_1(M')$.
\end{proof}

\begin{lemma}
\label{lemma:construct-contract-tubular} 
Let $Y\subseteq \mathbb{C}^3$ be the hypersurface  defined by   $F(x,y,z)=0$, such that $F$ is  one of the following functions, 
\begin{eqnarray*}
    F(x,y,z) &=& x^2 +  y^2 +  z^l,   \  l \ge 2, \\
    F(x,y,z) &=& x^2 +  yz^2  +   y^l,   \  l \ge 3,  \\
    F(x,y,z) &=& x^2 +  y^3 +  yz^3, \\
    F(x,y,z) &=& x^2 +  y^3 +  z^l,   l=4,5,
\end{eqnarray*} 
where  $l$ is a positive integer.   
Let $\check{S}$ be a contractible open neighborhood of the origin in $\mathbb{C}^n$ and let  $X=Y\times \check{S}$.  
Assume that $(T_1,...,T_n)$ is a coordinates system of $\mathbb{C}^n$. 
Let $T = T_1\cdots T_k$ for some integer $0\le k \le n$, and let $H=\{T=0\}$. 
We set $ \check{S}^\circ = \check{S}\setminus H$. 
Then for any positive integer $N_0$, there exists positive integers  $p,q,r \ge N_0$ such that the morphism 
\[
\varphi \colon (x,y,z,s ) \mapsto (T^px,T^qy, T^r z, s)
\]
induces an automorphism of $Y \times  \check{S}^\circ$   over $ \check{S}^\circ$. 
\end{lemma}

\begin{proof}
We only treat the case when  $Y\subseteq \mathbb{C}^3$ is defined by 
\[0=F(x,y,z) = x^2 +  yz^2  +   y^l. \]
The other cases are similar. 
In order that  the map $\varphi$ is well-defined from  $Y \times  \check{S}^\circ$ to itself, 
it is sufficient that there is some positive integer $d$, so that 
\[
(T^px)^2 +  (T^qy)(T^rz)^2  +   (T^qy)^l = T^d (x^2 +  yz^2  +   y^l).
\]
It is equivalent to the system 
\[
2p = q+2r = lq. 
\]
There three unknowns  and two equations, hence the system is indeterminate and admits infinitely many integer solutions.  
Furthermore, since $l\ge 2$, we see that the solutions for $q,r$ tend to $+\infty$ if $p$ tends to $+\infty$.   
Hence there are integers $p,q,r \ge N_0$ such that $\varphi$ is well-defined. 
Since $T$ is nowhere vanishing  on $ \check{S}^\circ$, we deduce that $\varphi$ is an automorphism of $Y \times  \check{S}^\circ$  over $ \check{S}^\circ$.  
\end{proof} 

With the notation in the  previous lemma, we note that there is a finite subgroup $\Gamma$ of $\mathrm{SL}_2(\mathbb{C})$ such that $Y= \mathbb{C}^2/\Gamma$, see \cite[Section II.8]{Lamotke1986}. 
In particular, for any open subset $U\subseteq Y$ containing $\mathbf{0}_3$, 
there  is an open subset $W \subseteq U$ such that the natural morphism from 
$\pi_1(W_{\sm})$ to $\pi_1(Y_{\sm})$ is isomorphic.

\begin{lemma}
\label{lemma:tubular-fund-group-iso} 
With the notation in Lemma \ref{lemma:construct-contract-tubular}, 
Let $U\subseteq Y$ be an open neighborhood of  $\mathbf{0}_3$ such that the natural morphism of fundamental groups  $\pi_1(U_{\sm}) \to \pi_1(Y_{\sm})$ is an isomorphism.  
Let $S\subseteq \check{S}$ be an  open subset, 
such that $S$ is also isomorphic to a contractible open neighborhood of the origin in $\mathbb{C}^n$, and that $H|_S$ is isomorphic to the union of   coordinates hyperplanes. 
Let $S^\circ = S\setminus H$. 

Let $\varphi \colon U\times S^\circ \to Y\times \check{S}^\circ$ be an open embedding, 
whose restriction on $\{\mathbf{0}_3\} \times S^\circ$ is the identity map. 
Then  the induced morphism $\varphi_*$ of fundamental groups from  $\pi_1(U_{\sm}\times S^\circ) $ to $ \pi_1(Y_{\sm}\times  \check{S}^\circ)$ is injective. 
Furthermore, its image is equal to the image of $\pi_1(Y_{\sm}\times S^\circ)$
induced by the natural inclusion. 
In particular, if $k'$ is the number of irreducible components of $H|_S$, 
then we have 
\[
 \pi_1(Y_{\sm}\times  \check{S}^\circ) / \varphi_*( \pi_1(U_{\sm}\times S^\circ) ) \cong \mathbb{Z}^{k-k'}. 
\]
\end{lemma}

\begin{proof} 
Let $o\in S^\circ $, let $u\in U_{\sm}$ and let $V = \{u\}\times S^\circ \subseteq U_{\sm}\times S^\circ$.     
Then we can write  $\pi_1(U_{\sm}\times S^\circ) = H_1 \times G_1$,  where 
\[G_1 = \pi_1(V) \cong \pi_1(S^\circ)\cong \mathbb{Z}^{k'}, \ H_1 = \pi_1(U_{\sm} \times \{o\})\cong\pi_1(U_{\sm}).  
\]  
Similarly, we can write $\pi_1(Y_{\sm}\times S^\circ) = H_2 \times G_2$, where 
\[
G_2\cong \pi_1( \check{S}^\circ) \cong \mathbb{Z}^k  \mbox{ and } H_2 = \pi_1(Y_{\sm} \times \{o\}) \cong  \pi_1(Y_{\sm}). 
\]

Since $Y$ is a hypersurface in $\mathbb{C}^3$, locally around $(\mathbf{0}_3,o) \in U\times S^\circ$, the morphism $\varphi$ can be written in coordinates as 
\[
\varphi\colon (x,y,z, s) \mapsto (x',y',z',s')
\]
and we can regard $x',y',z',s'$ as vector-valued functions in $(x,y,z,s)$. 
The assumption implies that $s'(0,0,0,s) = s$.  
We let    
\[\psi \colon (x,y,z, s) \mapsto (x,y,z,s')  \]
be a  morphism defined in neighborhoods of $(\mathbf{0}_3,o) \in U\times S^\circ$. 
By the implicit function theorem, there is  an open neighborhood  $\overline{W}$ of $\mathbf{0}_3$ in $U$, and   an open neighborhood $\overline{T}$ of $o\in S^\circ$,  
such that  $\psi$ is an isomorphism from $\overline{W}\times \overline{T}$ to its image $Z$  inside $U\times  \check{S}^\circ$.  
We note that $Z$ is an open neighborhood of $(\mathbf{0}_3,o)$ in $U\times  \check{S}^\circ$.

By Lemma \ref{lemma:homotopy-map}, 
there is an open neighborhood  $W$ of $\mathbf{0}_3$ in $\overline{W}$, 
there are open neighborhoods $T\subseteq T'$ of $o$ in $\overline{T}$,  
such that  the restriction of $s'$ on $W\times T$ is deformation equivalent to the projection function $(x,y,z,s) \mapsto s$, with values inside $T'$. 
Up to shrinking $W$ and $T'$, we may assume that $W\times T' \subseteq  Z$. 
As a consequence, the image $ \psi(W_{sm}\times \{o\})$ is deformation equivalent to $W_{\sm}\times \{o\}$ inside $Z_{\sm} 
= Z\setminus (\{ \mathbf{0}_3 \} \times S^\circ )$.  
Then by Lemma \ref{lemma:same-image-fund},  the image of 
\[
\psi_* \colon \pi_1 (W_{\sm}\times \{o\}) \to \pi_1( Z_{\sm} )
\]
is equal to the one of $\pi_1(W_{\sm}\times \{o\})$ induced by the natural  inclusion 
\[W_{\sm}\times \{o\} \to Z_{\sm}.\]    
Moreover, we can choose $W$ so that the natural morphism 
$\pi_1(W_{\sm}) \to \pi_1(U_{\sm})$ is isomorphic.

Let  $\eta \colon Z \to Y\times  \check{S}^\circ$ be the morphism defined by $ \varphi\circ \psi^{-1}$, 
which is an open embedding.   
Then  $\eta$ can be written in coordinates as 
\[
\eta \colon (x,y,z, s') \mapsto (x',y',z',s'). 
\]
We see that the image $\eta(W \times \{o\})$ is a neighborhood of $(\mathbf{0}_3 ,o)$ in the fiber $Y\times \{o\}$.  
Since $\pi_1(Y_{\sm})$ can be  generated by loops  in an arbitrarily small neighborhood of $\mathbf{0}_3\in Y$,  
the morphism $\eta_*$  of fundamental groups from $\pi_1(W_{\sm}\times \{o\})$ to 
$H_2 = \pi_1( Y_{\sm}\times \{o\} )$ is surjective. 
Since both  groups  are finite   of the same order, we deduce that 
\[\eta_* \colon  \pi_1(W_{\sm}  \times \{o\} ) \to H_2 \] 
is an isomorphism.   
Since $H_2$ is a normal subgroup of $\pi_1( Y_{\sm}\times  \check{S}^\circ )$, 
we conclude with the previous paragraph that 
\[
\varphi_* = (\eta\circ \psi)_* \colon \pi_1(W_{\sm}\times \{o\}) \to \pi_1( Y_{\sm}\times  \check{S}^\circ )
\] 
induces an isomorphism from $\pi_1(W_{\sm}\times \{o\})$ to $H_2$. 
This implies that the morphism  $\varphi_*  \colon  \pi_1(U_{\sm}\times S^\circ) \to  \pi_1(Y_{\sm}\times  \check{S}^\circ)$
induces an isomorphism from $H_1$ to $H_2$.

Since $u$ is  connected to $\mathbf{0}_3$ by path in $U$, we see that $V$ is deformation equivalent  to $\{\mathbf{0}_3\} \times S^\circ$ inside $U\times S^\circ$.  
Thus  $\varphi(V)$ is  deformation equivalent  to $\{\mathbf{0}_3\} \times S^\circ$ inside $Y\times  \check{S}^\circ$. 
Hence by Lemma \ref{lemma:same-image-fund}, 
if $q\colon Y\times  \check{S}^\circ \to  \check{S}^\circ$ is the natural projection, 
then  $q \circ \varphi$ induces an isomorphism from $\pi_1(V)$ to 
the image of $\pi_1(S^\circ)$ in   $\pi_1( \check{S}^\circ)$.    
This implies that $\varphi_*$ induces an injective map on $G_1$. 
Since $G_1$ is a free Abelian group, we deduce that the intersection $\varphi_*(G_1)\cap \varphi_*(H_1)$ is trivial. 
Hence $\varphi_*$ is   injective. 
Finally, we notice that the image of $\varphi_*$ is equal to the preiamge of $\pi_1(S^\circ)$ under the morphism 
\[
q_*\colon \pi_1(Y_{\sm}\times  \check{S}^\circ) \to \pi_1(  \check{S}^\circ), 
\]
which is equal  to the image of   $ \pi_1(Y_{\sm}\times S^\circ)$ inside  $ \pi_1(Y_{\sm}\times  \check{S}^\circ)$. 
This completes the proof of the lemma.
\end{proof}

With the notation above, we remark that the fundamental group $\pi_1(Y_{\sm}\times  \check{S}^\circ)$ is isomorphic to $\pi_1(Y_{\sm}) \times \mathbb{Z}^k$. 
For this kind of groups, we construct a canonical subgroup (which is $N$) in the following lemma.  
After the Galois theory, it induces a canonical covering space of $Y_{\sm}\times  \check{S}^\circ$.

\begin{lemma}
\label{lemma:standard-subgroup}
Let $E=H\times G$ be a group, where $H\subseteq E$ be a finite subgroup  
and  $G\subseteq E$ is a subgroup isomorphic  to $\mathbb{Z}^k$ for some integer $k\ge 0$.   
Let $N$ be the intersection of all complements of $H$ in $E$. 
Then $N$ is of finite index in $E$, and is invariant under any automorphism  of $E$. 
Furthermore,  let  $E'\subseteq E$ be a normal  subgroup such that   $E/E' \cong \mathbb{Z}^l$ for some integer $0\le l \le k$.  
Then $H\subseteq E$.  
If $N'$ is  the intersection of all complements of $H$ in $E'$, 
then $N\cap E' = N'$. 
\end{lemma}

\begin{proof}
Since $E$ is finitely generated, for any fixed number $i$, 
there are only finitely many subgroups of $E$ of index $i$. 
Hence there are only finitely many complements of $H$ in $E$. 
It follows that $N$ is of finite index in $E$. 
Since $H$ is equal to the set of all elements of finite order in $E$, 
it is invariant  under any automorphisms of $E$. 
Hence so is $N$. 

For the second part of the lemma,  
since $E/E'$ is a free Abelian group and $H$ is finite, 
we must have $H\subseteq E'$. 
On the one hand, we assume that $K$ is a complement  of $H$ in $E$.  
Since $H\subseteq E'$, we deduce that  $K\cap E'$ is a complement  of $H$ in $E'$.  
Hence $N'\subseteq N\cap E'$. 

On the other hand, the quotient map $E\to E/E'$ admits a splitting $E/E' \to E$. 
Let $J\subseteq E$ be its image. 
Then $E=JE'$ and $J\cap E' = \{e\}$, where $e$ is the neutral element of $E$. 
Let $K'$ be   a complement  of $H$ in $E'$.  
Then $K:= JK'$ is a complement  of $H$ in $E$. 
Furthermore, $K\cap E' = K'$. 
It follows that $N\cap E' \subseteq N'$. 
This completes the proof of the lemma. 
\end{proof}

\subsection{Construction of covering spaces}

In the remainder of this section, we work with the following setup. 
Our goal is to prove Proposition \ref{prop:construct-local-cover}. 

\begin{setup}
\label{setup-local-cover} 
We fix the following notation for the remainder of this section. 
\begin{enumerate}
    \item We consider the data $X$, $\Delta$, $\check{S}$,  $H$ and  $G$ introduced as follows.  
Let $X$ be a complex analytic variety and let $\check{S}\subseteq X_{\sing}$ be an irreducible component of codimension $2$.   
Let $\Delta$ be a reduced divisor in $X$. 
There is a finite group $G$ acting on $X$ such that $\check{S}$ and $\Delta$ are $G$-invariant. 
The reduced pair $(\check{S}, H)$ is snc. 
If $H_1,\cdots, H_k$  are the irreducible components of $H$, then each $H_i$ is $G$-invariant.
Assume that $X$ is an open subset of  the hyperplane $M$ in $\mathbb{D}^3\times \check{S}$ defined by an equation   $F(x,y,z,s)=0$, where $(x,y,z)$ are coordinates of $\mathbb{D}^3$ and $s\in S$ is a point.  
Assume that $F$ is of one of the following shapes (as in Proposition \ref{prop:good-shape}),   
\begin{eqnarray*}
    F(x,y,z,s) &=& x^2 + \alpha(s)y^2 + \beta(s) z^l,   \ l \ge 2,    \\
    F(x,y,z,s) &=& x^2 + \alpha(s) y(\gamma(s)z +u(y,s) y )^2  +  \beta(s) y^l,   \ l\ge 3,  \\
    F(x,y,z,s) &=& x^2 + \alpha(s)y^3 +  \beta(s) \cdot y(\gamma(s)z +u(y,s) y)^3, \\
    F(x,y,z,s) &=& x^2 + \alpha(s)y^3 +  \beta(s) \cdot  (\gamma(s)z +u(y,s) y)^l, \  l =4,5, 
\end{eqnarray*} 
where  $l$ is an integer,  $\alpha(s)$, $\beta(s) $ and $\gamma(s)$ are holomorphic functions on $\check{S}$,    whose zero loci are contained in $H$, 
and $u(y,s)$ is a holomorphic function defined on $\mathbb{D} \times \check{S}$.    
The composite inclusion  $\check{S}\subseteq X \subseteq \mathbb{D}^3\times \check{S}$ is identified with $\{\mathbf{0}_3\}\times \check{S} \subseteq \mathbb{D}^3\times \check{S}$.  
Furthermore, if  $\theta\colon M\to \check{S}$ is the  projection sending $(x,y,z,s)$ to $s$, then $\Delta= \theta^{-1}(H)\cap X$. 
Let $\check{S}^\circ = \check{S}\setminus H$, $M^\circ =  \theta^{-1}(\check{S}^\circ)$, and  $X^\circ = \theta^{-1}(\check{S}^\circ)\cap X = X\setminus \Delta$.  

\item  The vanishing orders of $\alpha(s)$ and $\beta(s)$ along each $H_i$ are divisible by 6. 

\item  Let $o_S\in \check{S}$ be any point. 
We will work in a neighborhood of $o_S$ in $X$. 
Up to shrinking $X$ around $o_S$,   up to replacing $G$ by the stabilizer of $o_S$, 
and up to reducing the number $k$, 
we assume that $\check{S}$   an  open neighborhood of the origin $\mathbf{0}_n= o_S$ of $\mathbb{C}^n$ 
contained in a  polydisc of radius smaller than 1.     
Since $\check{S}$ is smooth, the action of $G$  on the germ $(o_S\in \check{S})$ is isomorphic to a linear action (see \cite[Lemme 1]{Cartan1954} or \cite[Lemma 9.9]{Kaw88}). 
Therefore,   we may assume that $\check{S}$ is stable under the scaling with any positive factor less than 1. 
In particular, it is contractible. 
We can also assume that there is a coordinates system   $(T_1,...,T_n)$  on $\check{S}$,  
such that  $H_i=\{T_i=0\}$, with $i=1,...,k$.
Let   $T=T_1\cdots T_k$.

\item 
Let $S = \frac{1}{2} \check{S}$ be the scaling of $\check{S}$ with factor $\frac{1}{2}$. 
Then $S$ is a $G$-invariant open neighborhood of $o_S$ and the closure of $S$ in $\mathbb{C}^n$ is contained in $\check{S}$. 
In particular,  there is a disc $\mathbb{D}' \subseteq \mathbb{D}$ such that 
\[
M':=  ((\mathbb{D}')^3\times S) \cap M  \subseteq X \subseteq M.
\]
Furthermore,  every holomorphic function on $\check{S}$ is   bounded on $S$.  
Let $S^\circ = S\setminus H$.

\item 
Let $F'(x,y,z)$ be the polynomial function obtained by replacing $\alpha,\beta, \gamma$ with 1 and replacing $u$ with 0 in the expression of $F$. 
We denote by $Y$  the hypersurface in $\mathbb{C}^3$ defined by $F'=0$. 
Then $Y$ has an isolated Du Val singularity at $\mathbf{0}_3$.  
Let $U \subseteq Y$ be a bounded open  neighborhood of  $\mathbf{0}_3$,  
such that  the natural morphism  of fundamental groups 
$ \pi_1 (U_{\sm}) \to \pi_1(Y_{\sm}) $
is isomorphic. 
\end{enumerate}
\end{setup}

The open subset $M'$ in the item (4) of the previous setup plays an important role in our proof of Proposition \ref{prop:construct-local-cover}. 
The reason that we work with two neighborhood $S$ and $\check{S}$  is following one.  
The projection $\theta\colon X\to \check{S}$ is not $G$-invariant in general.  
As a  consequence, $X$ may not contain an open subset of the shape $ ((\mathbb{D}')^3\times \check{S}) \cap M$.

\begin{prop}
\label{prop:construct-local-cover}
With the notation of Setup \ref{setup-local-cover},  
there is an integer $N\ge 0$,   
such that the following properties hold for any integer  $N'\ge N$. 
Let  $\rho\colon \widehat{X} \to X$ be the  blowup of $X$ at the $H_i$'s as in Subsection \ref{subsection:blowup-divisor}, 
such that there are  $6N'$ blowups whose centers are $H_i$ for all $i$.  
Particularly, there is a natural action of $G$ on $\widehat{X}$ so that $\rho$ is $G$-equivariant. 
Let $\widehat{S}$ be the strict transform of $S$ in $\widehat{X}$, 
let $\widehat{H}\subseteq \widehat{S}$ be the preimage of $H$, 
and let $\widehat{X}^\circ = \rho^{-1}(X^\circ) = \widehat{X} \setminus \rho^{-1}(\Delta)$. 

Then there is an open neighborhood $W$ of $\widehat{S}$ and  
a finite  cover $\nu \colon \overline{W} \to W$, which  satisfy  the following properties.  
\begin{enumerate}
\item $\nu$ is quasi-\'etale  and Galois, over the open subset   $W^\circ := W\cap \widehat{X}^\circ = W\setminus \rho^{-1}(\Delta)$. 
\item $\overline{W}$ is normal.   
\item $\overline{W}$ is smooth over  $W^\circ$.  
\item $W$ is $G$-invariant,  
    and the automorphisms of $W$ over $W/G$    can be lifted to automorphisms of  $\overline{W}$.  
\item The finite morphism $\nu$  is functorial in the following sense. 
For any point $\hat{o}\in W$, let $o=\rho(\hat{o}) \in X$. 
Assume that there is an open neighborhood $W'$ of $\hat{o}$  in $W$,  
such that there is a finite morphism  $\nu' \colon \overline{W}' \to W'$,  
which is  constructed by the same method on 
  a  neighborhood $X'$  of $\rho(W')$ in $X$. 
Then there is an open neighborhood $W''$ of $\hat{o}$ in   $ W'$, 
such that every connected component of $\nu^{-1}(W'')$ is isomorphic to every connected component  $(\nu')^{-1}(W'')$.  
\end{enumerate}
\end{prop}

We will apply this local construction on orbifold charts of some  complex analytic orbispaces.  
The item (5) above is to ensure that the constructions are compatible along the overlaps.  
To explain the assumption on the finite morphism $\nu'\colon \overline{W}'\to W'$ in this item, 
we first illustrate the construction of the finite morphism $\nu\colon \overline{W} \to W$. 
Initially, we construct an open embedding $X^\circ \to Y \times \check{S}^\circ$,  
such that its restriction on $\check{S}^\circ$ is the identity map onto $\{\mathbf{0}_3\}\times \check{S}^\circ$.  
We note that $\check{S} \subseteq \mathbb{C}^n$, and $\check{S}^\circ \subseteq \mathbb{C}^n\setminus H$, where we still denote by $H$ its Zariski closure in $\mathbb{C}^n$, 
which is again the union of some coordinate  hyperplanes.  
Let  $Z\to Y_{\sm} \times \check{S}^\circ$ be the \'etale finite morphism induced by the canonical subgroup in Lemma \ref{lemma:standard-subgroup} of the fundamental group $\pi_1(Y_{\sm} \times \check{S}^\circ)$. 
Then we construct $\overline{W} \to W$ as the unique finite morphism induced by the fiber product $W^\circ_{\sm} \times_{X^\circ} Z$, 
see Theorem \ref{thm:GR-cover}.    
Now for the morphism $\nu \colon \overline{W}' \to W'$ in the item (5),  
the assumption particularly implies that there is an open embedding  
$X'\cap X^\circ \to Y \times (\mathbb{C}^n\setminus H')$,  
where $H'$ is the union of some coordinate  hyperplanes,  
such that $X'\cap S^\circ$ is mapped to  $\{\mathbf{0}_3\} \times (\mathbb{C}^n\setminus H')$.  
The morphism $\nu'$ is  induced by taking the fiber product of $ W'_{\sm}\cap W^\circ $ over the canonical finite morphism $Z'\to  Y_{\sm}\times (\mathbb{C}^n\setminus H')$ as above.

We   will divide the proof of the proposition into a sequence of lemmas. 
The sketch is as follows. 
As explained in the previous paragraph, we first construct an open embedding $X^\circ \to Y \times \check{S}^\circ$. 
We may construct a finite morphism $\overline{X} \to X$ by taking basechange as in the previous paragraph. 
However, since we have little information on the fundamental group of $X^\circ_{\sm}$,  
it is not clear to us that  $\overline{X}$  satisfies the item (4) and (5) of the proposition, 
which signify that $\overline{X}$ is independent of the open embedding $X^\circ \to Y \times \check{S}^\circ$.  
To deal with this problem, 
we will construct explicitly a tubular neighborhood $V$ of $S^\circ$ in $X^\circ$, 
such that $\pi_1(V_{\sm}) \cong \pi_1(Y_{\sm})\times \mathbb{Z}^k \cong \pi_1(Y_{\sm} \times \check{S}^\circ)$. 
Then the finite  cover  $\overline{V}\to V$  constructed by this method 
is independent of the choice of the open embedding from $X^\circ$ to $Y \times \check{S}^\circ$.  
A new issue is that,  $V$ may not extend to an open neighborhood of $S$ in $X$.  
It can collapse  when   approaching to $H\subseteq S$. 
Thus  $\overline{V}$ may not extend  to a covering space over any neighborhood of $S$. 
Our  solution to this is to blow up $X$ at the $H_i$'s as in Subsection \ref{subsection:blowup-divisor}. 
We show that, after some blowups, the preimage of $V$ extends to some open neighborhood $W$ of the strict transform of $S$. 
Then the covering space  $\overline{V}$ extends to a finite cover   $\overline{W} \to W$.

\begin{lemma}
\label{lemma:construct-tubular-1}
With the notation in Proposition \ref{prop:construct-local-cover},   
there is an open embedding, denoted by $\varphi^{-1}$,   from $M^\circ$ to $Y\times \check{S}^\circ$, 
whose inverse $\varphi$ can be written in coordinates as 
\[
\varphi\colon (\overline{x}, \overline{y}, \overline{z},s) \mapsto (a(s)\overline{x}, b(s)\overline{y}, \zeta(\overline{y},\overline{z},s),s) = (x,y,z,s),  
\]
where $a(s),b(s)$ are holomorphic functions on $\check{S}$, and 
\[\zeta(\overline{y},\overline{z},s)  
= \frac{1}{\gamma(s)} \Big( c(s)\overline{z}  
-u (b(s) \overline{y},s ) \cdot b(s)\overline{y}\Big), \]  
for some holomorphic function $c(s)$ on $\check{S}$.  
Moreover, we can impose that the zero loci of $a,b,c$ are contained in $H$,  
and that $b,c$ are  divisible by $\gamma$.
\end{lemma}

\begin{proof}
We treat the case when 
\[F(x,y,z,s)= x^2 + \alpha(s)y^3 +  \beta(s) \cdot y(\gamma(s)z +u(y,s) y)^3. \]
The other cases are similar.  
In order that $\varphi$ is formally well-defined, it is sufficient that, there is a unit function $v(s)$ on $\check{S}^\circ$, such that 
\begin{equation}\label{eqn:compare-isomophism}
    a^2\cdot \overline{x}^2 +  b^3\alpha \cdot \overline{y}^3 + \beta b \cdot \overline{y}(\gamma \cdot \zeta + u(b\overline{y} ,s) \cdot  b\overline{y} )^3 = v \cdot (\overline{x}^2+\overline{y}^3+\overline{y}\cdot \overline{z}^3). 
\end{equation}
The LHS above is equal to 
\[a^2\cdot \overline{x}^2 + \alpha \cdot b^3 \cdot \overline{y}^3 + \beta\cdot bc^3 \cdot \overline{y}\cdot \overline{z}^3.\]
Hence we need to solve 
\[
a^2 =  \alpha\cdot b^3=  \beta \cdot b c^3 = v. 
\]
Since the zero loci of $\alpha$ and $\beta$ are contained in $H$, we can write 
\[
\alpha = T_1^{l_1}\cdots T_k^{l_k} \cdot \alpha'  \mbox{ and }  \beta = T_1^{l'_1}\cdots T_k^{l'_k} \cdot \beta'
\]
where $l_1,...,l_k,l_1',...,l_k' \ge 0$ are integers divisible by $6$ after the item (2) of Setup \ref{setup-local-cover}, 
and $\alpha',\beta'$ are unit holomorphic functions on $\check{S}$. 
Since $\check{S}$ is simply connected, we can take the logarithms of $\alpha'$ and $\beta'$. 
Then we deduce that $a,b,c,v$ exist as meromorphic functions on  $\check{S}$, with poles and zeros contained in $H$.  
For example, we can set  $b=1$, $v= \alpha$, 
\[a  =  T_1^{\frac{1}{2}\cdot l_1}\cdots T_k^{\frac{1}{2}\cdot l_k} \cdot (\alpha')^{\frac{1}{2}} 
\mbox{ and }  c=  T_1^{\frac{1}{3}(l_1-l'_1)}\cdots T_k^{\frac{1}{3}(l_k-l'_k)} \cdot (\alpha')^{\frac{1}{3}} (\beta')^{-\frac{1}{3}}.\]  
We also note that, if $(a,b,c,v)$ satisfies \eqref{eqn:compare-isomophism}, then so does $(T^9 \cdot a,   T^6 \cdot b, T^4 \cdot  c, T^{18} \cdot v)$, see also Lemma \ref{lemma:construct-contract-tubular}. 
Since the zero locus of $\gamma$ is contained in $H$, 
up to multiplying $a,b,c$ by powers of $T$,  
we may assume that they are holomorphic functions and that  $b,c$ are divisible by $\gamma$.

The morphism  $\varphi^{-1}$ from $M^\circ$ to $Y\times \check{S}^\circ$ can be written in coordinates as 
\[
\varphi^{-1}\colon  (x,y,z,s) \mapsto ( a(s)^{-1}x, b(s)^{-1}y, c(s)^{-1}( \gamma(s) z + u(y,s)y ),s ) = (\overline{x}, \overline{y}, \overline{z},s).
\] 
Then $\varphi^{-1}$ is a well-defined morphism on $M^\circ$. 
Since $a,b,c,\gamma$ are nowhere vanishing holomorphic functions on $\check{S}^\circ$, we see that $\varphi^{-1}$ is   locally biholomorphic    
and  is injective. 
Hence  $\varphi^{-1}$ is an open embedding. 
\end{proof}

We remark that the divisibility condition in the item (2) of Setup \ref{setup-local-cover} is to guarantee the existence of  $a,b,c$ in the previous lemma.

\begin{lemma}
\label{lemma:construct-tubular-2}
We fix a morphism $\varphi$  as in Lemma \ref{lemma:construct-tubular-1} so that $\varphi^{-1}$ is an open embedding from $M^\circ$ to $Y\times \check{S}^\circ$.    
Then there are   integers $p,q,r>0$ such that the endomorphism $\eta$ of $Y\times \check{S} $ defined by 
\[
\eta \colon  (\overline{x},\overline{y}, \overline{z},s) \mapsto  (T^p\overline{x}, T^q\overline{y}, T^r\overline{z},s), 
\]
satisfies the following properties. 
The image $\eta (U\times S^\circ)$ is contained in $\varphi^{-1}(M'\cap M^\circ)$.  
In  other words, we have 
\[(\varphi\circ \eta)(U\times S^\circ) \subseteq M' \cap M^\circ \subseteq X^\circ.\]
Furthermore, for any point $s\in S\cap H$ and any neighborhood $X' $ of $s$ in $X$, 
there is an open neighborhood $S' $ of $s$ in $S$, 
such that 
\[
(\varphi \circ \eta)  \big(U \times (S'\cap S^\circ) \big) 
\]
is contained in $X' \cap X^\circ$.   
 \end{lemma}

\begin{proof}
Since $a$, $b$, $c\cdot \gamma^{-1}$ and $b\cdot \gamma^{-1}$ are holomorphic functions on $\check{S}$, 
they are bounded on $S$. 
Hence there is a positive number $K$ such that, for any $s\in S$ and any $(\overline{x}, \overline{y}, \overline{z}) \in \mathbb{C}^3$,  
\[
|a(s)\overline{x}| \le K |\overline{x}|, \  |b(s)\overline{y}| \le K |\overline{y}| 
\mbox{ and }  |\frac{c(s)}{\gamma(s)}  \overline{z}| \le  K|\overline{z}|. 
\]
Since $u$ is   a holomorphic function on $\mathbb{D}\times \check{S}$, we may assume that $|u(b(s)\overline{y},s)|$ is bounded for $(\overline{y},s)\in \mathbb{D}_{\varepsilon} \times S$, where $\mathbb{D}_{\varepsilon}\subseteq \mathbb{C}$ is the disc of some radius $\varepsilon >0$.  
Hence we may assume further that 
\[
|\zeta(\overline{y},\overline{z},s) |  = \Big|  \frac{c(s)}{\gamma(s)}  \overline{z}  - \frac{b(s)}{\gamma(s)}  \cdot u(b(s)\overline{y},s)  \cdot \overline{y}     \Big| \le K(|\overline{z} | + |\overline{y}|)
\]
whenever $|\overline{y}|< \varepsilon$ and $s\in S$.

Since $M'\cap M^\circ$ is a hypersurface defined in $\mathbb{D'}^3 \times S^\circ$, in order that a point $ (\overline{x},\overline{y}, \overline{z},s) \in Y\times S^\circ$ is contained in $\varphi^{-1}(M'\cap M^\circ)$, it is necessary and sufficient that 
\[
|a(s)\overline{x}| < R',\  |b(s)\overline{y}|   <R' \mbox{ and } |\zeta(\overline{y},\overline{z},s)| < R', 
\]
where $R'$ is the radius of $\mathbb{D}'$. 
Since $S$ in contained in a polydisc of  radius   smaller than 1,  
we have $|T|<1$.  
Since $U$ is bounded,  we deduce that  there is some integer ${N_0}>0$ such that 
\[
K |T^{N_0}\overline{x}| < R', \  K|T^{N_0}\overline{y}| < R',\   |T^{N_0}\overline{y}| < \varepsilon,\   \mbox{ and } K(|T^{N_0} \overline{y}| + |T^{N_0} \overline{z}|) <  R'. 
\]
for all $(\overline{x},\overline{y}, \overline{z},s) \subseteq U \times S^\circ$. 
By Lemma \ref{lemma:construct-contract-tubular}, there exist integers $p,q,r \ge {N_0}$ such that 
\[
\eta \colon  (\overline{x},\overline{y}, \overline{z},s) \mapsto  (T^p\overline{x}, T^q\overline{y}, T^r\overline{z},s), 
\]
is a well-defined endomorphism of $Y\times S $. 
We then deduce that $\eta(U\times S^\circ)$ is contained in $\varphi^{-1}(M'\cap M^\circ)$.  

For the second property on $\eta$, we assume that  $X'$ is a neighborhood of $s$ in $X$. 
Then there is a positive number $\varepsilon'$, 
there is an open  neighborhood $S'\subseteq S$ of $s$ in $S$, 
such that 
\[
M \cap \big((\mathbb{D}_{\varepsilon'})^3 \times S' \big)  \subseteq X', 
\]
where $\mathbb{D}_{\varepsilon'} \subseteq \mathbb{C}$ is the disc of radius $\varepsilon'$. 
We remark that the value of $T$ at $s$ is equal to 0. 
Hence we may pick $S'$ in order that  $|T|^{N_0}$ is very small on $S'$. 
More precisely, since $U$ is bounded, we can pick $S'$ so that 
\[
K |T^p\overline{x}| <  \varepsilon', \  K|T^q\overline{y}| <  \varepsilon' , \  |T^q\overline{y}| <  \varepsilon  \mbox{ and } K(|T^q \overline{y}| + |T^r \overline{z}|) < \varepsilon'. 
\] 
for   all $(\overline{x},\overline{y}, \overline{z},s) \in U \times S'$. 
It follows that 
\[
(\varphi \circ \eta)  \big( U \times (S'\cap S^\circ) \big) 
\]
is contained in $X'$.   
This completes the proof of the lemma. 
\end{proof}

In the remainder of the section, by replacing $\varphi$ with $\varphi\circ \eta$, 
we assume that $\varphi$ satisfies the properties of $\varphi\circ \eta$  in  Lemma \ref{lemma:construct-tubular-2}. 
Let 
\[V= \varphi(U \times S^\circ) \subseteq M'\cap M^\circ \subseteq X^\circ.    \]
Then $V$ is a  tubular  neighborhood of $S^\circ$ in $X^\circ$. 
The fundamental group $\pi_1( V_{\sm} )$ is isomorphic to $ \pi_1(Y_{\sm}) \times \mathbb{Z}^k$.   
We observe that  $V$ does not extend to an  open neighborhood of $S$ in $X$. 
Indeed, since $N_0>0$,  a point  $(x,y,z,s) \in V$ approaches  to $(0,0,0,s_H)$ when $s$ tends  to a point $s_H\in H$.  
In the following lemma, we show that, if we blowup $X$ at the $H_i$'s for sufficiently many times, 
then the preimage  of $V$ extends to a neighborhood of the strict transform of  $S$.

\begin{lemma}
\label{lemma:construct-tubular-3}
With the notation above,  
there exists an integer $N>0$ such that the following property holds. 
Let $\rho \colon \widehat{M} \to M$ be a   morphism obtained by blowing up $M$ at the $H_i'$s as in Subsection \ref{subsection:blowup-divisor}, 
such that there  are at least $N$ blowups whose centers are $H_i$ for all $i$.    
Let $\widehat{X} = \rho^{-1}(X)$, $\widehat{X}^\circ = \rho^{-1}(X^\circ)$, $\widehat{V} = \rho^{-1}(V)$ and  let $\widehat{S}$ be the strict transform of $S$ in $\widehat{M}$. 
Then there is  an  open neighborhoods $W$  of $\widehat{S}$ in $\widehat{X}$, 
such that  $ W^\circ := W \cap \widehat{X}^\circ $  is contained in $ \widehat{V}$, 
and that  $W $ is invariant under the natural action of $G$ on $\widehat{X}$. 
Moreover, we have  $\rho(W \cap \rho^{-1}(\Delta)) \subseteq H\cap S$. 
\end{lemma}

\begin{proof}
There is some $\varepsilon >0$ such that the intersection  $Y \cap  (\mathbb{D}_{\varepsilon})^3$ is contained in $U$.   
Then $ \varphi   ( (Y \cap (\mathbb{D}_{\varepsilon})^3) \times S^\circ ) \subseteq V$. 
The  application $\varphi^{-1}\colon M^\circ \to Y\times S^\circ$ can be written with coordinates as 
\[
\varphi^{-1}\colon (x,y,z,s) \mapsto  (a(s)^{-1}x, b(s)^{-1}y, \xi(y,z,s),s),  
\] 
where $\xi(y,z,s) =  c(s)^{-1}(\gamma(s)z + u(y,s)y)$.   
Hence,  a point $(x,y,z,s) \in M^\circ $ with $s\in S^\circ$ is contained in 
\[ \varphi  \big( (Y \cap (\mathbb{D}_{\varepsilon})^3) \times S^\circ \big)\]
if and only if  
\begin{equation}\label{eqn:condition-tubular}
|   a(s)^{-1}x| < \varepsilon, \ | b(s)^{-1}y|<\varepsilon  \mbox{ and } | c(s)^{-1}(\gamma(s)z + u(y,s)y)|<\varepsilon. 
\end{equation}
Since the zero loci of $a,b,c$ are contained in $H$, 
there is some   positive integer $N$ such that 
\[
T^N\cdot   a(s)^{-1}, \  T^N\cdot    b(s)^{-1}    \mbox{ and }  T^N\cdot   c(s)^{-1} 
\]
are holomorphic functions on $\check{S}$. 
We will show that this number $N$ satisfies the assertion of the lemma.

As shown in Subsection \ref{subsection:blowup-divisor},  
there is a coordinates system $(\hat{x}, \hat{y}, \hat{z},s)$ on  a neighborhood of the closure of $\widehat{S}$ in $\widehat{M}$, 
such that 
$\rho$ can be expressed as 
\[
\rho\colon (\hat{x}, \hat{y}, \hat{z},s) \mapsto (\hat{T} \hat{x}, \hat{T} \hat{y}, \hat{T} \hat{z}, s), 
\]
where $\hat{T} = T_1^{d_1}\cdots T_k^{d_k}$ and $d_i$ is the number of blowups in $\rho$ whose centers are $H_i$.  
By pulling back the conditions of \eqref{eqn:condition-tubular} to $\widehat{M}$, we obtain the following conditions 
\begin{equation*}
|\hat{T}\cdot    a(s)^{-1} \hat{x}|<\varepsilon,  \ |\hat{T}\cdot  b(s)^{-1} \hat{y}|< \varepsilon \mbox{ and } |\hat{T} \cdot c(s)^{-1}(\gamma(s)\hat{z} + u(\hat{T} \hat{y},s)\hat{y})|<\varepsilon. 
\end{equation*}
Since $d_1,...,d_k \ge N$, the functions 
\[
\hat{T}\cdot    a(s)^{-1}, \  \hat{T}\cdot  b(s)^{-1}  \mbox{ and } \hat{T} \cdot c(s)^{-1}
\]
are holomorphic functions on $\check{S}$. 
Hence they are bounded on $S$. 
Therefore there is some $\delta>0$ such that if  
 $\widehat{M}^\circ  := \rho^{-1}(M^\circ)$  and 
if $\hat{o}= (\hat{x}, \hat{y}, \hat{z},s)$ is contained in  
$\widehat{M}^\circ  \cap ( (\mathbb{D}_\delta)^3 \times \widehat{S})$, 
then 
\[\rho(\hat{o}) \in  \varphi  \big( (Y \cap (\mathbb{D}_{\varepsilon})^3) \times S^\circ \big). \]
Let $\widehat{W} =  ( (\mathbb{D}_\delta)^3 \times \widehat{S} )\cap \widehat{M}$.   
It follows  that 
\[\rho(\widehat{W} \cap \widehat{M}^\circ) \subseteq   V.  \]
Thus we get $\widehat{W} \cap \widehat{M}^\circ \subseteq \widehat{V} \subseteq \widehat{X}$.

Assume that $\hat{o} \in \widehat{W} \cap \rho^{-1}(\Delta) = \widehat{W} \setminus \widehat{M}^\circ$ is a point with coordinate $(\hat{x}, \hat{y}, \hat{z},s)$ in $(\mathbb{D}_\delta)^3 \times \widehat{S}$.  
Then $\hat{T}(s) = 0$ and  $\rho(\hat{o})$ has coordinates $(0,0,0,s)$ in $M$. 
This implies that  $\rho(\hat{o}) \in H\cap S \subseteq X$. 
In particular, we see that $\widehat{W} \subseteq \widehat{X}$.  
Hence we can define $W= \bigcap_{g\in G} g(\widehat{W})$.  
Then $W$ is a $G$-invariant neighborhood of $\widehat{S}$ in $\widehat{X}$.
This completes the proof of  the lemma.  
\end{proof}

The following lemma is the key to the item (4) of Proposition \ref{prop:construct-local-cover}.

\begin{lemma}
\label{lemma:construct-local-cover-1}
With the notation in Lemma \ref{lemma:construct-tubular-3},   
we fix  such a bimeromorphic morphism $\rho \colon \widehat{X} \to X$  such that there are  $6N'$ blowups whose centers are $H_i$ for all $i$, where $N' \ge N$ is an integer.    
Then there is a sequence of morphisms of fundamental groups  
\begin{equation*}
\xymatrix{ 
	\pi_1(W^\circ_{\sm} ) \ar[r]^-{\rho_*}    &    \pi_1( V_{\sm} )  \ar[r]^-{\varphi^{-1}_*}  &   \pi_1( Y_{\sm} \times \check{S}^\circ ).
}
\end{equation*}
We denote by $\iota \colon  \pi_1(W^\circ_{\sm} )  \to \pi_1( Y_{\sm} \times \check{S}^\circ )$ the composition of the previous sequence.  
Then  $\iota$ is surjective. 

Furthermore, if  $\psi \colon X^\circ  \to Y\times \check{S}^\circ$ is an open embedding, 
whose restriction on  $\check{S}^\circ \subseteq X^\circ$ is equal to the identity map from $\check{S}^\circ$ to  $\{\mathbf{0}_3\} \times \check{S}^\circ$. 
Then   the morphism of fundamental groups 
\[
(\psi \circ \rho)_*  \colon \pi_1(W^\circ_{\sm} )  \to \pi_1( Y_{\sm} \times \check{S}^\circ )
\]
is surjective and its kernel is equal to the kernel of $\iota$. 
\end{lemma} 

\begin{proof}
We notice that  $\widehat{S} \cong S$. 
From the construction of $W$ in Lemma \ref{lemma:construct-tubular-3}, 
we see that if $\widehat{H}$ is the preimage of $H$ in $\widehat{S}$, 
then  the data $W$, $\rho^{-1}(\Delta)|_W$, $\widehat{S}$, $\widehat{H}$ and $G$  satisfy the items (1) and (2) of Setup \ref{setup-local-cover}.   
If $o_{\widehat{S}}$ is the preimage of $o_S$ in $\widehat{S}$, 
then there is an open neighborhood $S_1$ of $o_{\widehat{S}}$ in $\widehat{S}$, 
which plays the same role for $o_{\widehat{S}}$ as $S$ for $o_S$ in the item (4) of Setup \ref{setup-local-cover}. 
By applying Lemma \ref{lemma:construct-tubular-2}  to $W$, 
there is an open embedding $\hat{\varphi}\colon U\times S_1^\circ  \to W$ whose restriction on $\{\mathbf{0}_3\}\times S_1^\circ$ is the identity map to $S_1^\circ$,   
where $S_1^\circ = S_1\setminus \widehat{H}$.    
Since there are natural isomorphisms 
\[\pi_1(S_1^\circ) \cong \pi_1(S^\circ) \cong \pi_1(\check{S}^\circ) \cong \mathbb{Z}^k,\]
by Lemma \ref{lemma:tubular-fund-group-iso}, 
the composition of the following  sequence 
\begin{equation*}
\xymatrix{ 
			  \pi_1(U_{\sm}\times S_1^\circ) \ar[r]^-{\hat{\varphi}_*}    &    \pi_1( W^\circ_{\sm} )  \ar[r]^-{\iota }  &   \pi_1( Y_{\sm} \times \check{S}^\circ ) 
}
\end{equation*} 
is an isomorphism.  
Thus $\iota$ is surjective.

We recall that $  \varphi   \colon U\times S^\circ \to V$ is an isomorphism. 
Since there is a natural isomorphism $\pi_1(U_{\sm}) \cong \pi_1(Y_{\sm})$, we deduce that
\[\varphi^{-1}_*\colon \pi_1(V_{\sm}) \to  \pi_1( Y_{\sm} \times \check{S}^\circ )\] 
is an isomorphism,   
and we obtain that $\rho_*\colon   \pi_1(W^\circ_{\sm} ) \to     \pi_1( V_{\sm} )$ is surjective. 
Moreover, the kernel of $\iota$ is equal to the one of $\rho_*$. 

Assume that we have a morphism $\psi\colon X^\circ \to Y\times \check{S}^\circ$ as in the second part of the lemma. 
Then the morphism 
\[
\psi\circ  \varphi \colon   U\times S^\circ  \to Y\times \check{S}^\circ
\]
is an open embedding whose restriction on $\{\mathbf{0}_3\}\times S^\circ$ is the identity map. 
Thus it induces an isomorphism   from $\pi_1( U_{\sm}\times S^\circ)$ to $ \pi_1(Y_{\sm} \times \check{S}^\circ)$ by Lemma \ref{lemma:tubular-fund-group-iso}. 
Hence
\[
\psi_* \colon \pi_1(V_{\sm} )  \to  \pi_1(Y_{\sm} \times \check{S}^\circ) 
\]
is an isomorphism.  
Since $\rho_*\colon   \pi_1(W^\circ_{\sm} ) \to     \pi_1( V_{\sm} )$ is surjective, 
we deduce that   $(\psi\circ \rho)_*$ is surjective. 
Furthermore, the kernel of $(\psi\circ \rho)_*$ is the one of  $\rho_*$, 
which is also  the kernel of $\iota$. 
This completes the proof of the lemma. 
\end{proof}

The item (5) of Proposition \ref{prop:construct-local-cover} follows essentially from the next lemma.

\begin{lemma}
\label{lemma:construct-local-cover-2}
With the notation in Lemma \ref{lemma:construct-local-cover-1}, 
let $\hat{o}\in W\setminus \widehat{X}^\circ_{\sm}$ be a point and let $o=\rho(\hat{o}) \in X$. 
Let $\Sigma \subseteq  X$ be an open neighborhood  of $o$, and let $\Sigma^\circ = \Sigma\cap X^\circ$. 
Assume that for $i=1,2$, there is an open embedding  
$\psi_i \colon \Sigma^\circ \to Y\times (\mathbb{C}^n\setminus H^i)$
satisfying the following properties. 
Each $ H^i$  is the union of some coordinate hyperplanes in $\mathbb{C}^n$, 
the restriction $\psi_i|_{\check{S}\cap \Sigma^\circ}$ has images contained in $\{\mathbf{0}_3\}\times \mathbb{C}^n$, 
and it extends to an open embedding from $\check{S}\cap \Sigma$ to $\mathbb{C}^n$ such that $\psi_i(\check{S}\cap \Sigma^\circ) = \psi_i(\check{S}\cap \Sigma)\setminus  H^i$.  

Then  there is an open neighborhood  $\widehat{\Sigma} \subseteq W$ of $\hat{o}$, such that 
the following properties hold. 
\begin{enumerate}
\item The morphisms of fundamental groups  from  $ \pi_1(\widehat{\Sigma}_{\sm} \cap  \widehat{X}^\circ)  $ to $ \pi_1( Y_{\sm}\times  (\mathbb{C}^n\setminus H^i) )$ 
induced by $\psi_i\circ \rho $, with $i=1,2$,    have the same kernel. 
\item Let $k^1$ and $k^2$ be the numbers of the irreducible components of  $H^1$ and $H^2$ respectively, 
and let  $k'$ be the number of the irreducible components of $H$ passing through $o$. 
For $i=1,2$,   the image of $(\psi_i\circ \rho)_*$ is the conjugacy class of a normal subgroup and 
\[
\pi_1( Y_{\sm}\times  (\mathbb{C}^n\setminus H^i) )  /  (\psi_i\circ \rho)_* (\pi_1( \widehat{\Sigma}_{\sm} \cap  \widehat{X}^\circ ) )
\cong 
\mathbb{Z}^{k^i - k'}. 
\] 
\end{enumerate} 
\end{lemma}

\begin{proof} 
We first assume that the point $\hat{o}$ is in $W^\circ$. 
Then   $\hat{o}  \in  W^\circ_{\sing}$. 
Since $\rho$ is an isomorphism on $W^\circ$, 
we deduce that $o \in S^\circ \subseteq X^\circ$.  
It follows that $k'=0$. 
We note that $X$ has the same type of singularities at points of $S^\circ$. 
Hence there is a contractible open neighborhood $S'$ of $o$ in $S^\circ$, 
an open neighborhood $V'$ of $o$ in $\Sigma^\circ$, 
and an open neighborhood $Y'$ of $\mathbf{0}_3$ in $Y$,  
such that $V'\cong  Y' \times  S'$. 
We may assume further that the natural morphism $\pi_1(Y'_{\sm}) \to \pi_1(Y_{\sm})$ is isomorphic.  
Let $\widehat{\Sigma}$ be the preimage of $V'$ in $\widehat{X}$. 
Then we have  $\hat{o} \in \widehat{\Sigma} \subseteq \widehat{X}^\circ$ and $\rho$ induces an isomorphism from $\widehat{\Sigma}$ to $V'$. 
In particular,  $\rho(\widehat{\Sigma}) \cap \check{S}^\circ = S'$, and there are natural isomorphisms 
\[
\pi_1(\widehat{\Sigma}_{\sm}) \cong  \pi_1(V'_{\sm}) \cong \pi_1(Y_{\sm}). 
\]

Since   $\psi_i$  induces open embeddings from $V' \cong Y'\times S'$  to $Y\times (\mathbb{C}^{n}\setminus H^i)$,  whose restriction on $S'=V'_{\sing}$ has images contained in $\{\mathbf{0}_3\} \times (\mathbb{C}^{n}\setminus H^i)$, 
by Lemma \ref{lemma:tubular-fund-group-iso}, it induces an injective morphism  from $\pi_1(V'_{\sm})$
to $\pi_1(Y_{\sm}\times (\mathbb{C}^{n}\setminus H^i))$. 
Therefore, the morphism  
from $ \pi_1(\widehat{\Sigma}_{\sm} \cap \widehat{X}^\circ) $ to $\pi_1( Y_{\sm}\times  (\mathbb{C}^{n}\setminus H^i))$   
induced by $\psi_i\circ \rho $ is injective. 
This shows the item (1). 
Lemma \ref{lemma:tubular-fund-group-iso} also implies that  
the image of $(\psi_i\circ \rho)_*$ is the image of $\pi_1(Y_{\sm}\times \psi_i(S'))$ induced by the natural inclusion.  
This proves the item (2).

It remains to treat the case when $\hat{o} \in W \setminus W^\circ = W\cap \rho^{-1}(\Delta)$.  
This assumption implies that $o\in H \cap S$  by Lemma \ref{lemma:construct-tubular-3}. 
We recall that $\theta\colon M\to \check{S}$ is the projection sending a point $(x,y,z,s)\in M$ to $s\in \check{S}$.   
Let $\hat{\theta}\colon  \widehat{M} \to \check{S}$  be the  composition $\theta\circ \rho$.  
By Lemma \ref{lemma:construct-tubular-2}, there is an open neighborhood $S' $ of $o$ in $S$, such that 
\[
 V' := \varphi   \big(U \times (S'\cap S^\circ) \big)  = V \cap \theta^{-1}(S')
\]
is contained in $\Sigma^\circ$.   
Up to shrinking $S'$, we can assume that it is isomorphic to a polydisc whose origin is $o$,  
and  $H|_{S'}$ is isomorphic to the union of coordinates hyperplanes.  
We define $\widehat{\Sigma} = W  \cap \hat{\theta}^{-1}(S')$.  
Then $\hat{o}\in \widehat{\Sigma}$,     
\[
\rho(\widehat{\Sigma})\cap \check{S} = S', \ 
\rho(\widehat{\Sigma} \cap \widehat{X}^\circ) \subseteq V'  
\mbox{ and }   
\rho(\widehat{\Sigma}_{\sm} \cap \widehat{X}^\circ) \subseteq V'_{\sm}. 
\]
We notice that $V' \cong U \times (S'\cap  \check{S}^\circ)$.  
By assumption, for $i=1,2$,   
$\psi_i$ induces open an  embedding 
from $V' $ to $Y\times  (\mathbb{C}^{n}\setminus H^i)$,  
whose restriction on $S'\cap \check{S}^\circ=V'_{\sing}$  has image contained in $\{\mathbf{0}_3\}\times \mathbb{C}^n$,  
and extends to an open embedding from $S'$ to $\mathbb{C}^n$,  
such that $\psi_i(S'\cap \check{S}^\circ) = \psi_i(S') \setminus H^i$. 
Hence by Lemma \ref{lemma:tubular-fund-group-iso},   
the morphism 
\[\pi_1(V'_{\sm})  \to \pi_1( Y_{\sm}\times  (\mathbb{C}^{n}\setminus H^i)) \] 
induced by $\psi_i$ is injective.   
We then deduce that  the  kernel  of the morphism  
from $ \pi_1(\widehat{\Sigma}_{\sm} \cap \widehat{X}^\circ) $  to  $\pi_1( Y_{\sm}\times  (\mathbb{C}^{n}\setminus H^i))$ 
induced by $\psi_i\circ \rho $ is the same as the one of 
$\rho_*\colon   \pi_1(\widehat{\Sigma}_{\sm} \cap  \widehat{X}^\circ)  \to \pi_1( V'_{\sm} )$, 
which is independent of $i$. 
Hence we obtain the item (1). 

Lemma \ref{lemma:tubular-fund-group-iso} also implies that the image of $(\psi_i)_*(\pi_1(V'_{\sm}))$ is the same as the image of $\pi_1(Y_{\sm}\times \psi_i(S'\cap \check{S}^\circ))$ induced by the natural inclusion.  
Furthermore, by  the same argument as in the first two paragraph of the proof of Lemma \ref{lemma:construct-local-cover-1}, we see that the morphism $\rho_*\colon   \pi_1(\widehat{\Sigma}_{\sm} \cap  \widehat{X}^\circ)  \to \pi_1( V'_{\sm} )$ is surjective.  
Hence   the item (2) holds as well. 
This   completes the proof of the lemma.
\end{proof}

Now we can complete the proof of Proposition \ref{prop:construct-local-cover}. 

\begin{proof}[{Proof of Proposition \ref{prop:construct-local-cover}}] 
We let $N$ be the integer in Lemma \ref{lemma:construct-tubular-3} and  we fix an integer  $N'\ge N$. 
With the notation in Lemma \ref{lemma:construct-local-cover-1}, there is a morphism of fundamental groups\[
\iota = (\varphi^{-1}\circ \rho)_* \colon \pi_1(W^\circ_{\sm}) \to \pi_1 (Y_{\sm} \times \check{S}^\circ). 
\]
We note that $\pi_1 (Y_{\sm} \times \check{S}^\circ) \cong G_Y \times  \mathbb{Z}^k$, where $G_Y$ is the fundamental group of $Y_{\sm}$, which is finite.  
Let $\mathcal{N}\subseteq \pi_1 (Y_{\sm} \times \check{S}^\circ)$ be the normal subgroup defined in Lemma \ref{lemma:standard-subgroup},  
and let $\mathcal{H} = \iota^{-1}(\mathcal{N})$. 
Then $\mathcal{H}$ is a normal subgroup of $\pi_1(W^\circ_{\sm})$, and it determines a finite Galois  \'etale cover of $W^\circ_{\sm}$.      
Let $\nu\colon \overline{W} \to W$ be the induced finite   cover with $\overline{W}$ normal, see Theorem \ref{thm:GR-cover}. 
Then   $\nu$    is Galois and  quasi-\'etale over $W^\circ$.   
We have hence proved the items (1) and (2).

We will show that $\overline{W}^\circ:=  \nu^{-1}(W^\circ)$ is smooth.
Let $Z\to Y\times \check{S}^\circ$ be the finite quasi-\'etale morphism induced by the subgroup $\mathcal{N} \subseteq \pi_1 (Y_{\sm} \times \check{S}^\circ)$, see Theorem \ref{thm:GR-cover}.  
Then it factors through $\overline{Y}\times \check{S}^\circ \to Y\times \check{S}^\circ$, where $\overline{Y} \to Y$ is the quasi-\'etale cover induced by the universal cover of $Y_{\sm}$. 
In particular, $\overline{Y}$ is smooth.  
The Zariski's purity theorem then implies that $Z$ is smooth. 
We note that 
$\nu|_{\overline{W}^\circ}$ is the basechange of $ \varphi ^{-1}\circ \rho \colon  W^\circ \to Y\times \check{S}^\circ$ 
over $Z\to Y\times \check{S}^\circ$. 
\begin{equation*}
\xymatrix{ 
  \overline{W}^\circ \ar[d]	\ar[rr]^{\nu}  && W^\circ     \ar[d] \\ 
  Z \ar[r] & \overline{Y}\times \check{S}^\circ  \ar[r]  & Y\times \check{S}^\circ   
}
\end{equation*}
Since $\rho$ is an isomorphism on $W^\circ$, we see that 
$\varphi^{-1}\circ \rho$ is an open embedding on $W^\circ$. 
Hence  $\overline{W}^\circ$ is  embedded into $Z$ as an open subset, and is  smooth. 
This proves the item (3).

To  prove the item (4), it is enough to show that $\mathcal{H}$ is $G$-invariant.  
Let $g\in G$ be an element. 
Then $g^{-1} (\mathcal{H})$ is the preiamge of $\mathcal{H}$ under the morphism of fundamental groups induced by $g\colon W\to W$.
Let 
\[
\iota' :=  ( g^{-1}\circ \varphi^{-1} \circ \rho \circ g)_* \colon \pi_1(W^\circ_{\sm})  \to \pi_1 (Y_{\sm} \times \check{S}^\circ), 
\]
where the morphism $g^{-1}$ above is the automorphism of $Y\times \check{S}^\circ$ defined by the action of $G$ on $\check{S}^\circ$.  
Then we have 
\[ 
g^{-1} (\mathcal{H}) =    ((\varphi^{-1} \circ \rho \circ g)_*)^{-1}(\mathcal{N})    =  (\iota')^{-1} (\mathcal{N}). 
\]
We note that $\rho\circ g  = g\circ \rho$ as morphisms from $W$ to $X$.  
By  applying the second part of  Lemma \ref{lemma:construct-local-cover-1} with 
\[\psi = g^{-1}\circ \varphi^{-1} \circ g  \colon X^\circ \to Y\times \check{S}^\circ ,  \]
we deduce that the two morphisms  $\iota$ and $\iota'$ are surjective of the same kernel.  
Hence, there is an automorphism  $\tau$ of $\pi_1( Y_{\sm} \times \check{S}^\circ ) $  such that 
\[
\iota = \tau \circ \iota'. 
\]
By Lemma \ref{lemma:standard-subgroup}, we have $\tau^{-1}(\mathcal{N}) = \mathcal{N}$. 
It follows that   $g^{-1} (\mathcal{H}) = \mathcal{H}$. 
This shows  the item (4). 

It remains to prove the item (5).    
If $\hat{o}\in \widehat{X}^{\circ}_{\sm}$, 
then by construction, $\nu$ an $\nu'$ are both \'etale over $\hat{o}$. 
We can let $W''$ be a smooth contractible  neighborhood of $\hat{o}$ in $W'$ in this case.

Next we assume that $\hat{o}$ is not contained in $\widehat{X}^{\circ}_{\sm}$.  
Then $o=\rho(\hat{o}) \in S$, see Lemma \ref{lemma:construct-tubular-3}.  
As explained in the paragraph after Proposition \ref{prop:construct-local-cover}, 
by assumption, 
there is an open neighborhood $X'$ of $o=\rho(\hat{o})$ containing $\rho(W')$, 
there is an open embedding $ \psi \colon (X' \cap X^\circ) \to Y\times (\mathbb{C}^n\setminus H')$, where $H'$ is the union of some coordinate hyperplanes in $\mathbb{C}^n$,  such that the following properties hold. 
The restriction of  $\psi$  on $X'\cap \check{S}^\circ$ 
has images contained in $\{\mathbf{0}_3\} \times \mathbb{C}^n$. 
It extends to an open embedding from $X'\cap \check{S}$ to $\mathbb{C}^n$, 
such that $\psi(X'\cap \check{S}^\circ) = \psi(X'\cap \check{S})\setminus H'$. 
Furthermore,  
$\nu'$ is induced by the kernel of the composition  of the following (conjugacy classes of) morphisms of groups  
\begin{equation*}
\xymatrix{ 
			&   \pi_1( (W' \cap \widehat{X}^\circ)_{\sm} )  \ar[r]^{\rho_*}    &    \pi_1( (X'\cap X^\circ)_{\sm} ) \\ 
            \ar[r]^-{\psi_*}  &   \pi_1( Y_{\sm} \times (\mathbb{C}^n\setminus H') ) \ar[r] &  \pi_1( Y_{\sm} \times (\mathbb{C}^n\setminus H') ) / \mathcal{N}', 
}
\end{equation*}
where $\mathcal{N}'$ is the canonical normal subgroup of $\pi_1( Y_{\sm} \times (\mathbb{C}^n\setminus H') )$ defined in Lemma \ref{lemma:standard-subgroup}.

By Lemma \ref{lemma:construct-local-cover-2} and its proof,  
there is a contractible open neighborhood $S''$ of $o$ in  $S$,  
and an open neighborhood $W''$ of $\hat{o}$ in $W'$,   
such  that $\rho(W'')\cap \check{S} = S''$, 
and that the two   morphisms 
\[
 \sigma_1 =(\varphi^{-1} \circ \rho)_* \colon \pi_1( (W'' \cap \widehat{X}^\circ)_{\sm} )  \to  \pi_1( Y_{\sm} \times   \check{S}^\circ  ) 
\]
and 
\[
\sigma_2 =  (\psi \circ \rho)_*  \colon \pi_1( (W'' \cap \widehat{X}^\circ)_{\sm} ) \to \pi_1( Y_{\sm} \times   (\mathbb{C}^n\setminus H')  )
\]
have the same kernel.  
Their images are both normal subgroups and both isomorphic to $G_Y \times \mathbb{Z}^{k''}$, 
where  $k''$ is the number of the components  of $H$ which pass through $o$.   
Hence there is a surjective morphism $\mu \colon \pi_1( (W'' \cap \widehat{X}^\circ)_{\sm} ) \to G_Y \times \mathbb{Z}^{k''}$, such that  $\sigma_i = \chi_i\circ \mu$ for $i=1,2$,  where 
\[
 \chi_1  \colon  G_Y \times \mathbb{Z}^{k''}  \to  \pi_1( Y_{\sm} \times   \check{S}^\circ  ) 
\] 
and 
\[
 \chi_2  \colon  G_Y \times \mathbb{Z}^{k''}  \to  \pi_1( Y_{\sm} \times   (\mathbb{C}^n\setminus H')  )
\] 
are injective morphisms. 
Furthermore, the cokernel of $\chi_1$ is isomorphic to $\mathbb{Z}^{k-k''}$ and the one of $\chi_2$ is isomorphic to  $\mathbb{Z}^{k'-k''}$, where $k'$ is the number of irreducible components of $H'$.

We notice that, if  ${\nu^{-1}(W'')}_{\circ} $ is a connected component of ${\nu^{-1}(W'')}$, 
then $\nu|_{\nu^{-1}(W'')_\circ}\colon {\nu^{-1}(W'')}_{\circ} \to W''$ is induced by the kernel of the  composition of the following sequence of morphisms of groups
\begin{equation*}
\xymatrix{ 
 &	\pi_1( (W'' \cap \widehat{X}^\circ)_{\sm} ) \ar[r]^-{\mu}    
 &   G_Y\times \mathbb{Z}^{k''} \\ 
     \ar[r]^-{\chi_1} &   \pi_1( Y_{\sm} \times \check{S}^\circ )         \ar[r]  
 &     \pi_1( Y_{\sm} \times \check{S}^\circ ) / \mathcal{N},    
}  
\end{equation*}
which is equal to $\mu^{-1}(\chi_1^{-1}(\mathcal{N}))$. 
Since  the cokernel of $\chi_1$  is isomorphic to $\mathbb{Z}^{k-k'}$, 
the subgroup $\mathcal{N}'' := \chi_1^{-1}(\mathcal{N})$ of  $ G_Y\times \mathbb{Z}^{k''} $ is the canonical subgroup defined in Lemma \ref{lemma:standard-subgroup}. 
Then  $\nu|_{\nu^{-1}(W'')_{\circ}}$ is induced by the subgroup $\mu^{-1}(\mathcal{N''})$. 
Similarly,  if  ${(\nu')^{-1}(W'')}_{\circ} $ is a connected component of ${(\nu')^{-1}(W'')}$,
then the finite morphism  $\nu'|_{(\nu')^{-1}(W'')_{\circ}} \colon  {(\nu')^{-1}(W'')}_{\circ} \to W''$ is induced by the subgroup $\mu^{-1}( \chi_2^{-1}(\mathcal{N'}) )$. 
The same argument  implies that $ \chi_2^{-1}(\mathcal{N'})$ is the canonical subgroup of $ G_Y\times \mathbb{Z}^{k''}  $ in Lemma \ref{lemma:standard-subgroup}. 
Hence $\chi_2^{-1}(\mathcal{N}) = \mathcal{N}''$, and we deduce that $\nu|_{\nu^{-1}(W'')_{\circ}}$ is isomorphic to $\nu'|_{(\nu')^{-1}(W'')_{\circ}}$.   
This completes the proof of the proposition. 
\end{proof}

To end this section, we observe that the item (2) of Setup \ref{setup-local-cover} on the divisibility of vanishing orders can be achieved by taking cyclic covers. 

\begin{lemma}
\label{lemma:div-reduction}
With the notation in the item (1) of Setup \ref{setup-local-cover}, we let 
$  \overline{S} = \check{S}[\sqrt[6]{T_1},..., \sqrt[6]{T_n}] $ and let 
$\overline{M}  = M\times_{\check{S}} \overline{S}.$   
We denote by $\overline{H}_i$ the preimage of $H_i$ in $  \overline{S}$. 
Then $(\overline{S}_i, \sum_{i=1}^n \overline{H_i})$ is snc and 
$M$ is a hypersurface in $\mathbb{D}^3\times \overline{S}$ defined by a function of the shape of Proposition \ref{prop:good-shape} with respect to $(\overline{S}_i, \sum_{i=1}^n \overline{H_i})$, such that the vanishing orders of the functions $\alpha$ and $\beta$ on $\overline{S}$ along each $\overline{H}_i$ are divisible by 6.
\end{lemma}

\begin{proof}
In order to fix the notation, we assume without loss of the generality that the defining function of $M$ is
\[
 F(x,y,z,s) = x^2 + \alpha(s)y^3 +  \beta(s) \cdot y(\gamma(s)z +u(y,s) y)^3.
\]
By Lemma \ref{lemma:cyclic-cover-Cartier}, the pair $(\overline{S}, \sum_{i=1}^n \overline{H_i})$ is snc. 
Indeed, $\overline{S}$ is also isomorphic to a neighborhood of the origin in $\mathbb{C}^n$, 
and $\overline{H}_i$ is isomorphic to a coordinate hyperplane in $\overline{S} \subseteq \mathbb{C}^n$.  
Then $\overline{M}$ is the hypersurface in $\mathbb{D}^3\times \overline{S}$ define by the function  
\[
\overline{F}(x,y,z,\overline{s}) = x^2 + \alpha(p(\overline{s}))y^3 +  \beta(p(\overline{s})) \cdot y(\gamma(p(\overline{s}))z +u(y,p(\overline{s})) y)^3,
\] 
where $\overline{s}\in \overline{S}$ is a point and $p\colon \overline{S} \to \check{S}$ is the natural projection. 
From the construction, we deduce directly that the vanishing orders of $\alpha\circ p$ and $\beta \circ p$ along each $\overline{H}_i$ are divisible by 6.  
This completes the proof of the lemma. 
\end{proof}

\section{Modification around an irreducible singular  locus} 
\label{section:local-proof}

The objective of this section is to prove the following proposition, it combines the steps from Section \ref{section:double-point} to Section \ref{section:local-cover}.

\begin{prop}
\label{prop:local-proof-1}
Let $X$ be a complex analytic variety with klt singularities.  
Assume that $S\subseteq X_{\sing}$ is a compact irreducible component of codimension 2 in $X$.  
Then there is a projective bimeromorphic morphism  $f\colon Y\to X$,   with $Y$ normal, such that  the following properties hold. 
Let $\Delta$ be the whole reduced exceptional divisor of $f$. 
\begin{enumerate}
    \item There is a $f$-ample and $f$-exceptional divisor. 
    \item The indeterminacy locus of $f^{-1}$   is a proper closed subset of  $S$.  
          In particular, we can define  the strict transform $S_Y$ of $S$ in $Y$. 
    \item There is a normal complex analytic orbispace $\mathfrak{W} = (W_i, G_i, \psi_i)_{i\in I}$,  such that the quotient space $W$ is an open neighborhood of $S_Y$ in $Y$. 
    \item The divisorial  critical locus of $\psi_i\colon W_i \to W$ is contained in $\Delta$.  
    \item $W_i$ is smooth on $\psi_i^{-1} (W\setminus \Delta)$. 
\end{enumerate}
\end{prop}

We first fix the following setup for the section. 

 \begin{setup}
\label{setup-local-proof} 
We consider the data $X$, $\mathfrak{X}$, $S$ and $H$ as follows. 
Let $X$ be a  complex analytic variety. 
Assume that $S\subseteq X_{\sing}$ is a compact irreducible component of codimension 2.  
We are only interested in the behavior of some neighborhood of $S$ in $X$.   
Let $H\subseteq S$ be a proper closed subset, 
and let $\mathfrak{X} = (X_i,G_i,\pi_i)_{i\in I}$ be a complex analytic orbispace,  
with a finite family $I$  of orbifold charts, 
whose quotient space is $X$.  
We set  $S_i = \pi_i^{-1} (S)$ and $H_i = \pi_i^{-1}(H)$.  
Assume that  $S_i$ is an irreducible component of the singular locus of $X_i$, 
and that $X_i$ has the same type of canonical singularities at points of $S_i\setminus H_i$.  
Furthermore, the intersection of the divisorial critical locus of $X_i\to X$ and $S$ is contained in $H$. 
\end{setup}

We will prove Proposition \ref{prop:local-proof-1} in several steps, and the sketch is as follows. 
Since $X$ has klt singularities, there is a complex analytic orbispace   $(X_i,G_i)$ induced by the local index-one covers of $\omega_X$, see Example \ref{example: complex analytic orbispace}.  
The difficult case is when $X_i$ is singular around the preimage of $S$. 
To study this case, we may assume the situation of Setup \ref{setup-local-proof}. 
We first apply Proposition \ref{prop-first-local-reduction} on each orbifold chart $X_i$. 
Then we can obtain some projective bimeromorphic morphism $f\colon Y \to X$. 
We then locate on a neighborhood $U$ of the strict transform $S_Y$ of $S$ in $Y$.
Replacing $X$ by $U$ and $S$ by $S_Y$, we can assume that each $X_i$ satisfies 
the conclusion of Proposition \ref{prop-first-local-reduction}. 
As a result,  we are able to  apply Proposition \ref{prop:good-shape} on each $X_i$. 
Afterward, we may  assume that each $X_i$ satisfies the conclusion of Proposition \ref{prop:good-shape}.
Now we can apply Proposition \ref{prop:construct-local-cover} on each $X_i$. 
This will complete the proof of Proposition \ref{prop:local-proof-1}.

\begin{lemma}
\label{1-construction}
With the notation of Setup \ref{setup-local-proof},   there is a projective bimeromorphic morphism $f\colon Y\to X$ such that the following properties hold. 
\begin{enumerate}
\item There is a $f$-exceptional and $f$-ample divisor. 
\item The indeterminacy locus of $f^{-1}$ is a proper closed subset of  $S$.  
In particular, we can define the strict transform $S_Y$ of $S$ in $Y$
\item The union of the intersection of $S_Y$ with the  $f$-exceptional locus  and the preimage  $(f|_{S_Y})^{-1}(H)$ is pure of codimension 1 in $S_Y$. 
We denote it by $H_Y$. 
\item There is  a complex analytic orbispace $\mathfrak{Y} = (Y_i,G_i)_{i\in I}$ whose quotient space is  $Y$. 
\item Let $(S_Y)_i$ and $(H_Y)_i$ be the preimage of $S_Y$ and $H_Y$ respectively in $Y_i$. 
Then $(S_Y)_i$ is smooth and $H_Y$ induces a snc orbi-divisor on the orbifold $((S_Y)_i,G_i)_{i\in I}$. 
\item For every point $o\in (S_Y)_i$, there is an open neighborhood $U$ of $o$ in $Y_i$, such that $U$ is isomorphic to the hypersurface in  $\mathbb{D}^3 \times ((S_Y)_i \cap U)$ defined by a function of standard form with respect to $((S_Y)_i\cap U,(H_Y)_i\cap U)$, see Section \ref{section:double-point}. 
\item $Y_i$ has the same type of canonical singularities at points of $(S_Y)_i\setminus (H_Y)_i$. 
\end{enumerate}
\end{lemma}

\begin{proof} 
For each $i$, we let $f_i \colon Y_i \to X_i$ be the construction of Proposition \ref{prop-first-local-reduction}, with respect to $H_i\subseteq S_i \subseteq X_i$.  
Since the construction of Proposition \ref{prop-first-local-reduction} is functorial, we see that there is a natural action of  $G_i$ on $Y_i$, and the family $(Y_i,G_i)_{i\in I}$ defines a complex analytic orbispace $\mathfrak{Y}$. 
Let $Y$ be its quotient space. 
Then there is an induced bimeromorphic morphism $f\colon Y\to X$.
Since each step in the construction of $f_i$ is the blowup of some ideal sheaf, which is compatible along the overlaps, 
there is a $f_i$-exceptional and $f_i$-ample divisor $E_i$ which is compatible along the overlaps.  
The collection  $(E_i)_{i\in I}$ induces a $f$-exceptional and $f$-ample divisor in $Y$. 
This proves the items (1), (2) and (4).

We note that   $(S_Y)_i$  is the strict transform of $S_i$ in $Y_i$, hence is smooth. 
We also note that  $(H_Y)_i$ is the union of the preimage of $H_i$ in $(S_Y)_i$ and the intersection of $(S_Y)_i$ with the $f_i$-exceptional locus. 
By construction,   every $(H_Y)_i$ is a snc divisor in $(S_Y)_i$, and $Y_i$ has the same type of canonical singularities around  $(S_Y)_i\setminus (H_Y)_i$.  
This proves the item (3)  and (7). 
The item (6) also follows from the construction of Proposition \ref{prop-first-local-reduction}.

It remains to prove the item (5). 
To this end, we can apply Lemma \ref{lemma:snc-orbidivisor} to the orbi-divisor $((H_Y)_i)_{i\in I}$ in the complex orbifold $((S_Y)_i,G_i)_{i\in I}$.  
Since $(H_Y)_i$ is a snc divisor in $(S_Y)_i$, 
the construction of Lemma \ref{lemma:snc-orbidivisor} is to successively blowing up certain strata of $(H_Y)_i$, which are compatible along the overlaps. 
Hence we can blow up $Y_i$  at the same centers, see Subsection \ref{subsection:blowup}. 
These blowups are again compatible along the overlaps. 
Let $Y'_i \to Y_i$ be this blowup. 
Then we obtain a complex analytic orbi-space $(Y_i', G_i)_{i\in I}$ with quotient space $Y'$. 
Replacing $Y$ by $Y'$,   we can assume that $H_Y\subseteq S_Y$ induces a snc orbi-divisor on the orbifold $((S_Y)_i,G_i)_{i\in I}$.  
This proves the item (5) and completes the proof of the lemma. 
\end{proof}

We recall that the constructions in Section \ref{section:equation} and Section \ref{section:local-cover} are local around each point of the singular locus. 
In the following lemma, we   replace the family $(X_i,G_i, \pi_i)_{i\in I}$ of orbifold charts by a refinement, so that we can apply a local construction on each  $X_i$ entirely.

\begin{lemma}
\label{2-construction}
With the notation of Setup \ref{setup-local-proof}, 
we assume that $\mathfrak{X} = (X_i,G_i, \pi_i)_{i\in I}$,  $S$ and $H$ satisfies the properties of 
 $\mathfrak{Y} = (Y_i,G_i)_{i\in I}$,  $S_Y$ and $H_Y$ in Lemma \ref{1-construction}. 
Then, up to shrinking $X$ around $S$, and  up to replacing  $(X_i,G_i)_{i\in I}$ by an equivalent finite family of orbifold charts, 
we can assume in addition the following properties. 
\begin{enumerate} 
\item If $n=\dim S$, then  $S_i$ is isomorphic to a contractible neighborhood of the origin in  $\mathbb{C}^n$. 
Moreover,  $H_i$ is isomorphic to the union of certain coordinate hyperplanes, and the origin is fixed by the action of $G_i$.
\item $X_i$ is  isomorphic to an neighborhood of  $\{\mathbf{0}_3\}\times S_i$ in the hypersurface in  $\mathbb{D}^3 \times S_i$ defined by an equation of standard form with respect to $(S_i,H_i)$. 
In addition, the composite inclusion $S_i\subseteq X_i \subseteq \mathbb{D}^3 \times S_i$ is identified with  $\{\mathbf{0}_3\}\times S_i \subseteq \mathbb{D}^3 \times S_i$.   
\item There is an integer $N>0$, such that  Proposition \ref{prop:good-shape} holds for each $H_i\subseteq S_i\subseteq X_i$ with $N$, without shrinking $S_i$.  
\end{enumerate} 
\end{lemma}

\begin{proof} 
From the item (6) of Lemma \ref{1-construction}, 
for each point $x\in S_i \subseteq  X_i$, 
we can take a   neighborhood $U_{x}\subseteq X_i$ of it, 
so that the property (2) of the present lemma holds on $U_{x}$.  
Furthermore, we can assume that  Proposition \ref{prop:good-shape} holds on $U_{x}$ entirely,  
with respect to $H_i\cap U_{x}\subseteq S_i \cap U_{x} \subseteq U_{x}$.  
Let $G_{x}$ be the subgroup of $G_i$ fixing $x$. 
By shrinking $U_{x}$, we can assume that $(U_{x},G_{x})$ is an orbifold chart
compatible with   $\mathfrak{X}$.   
Since $S_i$ is smooth, the action of $G_{x}$   on the germ $(x\in S_i)$ is isomorphic to a linear action (see \cite[Lemme 1]{Cartan1954} or \cite[Lemma 9.9]{Kaw88}).
Thus, up to shrinking $U_{x}$, we can assume that $S_i\cap U_{x}$ is   isomorphic to a  small enough contractible neighborhood of the origin in  $\mathbb{C}^{n}$. 
Furthermore,  $H_i\cap U_{x}$ is isomorphic to the union of certain coordinate hyperplanes.  
It follows that the property (1) of the lemma holds on $(U_{x},G_{x})$.

We hence obtain a family of orbifold charts $\{(U_{x},G_{x})\}$ by considering all points $x\in X_i$ and  all $i\in I$. 
Since $S$ is compact, we can extract a finite subset of this family, so that the images of the orbifold charts cover $S$. 
Replacing $(X_i,G_i)_{i\in I}$ by this finite family, and replacing  $X$ by the union of the images of these charts, 
we may assume that the following property  holds for the  complex analytic orbispace $\mathfrak{X} = (X_i,G_i)_{i\in I}$. 
There is an integer $N>0$, such that  Proposition \ref{prop:good-shape} holds for each $H_i\subseteq S_i\subseteq X_i$,  with $N$ and without shrinking $S_i$.  
This completes the proof of the lemma. 
\end{proof}

In the next lemma, we   apply Proposition \ref{prop:good-shape} on each orbifold chart $X_i$ above. 

\begin{lemma}
\label{3-construction} 
With the notation of  Setup \ref{setup-local-proof}, 
we  assume the properties of Lemma \ref{2-construction}  hold. 
Then there is a projective bimeromorphic morphism $f\colon Y\to X$ such that the following properties hold. 
\begin{enumerate}
\item There is a $f$-exceptional and $f$-ample divisor. 
\item The indeterminacy locus of $f^{-1}$ is a proper closed subset of  $S$.   
In particular, we can define the strict transform $S_Y$ of $S$ in $Y$, and the preiamge $H_Y$ of $H$ in $S_Y$ under the morphism $f|_{S_Y}\colon S_Y \to S$. 
\item Up to refining $I$, there is a complex analytic orbispace $ \mathfrak{Y} = (Y_i,\overline{G}_i)_{i\in I}$ whose quotient space   is an open neighborhood of $S_Y$ in $Y$. 
\item Let $(S_Y)_i$ and $ (H_Y)_i$  be the preimages of $S_Y$ and $H_Y$ respectively in $Y_i$.   
Then $H_Y$ induces a snc orbi-divisor on the  complex orbifold $( (S_Y)_i, \overline{G}_i )$. 
There is some open neighborhood $U_i$ of $(S_Y)_i$ in $Y_i$, 
such that $U_i$ is isomorphic to an open neighborhood of $\{\mathbf{0}_3\} \times (S_Y)_i$ 
in the hypersurface in $\mathbb{\mathbb{D}}^3 \times (S_Y)_i$  defined by a function as in Proposition \ref{prop:good-shape} with respect to the pair $((S_Y)_i, (H_Y)_i)$.  
\item Let $\Delta_Y$ be the exceptional locus  of $f$,   let $(\Delta_Y)_i$ be its preimage in $Y_i$ and  
let $U'_i=\cap_{g\in \overline{G}_i} g(U_i)$. 
Then $U'_i$, $(\Delta_Y)_i$, $(S_Y)_i$,   $(H_Y)_i$ and  $\overline{G}_i$ satisfy the items (1) and (2) of Setup \ref{setup-local-cover}.  
Moreover, the collection   $(U'_i,\overline{G}_i)_{i\in I}$  induces a complex analytic orbispace,  
whose quotient space $U'$  is an open neighborhood of $S_Y$.  
\item  The divisorial critical  locus of $Y_i \to Y$ is contained in  $\Delta_Y$. 
\end{enumerate}
\end{lemma}

\begin{proof}
We will first perform the construction of Proposition \ref{prop:good-shape} on $X_i$. 
Let $D^1,...,D^k$ be the irreducible components of $H$.
By the item (3) of Lemma \ref{1-construction},  $H\subseteq S$ induces a snc orbi-divisor on the  complex orbifold $( S_i, G_i )_{i\in I}$.  
Hence for each irreducible component $D^j$, its preimage $D_i^j$ in $X_i$ is smooth. 
Since any two irreducible components of $H_i$ intersect by the item (1) of Lemma \ref{2-construction}, we deduce that if  $D^j_i$ is not empty, then it is an irreducible component of $H_i$. 

\textit{Case 1.}  
We   assume that  the singularities of $X_i$ at $S_i\setminus H_i$ are of   $A$-type or of $E$-type.  
Then for each $i\in I$, we perform the following sequence of blowups on $X_i$,  
\[
Z_i = X_{i,k} \to \cdots \to X_{i,1} \to  X_{i,0}  = X_i 
\]
where $ X_{i,j+1}  \to  X_{i,j} $ is the blowup of $X_{i,j}$ at $D^{j+1}_i$ for $N$ times,  
as in Subsection \ref{subsection:blowup-divisor}. 
Then there is a natural action of $G_i$ on $Y_i$. 
Moreover, the family $(Z_i,G_i)_{i\in I}$ induces a complex analytic orbispace $\mathfrak{Z}$ with quotient space $Y$. 
If $\psi_i\colon Z_i \to X_i$ is the composition of the sequence above, then the collection $(\psi_i)_{i\in I}$
induces a proper bimeromorphic morphism $f\colon Y \to X$. 
In addition, since $\psi_i$ is the composition of blowups whose centers are compatible along the overlaps, 
we deduce that there is some $\psi_i$-ample and $\psi_i$-exceptional divisor which is compatible along the overlaps. 
Hence   there is a $f$-exceptional and $f$-ample divisor.  
Let $(S_Z)_i$ be the strict transform of $S_i$ in $Z_i$.  
By construction, there is some open neighborhood $U_i$ of $(S_Z)_i$, such that $U_i$ is isomorphic to an open neighborhood of $\{\mathbf{0}_3\} \times (S_Z)_i$ in the hypersurface in $\mathbb{\mathbb{D}}^3 \times (S_Z)_i$  defined by a function as in Proposition \ref{prop:good-shape}.  
Since the intersection of the divisorial critical locus of $X_i \to X$ and $S$ is contained in $H$, up to enlarging $N$, we can assume that the  divisorial critical locus of $Z_i \to Y$  is contained in $\Delta_Y$. 
We set $\widehat{G}_i = G_i$  and $\widehat{V}_i = X_i$ in this case. 

\textit{Case 2.} 
We assume that  the singularities of $X_i$ at $S_i\setminus H_i$ are of  $D$-type. 
We will proceed as in Lemma \ref{lemma:blowup-equation-D} on each $X_i$. 
Let $\varphi_i\colon W_i\to X_i$ be the composition of blowups at the centers $D^1_i,..., D^k_i$ successively.  
Then there is a natural action of $G_i$ on $W_i$ and $(W_i,G_i)_{i\in I}$ induces a complex analytic orbispace $\mathfrak{W}$ with quotient space $W$.   
Furthermore, there is an induced projective bimeromorphic morphism $\varphi\colon W\to X$ with a 
$\varphi$-ample and $\varphi$-exceptional divisor. 
Let $S_W$ be the strict transform of $S$ in $W$, 
and let  $(S_W)_i$ be the strict transforms of $S_i$ in $W_i$.  
Let  $H_W$ be the preimage of $H$ in $S_W$, and let $(H_W)_i$ the preimage of $H_W$ in $W_i$.  
Let $\mathcal{I}(D^j_i)$ be the ideal sheaf of $D^j_i$ in $X_i$.   
There is an open  neighborhood $V_i$ of $(S_W)_i$ such that $\varphi_i^{-1} \mathcal{I}(D^j_i) \cdot \mathcal{O}_{V_i}$ is generated  by a holomorphic function  $\sigma^j$ on $V_i$, for $j=1,...,k$, see Remark \ref{rmk:function-exc-div} and Remark \ref{rmk:finite-equation}. 
Shrinking $V_i$ we may assume that it is $G_i$-invariant.
We note that  for every $j=1,...,k$, the collection of subspaces $(D^j_i)_{i\in I}$ is compatible along the overlaps. 
By applying Lemma \ref{2-construction-2} for $k$ times,   
up to shrinking $V_i$ and up to refining the family $I$, 
we can assume the following properties.  
Let $\widehat{V}_i = V_i[\sqrt[6]{\sigma_1},...,\sqrt[6]{\sigma_k}]$.  
Then  the  morphism $\widehat{V}_i \to Z$ is finite and Galois over its image.  
Let $\widehat{G}_i$ be its Galois group. 
The collection $(\widehat{V_i}, \widehat{G_i})_{i\in I}$ induces a complex analytic orbispace 
$\widehat{\mathfrak{V}}$, 
whose  quotient space $\widehat{V}$ is  an open neighborhood of $S_W$ in $W$. 
Let $(S_{\widehat{V}})_i$ be the preimage of $(S_W)_i$ in $\widehat{V}_i$. 
Then the finite morphism $(S_{\widehat{V}})_i \to  (S_W)_i$ is a cyclic cover, 
branched over the components of $(H_W)_i$.   
Indeed, $(S_{\widehat{V}})_i$ is also isomorphic to a contractible neighborhood of the origin of $\mathbb{C}^n$, and the preimage $(H_{\widehat{V}})_i$ of $(H_W)_i$ in 
$(S_{\widehat{V}})_i$ is isomorphic to the union of some coordinate hyperplanes. 
In particular,  $H_W$  induces a snc orbi-divisor on the complex orbifold $((S_{\widehat{V}})_i, \widehat{G}_i)_{i\in I}$.

Now, as in the Case 1, we  blow up $\widehat{V_i}$ at the  preimage of each irreducible component of $H_W$ for $N$ times as in Subsection \ref{subsection:blowup-divisor}. 
We obtain a complex analytic orbispace $ \mathfrak{Z} = (Z_i,\widehat{G}_i)$ with quotient space $Y'$. 
Let $(S_Z)_i$ be the strict transform of $(S_{\widehat{V}})_i$ in $Z_i$.
By  the construction of Lemma \ref{lemma:blowup-equation-D}, 
there is some open neighborhood $U_i$ of $(S_Z)_i$ in $Z_i$, such that
$U_i$ is isomorphic to an open neighborhood of $\{\mathbf{0}_3\} \times (S_Z)_i$ 
in the hypersurface in $\mathbb{\mathbb{D}}^3 \times (S_Z)_i$  defined by a function as in Proposition \ref{prop:good-shape} with respect to the pair $((S_Z)_i, (H_Z)_i)$. 
Furthermore, there is a proper bimeromorphic morphism $h\colon  Y'\to \widehat{V}$ with some $h$-ample and $h$-exceptional divisor. 
Since the indeterminacy locus of $h^{-1}$ is contained in the compact set $S_{W}$, we deduce that $h$ extends to a projective bimeromorphic morphism $h\colon Y \to W$ by Lemma \ref{lemma:extend-morphism-1}. 
Hence we obtain a projective bimeromorphic morphism $f\colon Y \to X$, with a $f$-exceptional and $f$-ample divisor.  
Furthermore, up to enlarging $N$,  
we can assume that the   divisorial critical locus of $Z_i \to Y$  is contained in   $\Delta_Y$.
\\

In both cases, we have constructed a projective bimeromorphic morphism $f\colon Y\to X$, and a complex analytic orbispace $\mathfrak{Z}= (Z_i,\widehat{G}_i)$, 
whose quotient space, denoted by $Z$, is an open neighborhood of the strict transform $S_Y$ of $S$ in $Y$. 
By construction, the items (1) and (2) of the lemma hold  for $f$. 
Furthermore, the items (3) and (4) hold for   $\mathfrak{Z}= (Z_i,\widehat{G}_i)$.  
However, in the item (5), there is an issue which is the item (2) of Setup \ref{setup-local-cover}. 
To handle it, we apply the method of cyclic covers as in Lemma \ref{lemma:div-reduction}. 
Since $\psi_i\colon Z_i \to \widehat{V}_i$ is obtained by blowing up the components $(D_{\widehat{V}})_i^j$ of $(H_{\widehat{V}})_i$, 
we can proceed as in the first paragraph of the Case 2 above. 
By Lemma \ref{2-construction-2}, up to refining the family $I$, 
there is a  $\widehat{G}_i$-invariant  neighborhood $V'_i$ of $(S_Z)_i$ in $Z_i$, and a cyclic covering space  
$Y_i =  V'_i[ \sqrt[6]{\delta_1},..., \sqrt[6]{\delta_k}]$, 
such that $\delta_j$ is a generator of $\psi_i^{-1}\mathcal{I}((D_{\widehat{V}})_i^j) \cdot \mathcal{O}_{V'_i}$, 
and that $Y_i \to V'_i/\widehat{G}_i$ is Galois of group $\overline{G}_i$.   
Then the collection $(Y_i, \overline{G}_i)_{i\in I}$ defines a complex analytic orbispace $\mathfrak{Y}$,  
whose quotient space is an open neighborhood of $S_Y$. 
Moreover, by  Lemma \ref{lemma:div-reduction} and  Remark \ref{rmk:function-exc-div}, 
the item (4) holds on $Y_i$,  
and the vanishing orders of the functions $\alpha$ and $\beta$ in the defining function of $U_i$ in $\mathbb{D}^3\times (S_Y)_i$  are divisible by 6 along each component of $(H_Y)_i$. 
We notice that the item (1) of Setup \ref{setup-local-cover} is satisfied by the construction of $U_i'$. 
Hence we obtain the item (5) of the present lemma.

The  divisorial critical locus of $Y_i \to V'_i$ is contained in the $\psi_i$-exceptional locus, hence is contained in the preimages of $\Delta_Y$ in $V'_i$. 
Since the  divisorial critical locus of $V'_i \to Y$ is contained in $\Delta_Y$, 
we deduce the item (6). 
This completes the proof of the lemma.  
\end{proof}

The following lemma will be applied on the  neighborhood $U'$ of $S_Y$ in the item (5) of   Lemma \ref{3-construction}.

\begin{lemma}
\label{4-construction}
With the notation of Setup \ref{setup-local-proof}, 
let  $\Delta$ be a reduced divisor in $X$ and let $\Delta_i$ be its preimage in $X_i$.  
Suppose that $\Delta$ contains the divisorial critical locus of $X_i\to X$, 
and that $H$ induces a snc orbi-divisor on the  complex orbifold $(S_i,G_i)_{i\in I}$. 
Assume that the data $X_i$, $\Delta_i$, $S_i$, $H_i$ and $G_i$ satisfy  the items (1) and (2) of Setup \ref{setup-local-cover}. 
Then, there is a projective bimeromorphic morphism $f\colon Y\to X$ with $Y$ normal such that the following properties hold. 
\begin{enumerate}
\item There is a $f$-exceptional and $f$-ample divisor. 
\item The indeterminacy locus of  $f^{-1}$ is a proper closed subset of $S$. 
In particular, we can define the strict transform $S_Y$ of $S$ in $Y$, and the preiamge $H_Y$ of $H$ in $S_Y$ under the morphism $f|_{S_Y}\colon S_Y \to S$.  
\item There is  a complex analytic orbispace $\overline{\mathfrak{W}}  = (\overline{W}_k, \overline{G}_k, \psi_k)_{k\in K }$, 
whose quotient space $W$ is an open neighborhood of $S_Y$ in $Y$. 
\item The divisorial critical locus of $\psi_k\colon \overline{W}_k \to W$ is contained in $\Delta_Y$, where $\Delta_Y = f^{-1}(\Delta)$.
\item $\overline{W}_k$ is smooth on $\psi_k^{-1} (W\setminus \Delta_Y)$.  
\end{enumerate} 
\end{lemma}

\begin{proof} 
We will apply Proposition \ref{prop:construct-local-cover} on each $X_i$. 
Since the construction is  local around a  point, 
there may be infinitely many constructions to consider. 
However, since $S$ is compact,  
we can reduce to applying the construction for only finitely times. 

More concretely,  for each point $x \in S_i$, we let $S_{x}$ be an open neighborhood of it as the one in the item (4) of Setup \ref{setup-local-cover}.   
In particular, if $(G_i)_x\subseteq G_i$ is the subgroup which fixes $x$,  
then $(S_x,(G_i)_x)$ is an orbifold chart of $\pi_i(S_x)$, compatible with the orbifold $(S_i,G_i)_{i\in I}$.  
There is a positive integer $N_x$, 
such that  Proposition \ref{prop:construct-local-cover} holds  on $X_i$ with respect to $S_{x}$ and  $N_x$.
Since $S$ is compact, by considering   $S_{x}$  for all $x \in S_i$ and all $i\in I$, 
we can  obtain a finite family  $(x_k)_{k\in K}$ of points  and an application $\sigma\colon K \to I$ such that the following properties hold. 
Each point $x_k$ belongs to $S_{\sigma(k)}$ and  
\[\bigcup_{k\in K} \pi_{\sigma(k)} (S_{x_k}) =S.\]
Since $K$ is finite, we can define the positive integer $N:= \max\{ N_{x_k} \ | \ k\in K\}$.

Let $D^1,...,D^e$ be the irreducible components of $H$.
Since $H$ induces a snc orbi-divisor on the  complex analytic orbispace $(S_i,G_i)_{i\in I}$, if the preimage $D^j_i$  of $D^j$ in $X_i$  is not empty, 
then it is smooth and hence irreducible. 
Let $f_i \colon Y_i \to X_i$ be the morphism obtained by
blowing up of $X_i$ at  $D^j_i$  for $6N$ times,  successively for $j=1,...,e$, as in  Subsection \ref{subsection:blowup-divisor}. 
Then there is a natural action of $G_i$ on $Y_i$ and the collection $(Y_i,G_i)_{i\in I}$  induces a complex analytic orbispace $\mathfrak{Y}$ with quotient space $Y$.  
Furthermore, there is an induced projective bimeromorphic morphism $f\colon Y\to X$ with a $f$-ample and $f$-exceptional divisor.  
Let $S_Y$ be the strict transform of $S$ in $Y$. 

From the choice of $N$,  for each $k\in K$, there is an open neighborhood $W_k$ of the strict transform of $S_{x_k}$ in $Y_{\sigma(k)}$, which is invariant under $G_k:=(G_{\sigma(k)})_{x_k}$, 
such that there is a finite cover $\nu_k\colon \overline{W}_k \to W_k$ satisfying the conclusions of Proposition \ref{prop:construct-local-cover}.   
The property (4) of Proposition \ref{prop:construct-local-cover} implies that the natural  morphism $ \psi_k\colon  \overline{W}_k  \to Y$
is finite and Galois of group $\overline{G}_k$ over the normalization of its image.  
We remark that, up to shrinking $W_k$ around $S_{x_k}$, we can assume that $(W_k, G_k)$ is an orbifold chart over $Y$ compatible with $\mathfrak{Y}$.

Let $W$ be the union of the images of the $W_k$'s in $Y$. 
Then $W$ is the quotient space of the  complex analytic orbispace  $(W_k, G_k)_{k\in K}$ and is an open neighborhood of $S_Y$.  
By Lemma \ref{lemma:local-irr}, up to shrinking $W_k$, we can assume that  $W_k$ is locally irreducible. 
Hence $W$ is locally irreducible as well. 
Thus the normalization morphism   $W' \to  W$ is a homeomorphism, and we may identify $W'$ and $W$ as the same topological space. 
We will show that $(\overline{W}_k, \overline{G}_k)_{k\in K}$ induces a complex analytic orbispace, whose quotient space is  $W'$.  

Let $w\in  \psi_k(\overline{W}_k ) \cap \psi_l(\overline{W}_l ) \subseteq W$ for some $k,l\in K$. 
Then 
\[f(w)\in \pi_{\sigma(k)}(X_{\sigma(k)}) \cap \pi_{\sigma(l)}(X_{\sigma(l)}) \subseteq X.   \]
Hence there is an open neighborhood $X'$ of $f(w)$ in $X$, and an orbifold chart $(Z,G')$ of $X'$, such that there are  inclusions of orbifold charts $\zeta_k\colon Z\to X_{\sigma(k)}$  and $\zeta_l\colon Z\to X_{\sigma(l)}$. 
Let $(Z_Y,G')$ be the orbifold chart on $f^{-1}(X')$ induced by $(Z,G')$, 
which is compatible with $\mathfrak{Y}$. 
Then there is  an open neighborhood $U$ of $w$ in $\psi_k(\overline{W}_k ) \cap \psi_l(\overline{W}_l )$, and  an orbifold chart $(V,G)$ over $U$, such that there are inclusions of orbifold charts 
\[\iota_k \colon  V \to W_k, \  \iota_l \colon   V \to W_l  \mbox{ and }  V \to Z_Y. \]
Up to shrinking $U$, we may assume that $w$ has a unique preiamge $o$ in $V$.    
We set  \[\overline{V}_k = \nu_k^{-1}(\iota_k(V)) \subseteq \overline{W}_k \mbox{ and } \overline{V}_l = \nu_l^{-1}(\iota_l (V)) \subseteq \overline{W}_l.\] 
Then the finite morphism $\overline{V}_k \to \iota_k(V)$ is constructed by applying  Proposition \ref{prop:construct-local-cover} on  $ \zeta_k(Z) \subseteq  X_{\sigma(k)}$, 
and $\overline{V}_l \to \iota_l (V_l)$ is constructed by applying   Proposition \ref{prop:construct-local-cover} on $ \zeta_l(Z) \subseteq X_{\sigma(l)}$.  
By the  property (5) of Proposition \ref{prop:construct-local-cover}, we deduce that, up to shrinking $U$, 
every connected component of $\overline{V}_k$ is isomorphic to every connected component of $\overline{V}_l$. 
Lemma \ref{lemma:orbispace-def} then 
implies that $(\overline{W}_k, \overline{G}_k)_{k\in K}$ induces a complex analytic orbispace $\overline{\mathfrak{W}}$. 
Its  quotient space is exactly $W'$ since $W'$ and every $\overline{W}_k$ are  normal.  
Replacing $Y$ by its normalization, we can assume that   $W'= W$,  which is  an open neighborhood of $S_Y$ in $Y$.

It remains to prove the items (4) and (5).   
We notice that the divisorial critical locus of $W_k\to W$ is contained in $\Delta_Y$,  as the divisorial critical locus of $X_{\sigma(k)}\to X$ is contained in $\Delta$. 
Since the divisorial locus of $\overline{W}_k \to W_k$ is contained in the preimage of $\Delta_Y$ by construction, we deduce the item (4).  
The item (5) follows from the item (3) of Proposition \ref{prop:construct-local-cover}.
This completes the proof of the lemma. 
\end{proof}

Now we can complete the proof of Proposition \ref{prop:local-proof-1}. 

\begin{proof}[{Proof of Proposition \ref{prop:local-proof-1}}]
Let  $\mathfrak{X}=(X_k,G_k)_{k\in K}$ be the  complex analytic orbispace induced by the index-one covers of $\omega_X$.  
Then $X_k \to X$ is quasi-\'etale onto its image, and $X_k$ has at most canonical singularities.  
Let $S_k$ be the preimage of $S$ in $X_k$.  
We note that $S_k$ may not be irreducible. 
Let $h_k\colon X'_k \to X_k$ be the bimeromorphic projective morphism induced by the functorial  desingularization of $S_k$, see the beginning of Subsection \ref{subsection:blowup}. 
By construction, the strict trasnform $S'_k = (h_k^{-1})_* S_k$ is smooth.  
There is a natural action of  $G_k$  on $X'_k$ and we obtain a complex analytic orbispace
$\mathfrak{X}'=(X'_k,G_k,\mu_k)_{k\in K}$  with quotient space $X'$. 
Furthermore, there is a natural projective bimeromorphic morphism $h\colon X'\to X$, with a $h$-ample and $h$-exceptional divisor.   
We remark that the indeterminacy locus of $h^{-1}$ is a proper closed subset of $S$. 
Moreover, the divisorial critical locus of $\mu_k$ is contained in the $h$-exceptional locus.

First we assume that  $X_k$ is smooth at general points of the components of $S_k$.  
Then  $X'_k$ is smooth around general points of the components of $S'_k$. 
Let $Y_k\to X'_k$ be the projective bimeromorphic morphism as in Corollary \ref{cor:blowup-hypersurface-sing}, with respect to $S'_k$. 
If $(S_Y)_k$ is the strict transform of $S'_k$ in $Y_k$, then $Y_k$ is smooth around $(S_Y)_k$. 
Up to taking the normalization, we can assume that every $Y_k$ is normal.
Since the  constructions are functorial, there is a natural action of $G_k$ on $Y_k$, 
and  we have a complex  analytic orbispace $\mathfrak{Y} = (Y_k,G_k,\psi_k)_{k\in K}$ with quotient space $Y$, which is normal.  
Furthermore, there is a natural projective bimeromorphic morphism $f\colon Y\to X$.  
Since $Y_k$ is smooth around $(S_Y)_k$, there is a smooth $G_k$-invariant open neighborhood $W_k$ of  $(S_Y)_k$.   
Then  $(W_k, G_k, \psi_k)_{k\in K}$ is a  complex orbifold, whose quotient space is an open neighborhood $W$ of $S_Y$. 
Hence $f$ satisfies the properties of the proposition.

In the remainder of the proof, we assume that the irreducible components of $S_k$ are irreducible components of the singular locus of $X_k$.  
Let $S'$ be the strict transform of $S$ in  $X'$. 
Then $\mu_k^{-1}(S') = S_k'$ is smooth.  
Therefore, for a point  $o\in S'$, there is an orbifold chart $(U_o,G_o)$ over a  neighborhood of $o$ in $X'$, compatible with the complex analytic orbispace $\mathfrak{X}'$, such that the preiamge of $S'$ in $U_o$ is an irreducible component of the singular locus of $U_o$. 
Since $S'$ is compact, by considering  the $(U_o,G_o)$'s for all points $o\in S'$, we obtain  a complex analytic orbispace $\mathfrak{Z} =  (Z_j, G'_j, \nu_j)_{j\in J}$ with a finite family $J$, whose quotient space is an open neighborhood $Z$ of $S'$ in $X'$. 
Moreover, each orbifold chart $ (Z_j, G'_j)$ is compatible with   $\mathfrak{X}'$.
There is a proper closed subset $H'\subseteq S'$, such that  $Z_j$ has the same type of canonical singularities at points of $\nu_j^{-1}(S') \setminus \nu_j^{-1}(H')$.  
We can pick $H'$ so that it contains the intersection of $S'$  and the $h$-exceptional locus. 
It follows that the intersection of the divisorial critical locus of $\nu_j$  and $S'$ is contained in   $H'$. 
Therefore the data $Z$, $\mathfrak{Z}$, $S'$ and $H'$ satisfy the conditions of Setup \ref{setup-local-proof}.

We apply Lemma \ref{1-construction} on $Z$ and obtain a projective bimeromorphic morphism 
$h_1\colon Z_1\to Z$. 
Since the indeterminacy locus of $h_1^{-1}$ is contained in $S'$, which is compact, 
we deduce from Lemma \ref{lemma:extend-morphism-1}  that  $h_1$ extends to a proper bimeromorphic morphism $Y_1\to X'$.  
Let $f_1\colon Y_1 \to X$ be the natural morphism. 
Then  there is a $f_1$-ample and $f_1$-exceptional divisor. 
Let $S_{Y_1}$ be the strict transform of $S'$ in $Y_1$, and let $H_{Y_1}$ be the reduced divisor in $S_{Y_1}$ as in the item (3) of Lemma \ref{1-construction}.   
By applying  Lemma \ref{2-construction}  around $S_{Y_1}$, we can  obtain an open neighborhood $Q_1$ of it, so that we can   apply   Lemma \ref{3-construction} to $Q_1$. 
Then we   obtain a projective bimeromorphic morphism 
$g_2\colon P_2 \to Q_1$. 
By Lemma \ref{lemma:extend-morphism-1} again,  $g_2$ extends to a  projective bimeromorphic  morphism $g_2\colon Y_2 \to Y_1$. 
Let  $f_2\colon Y_2 \to X$ be the induced morphism.

Let $S_{Y_2}$ be the strict transform of $S_{Y_1}$ in $Y_2$,   
let $Q_2$ be the open neighborhood of $S_{Y_2}$ as in the item (5) of 
Lemma \ref{3-construction}, and let $\Delta_2$ be the whole reduced $g_2$-exceptional divisor.  
We remark that the intersection of the $f_1$-exceptional set  and $S_{Y_1}$ is contained in $H_{Y_1}$. 
Thus, up to blowing up more times in the construction of Lemma \ref{3-construction}, 
we can assume that the strict transforms of the $f_1$-exceptional divisors in $Y_2$ are disjoint from $Q_2$.   
In particular, if $\Gamma_2$ is the whole reduced $f_2$-exceptional divisor, then $\Gamma_2\cap Q_2= \Delta_2 \cap Q_2$.

We can apply Lemma \ref{4-construction} to $Q_2$ with the divisor $\Gamma_2\cap Q_2= \Delta_2 \cap Q_2$, and  obtain a projective bimeromorphic morphism 
$g_3\colon P_3 \to Q_2$.    
By Lemma \ref{lemma:extend-morphism-1}, we deduce that $g_3$ extends to a projective bimeromorphic morphism $g_3\colon Y_3 \to Y_2$.  
Let $Y=Y_3$ and let $f\colon Y \to X$ be the natural morphism.  
Then $f$ satisfies all the properties of the proposition. 
\end{proof}

\section{Proof of the modification theorem}
\label{section:modification}

We will complete the proof of Theorem \ref{thm:main-thm} in this section.

\begin{lemma}
\label{lemma:dlt-case-finial-step}
Let $X$ be a normal complex analytic variety, let $S\subseteq X_{\sing}$ be a compact irreducible component of codimension $2$, and let  $\Delta$ be a reduced divisor in $X$. 
Suppose  that  $\Delta$ does not contain $S$. 
Assume  that there is a normal complex analytic orbispace $\mathfrak{W} = (W_i, G_i, \psi_i)_{i\in I}$ 
satisfying the following properties. 
\begin{enumerate}
\item The quotient space $W$ of $\mathfrak{W}$ is an open neighborhood of $S$ in $X$. 
\item The divisorial critical locus of $\psi_i\colon W_i \to W$ is contained in $\Delta$.  
\item $W_i$ is smooth on $\psi_i^{-1} (W\setminus \Delta)$. 
\end{enumerate}

Assume that there is a projective bimeromorphic morphism  $h\colon (Z,\Gamma) \to (X,\Delta)$  satisfying the following properties.  
The divisor $\Gamma$ is the sum of  $h^{-1}_*\Delta$ and the whole reduced $h$-exceptional divisor.  
The $h$-exceptional locus is contained in $\Gamma$.  
The pair $(Z,\Gamma)$ is  dlt and each irreducible component of $\Gamma$ is $\mathbb{Q}$-Cartier. 
The morphism $h$ is an isomorphism 
over general points of $S$.  
The pair  $(Z,\Gamma)$ is  snc on $h^{-1}(X\setminus S)$.  

Then there is a projective bimeromorphic morphism $\varphi\colon Y \to Z$ with a $\varphi$-ample and $\varphi$-exceptional divisor, such that   $Y$ has quotient singularities. Furthermore,    $\varphi$  is an isomorphism over  $h^{-1}(X\setminus S)$ and over general points of  the strict transform of $S$ in $Z$. 
\end{lemma}

\begin{proof}  
Let $V =h^{-1}(W)$  and let $V_i$ be the normalization of the main component of  $W_i \times_W  V$.  
Then there is a natural action of $G_i$ on $V_i$. 
\begin{equation*}
\xymatrix{ 
		V_i \ar[d]	\ar[r]^{\mu_i}  & V    \ar[d]^h  \\ 
        W_i  \ar[r]^{\psi_i}  & W   
}
\end{equation*}
Since $V$ is normal,  the family $(V_i,G_i)_{i\in I}$ induces a complex analytic orbispace whose quotient space is $V$. 
Since $Z\setminus \Gamma $ is isomorphic to an open subset of  $ X\setminus \Delta$, 
we deduce from the condition (2) that, 
if $\mu_i\colon V_i \to V$ is the natural   morphism, 
then its divisorial  critical locus  is contained in $\Gamma$. 
In addition, the condition (3) implies that  
$V_i$ is smooth on $ \mu_i^{-1} (V\setminus \Gamma)$.

Let $S_Z$ be the strict transform of $S$ in $Z$.  
Let $\Gamma^j$ with $j=1,...,l$ be the irreducible components of $\Gamma$, and let $r^j$ be the Cartier index of $\Gamma^j$. 
For every point $z\in h^{-1}(S)$, there is an open neighborhood $T$ of $z$ in $V$ such that each 
$r^j\Gamma^j|_{T}$ is  defined by some holomorphic function $\sigma^j$. 
Thus 
we obtain a  family $(U_k)_{k\in K}$ of   open subsets of $V$, 
whose union contains $h^{-1}(S)$, 
such that $r^j\Gamma^j|_{U_k}$ is connected and  defined by some holomorphic function $\sigma^j_k$. 
Furthermore,  there is an application $\iota\colon K\to I$ such that  
$U_k$ is contained in  $\mu_{\iota(k)}(V_{\iota(k)})$.   
By convention, we set $\sigma_k^j=1$ if $\Gamma^j|_{U_k}$ is empty. 
Let $U=\bigcup_{k\in K} U_k$. 
Then it is a neighborhood of $h^{-1}(S)$  in $V$.

We notice that, since $(Z,\Gamma)$ is dlt, each $\Gamma^j$ is normal. 
It follows that if  $\Gamma^j|_{U_k}$ is not empty, then it is normal and hence irreducible. 
In particular, we can apply  the result from Lemma \ref{lemma:cyclic-functorial} to Lemma \ref{lemma:local-dlt-cover} on each pair $(U_k,\Gamma|_{U_k})$. 
Let $M>0$ be an integer which is divisible by the degrees of all the $\mu_i$'s. 
We define  $ \overline{U}_k$ as a connected component of 
\[
U_k[\sqrt[Mr^1]{\sigma_k^1},..., \sqrt[Mr^l]{\sigma_k^l}]^{nor}
\]
and we let $\xi_k\colon \overline{U}_k \to U_k$ be the natural morphism.  
Then the $\overline{U}_k$'s define a normal complex analytic orbispace with quotient space $U$, by  Lemma \ref{lemma:cyclic-functorial} and Lemma \ref{lemma:orbispace-def}.   
Let $Q_k$ be a connected component of the normalization of  $\overline{U}_k \times_{Z}  V_{\iota(k)}$.  
It follows that the natural finite morphism $Q_k\to U_k$ is Galois of Galois group $H_k$, and the collection $(Q_k,H_k)_{k\in K}$ induces a normal complex analytic orbispace $\mathfrak{Q}$ with quotient space $U$.  
\begin{equation*}
\xymatrix{ 
		Q_k \ar[d]_{\eta_k}	\ar[rr]^{ }  & & V_{\iota(k)}    \ar[d]  \\ 
        \overline{U}_k  \ar[r]^{\xi_k}  & U_k \ar[r]& Z  
}
\end{equation*}
We denote by $\eta_k \colon Q_k \to \overline{U}_k$  the natural morphism.   
Then by Lemma \ref{lemma:local-dlt-cover}, we obtain that  $\eta_k$ is quasi-\'etale;   
furthermore,  $Q_k$ is smooth in codimension 2 since $V_{\iota(k)}$ is smooth over $U_k\setminus \Gamma$.

The  Zariski's  purity theorem  implies that \[(Q_k)_{\sing}  \subseteq (\eta_k)^{-1}((\overline{U}_k)_{\sing}).\] 
Since $(Z,\Gamma)$  is  snc on $Z\setminus h^{-1}(S)$,  we deduce that  $\overline{U}_k$ is smooth over $Z\setminus h^{-1}(S)$ by Lemma \ref{lemma:cyclic-cover-Cartier}.  
It follows that 
\[(Q_k)_{\sing}  \subseteq  (\xi_k \circ \eta_k)^{-1}(h^{-1}(S)).  \]

Let $\widetilde{Q}_k \to Q_k$ be the functorial desingularization, which is an isomorphism over the smooth locus of $Q_k$. 
Then there is a natural action of $H_k$ on $\widetilde{Q}_k$, and we obtain a complex orbifold  
$\widetilde{\mathfrak{Q}} = (\widetilde{Q}_k , H_k)_{k\in K}$ with quotient space $\widetilde{Q}$.  
In particular, $\widetilde{Q}$ has quotient singularities only. 
There is an induced  projective bimeromorphic morphism $\varphi \colon \widetilde{Q} \to U$. 
From the properties of $(Q_k)_{\sing}$ above, we see that  the indeterminacy locus of $\varphi^{-1}$ is contained in $h^{-1}(S)$ and is of codimension at least 3 in $Z$. 
In particular, $\varphi $ is an isomorphism over general points of $S_Z$. 
Furthermore, there is an $\varphi $-exceptional and $\varphi $-ample divisor.  
Thus by Lemma \ref{lemma:extend-morphism-1}, $\varphi $ extends to a projective bimeromorphic morphism  
$\varphi\colon Y\to Z$. 
Since $Z$ is smooth outside $h^{-1}(S)$, we deduce that $Y$ has quotient singularities only.  
This completes the proof of the lemma.
\end{proof}

Now we can deduce Theorem \ref{thm:main-thm}.

\begin{proof}[{ Proof of Theorem \ref{thm:main-thm} }]
We will divide the proof into several steps. 
Let $S_e$ with $e=1,...,r$ be all the irreducible components of $X_{\sing}$ which have codimension 2 in $X$. 

\textit{Step 1.} In this step, we will reduce  to the case when $S_1,..., S_r$ are pairwise disjoint. 
By  blowing up at centers contained in at least two of the $S_e$'s, 
we obtain a projective bimeromorphic morphism $h\colon Z \to X$ such that the strict transforms $h^{-1}_* S_1,...,h^{-1}_* S_r$ are pairwise disjoint.  
Let $X' \to Z$ be the projective bimeromorphic morphism constructed in Lemma \ref{lemma:QcDV-modification}, with respect to the collection $\{h^{-1}_*S_1,...,h^{-1}_*S_r\}$. 
We observe that the indeterminacy locus of $X\dashrightarrow X'$ has codimension at least 3. 
Up to replacing $X$ by $X'$, we may assume that $S_1,..., S_r$  are pairwise disjoint.
Furthermore, $X$ has klt singularities.

\textit{Step 2.}   In this step, we will reduce  to the situation of Lemma \ref{lemma:dlt-case-finial-step}. 
Let $X_e$ be an open neighborhood of $S_e$ in $X$ such that the $X_e \cap X_{e'} = \emptyset$ if $e\neq e'$.  
We apply Proposition \ref{prop:local-proof-1} to each $X_e$. 
Then we have projective bimeromorphic morphisms $p_e\colon X'_e\to X_e$, which satisfies the properties of Proposition \ref{prop:local-proof-1}.  
Since for each $e$, there is a $p_e$-ample and $p_e$-exceptional divisor, and since the indeterminacy locus of $(p_e)^{-1}$ is contained in $S_e$ and hence compact,  
by applying Lemma \ref{lemma:extend-morphism-1} for several times, there is a projective bimeromorphic morphism $p\colon X' \to X$ such that
$p=p_e$ over $X_e$.  
Let $\Delta $ be the whole reduced $p$-exceptional divisor. 
Then  $p(\Delta)$ has codimension at least 3 in $X$.   

In the remainder of the proof, by abuse of notation, we denote $X'$ by $X$, and $X_e'$ by $X_e$. 
If $\Delta_e= \Delta|_{X_e}$, then $X_e$, $\Delta_e$ and $S_e$ satisfy the conditions (1)-(3) of Lemma \ref{lemma:dlt-case-finial-step}.  
Furthermore, from now on, it is enough to construct $f\colon Y\to X$, such that $Y$ has quotient singularities, 
and $f$ is an isomorphism over general points of each $S_e$ and over $X_{\sm}\setminus \Delta$.  
We remark that $X_{\sm}\setminus \Delta$ is contained in the snc locus of $(X,\Delta)$.

\textit{Step 3.} In this step, we will construct a bimeromorphic morphism $h\colon Z\to X$, 
which plays the same role  for $X_e$    
as $Z$ for $X$ in   Lemma \ref{lemma:dlt-case-finial-step}.  
We first note that, by the property (3) of Lemma \ref{lemma:dlt-case-finial-step}, $X$ has quotient singularities around $S\setminus \Delta$, 
where $S= \bigcup_{e=1}^r S_e$. 
Thus, we  can apply Corollary \ref{cor:partial-resolution} to $(X,\Delta)$, with respect to $S$,  
and  obtain  a projective  bimeromorphic morphism $\rho \colon X'\to X$ with $X'$ normal.  
By construction, $\rho$ is an isomorphism over the snc locus of $(X,\Delta)$ and over general points of each $S_e$; 
if $\Delta'$ is the  sum of  $\rho^{-1}_*\Delta$  and the whole reduced  $\rho$-exceptional divisor,   
then $(X',\Delta')$ is  snc on $\rho^{-1}(X\setminus S)$;  
furthermore, $\Delta'$ is the support of some Cartier divisor, and $X'\setminus \Delta'$ has quotient singularities. 
Let $S'$  be  the strict transform of $S$ in $X'$.    
Let $(Z,\Gamma) \to (X', \Delta')$ be a $\mathbb{Q}$-factorial dlt modification as in Lemma \ref{lemma:dlt-modification}, 
which is an isomorphism over the  snc locus of $(X', \Delta')$,  
and over general points of each component of $S'$.
We denote by $h\colon Z\to X$ the natural morphism. 
Then we have the following properties.  
The divisor $\Gamma$ is the sum of  $h^{-1}_*\Delta $ and the whole reduced $h$-exceptional divisor.  
The $h$-exceptional locus  is  contained in $\Gamma$. 
Each irreducible component of $\Gamma$ is $\mathbb{Q}$-Cartier. 
The morphism $h$ is an isomorphism over the  snc locus of $(X,\Delta)$ and over general points of each component of  $S$.  
The pair  $(Z,\Gamma)$ is  snc on $h^{-1}(X\setminus S)$.  

\textit{Step 4.} We complete the proof in this final step. 
Let $Z_e=h^{-1}(X_e)$ and $\Gamma_e=\Gamma|_{Z_e}$.  
Then  
\[Z_{\sing} \subseteq h^{-1}(S)  \subseteq \bigcup_{e=1}^r Z_e. \]
We can  apply  Lemma \ref{lemma:dlt-case-finial-step}  on each $X_e$, and obtain a projective bimeromorphic morphism  $\varphi_e\colon Y_e \to Z_e$ as in the Lemma \ref{lemma:dlt-case-finial-step}.   
The indeterminacy locus of $(\varphi_e)^{-1}$  is contained in $h^{-1}(S_e)$, hence is compact.
Since there is a $\varphi_e$-ample and $\varphi_e$-exceptional divisor, by applying Lemma \ref{lemma:extend-morphism-1} for several times,
we deduce that there is a projective bimeromorphic morphism $\varphi\colon Y \to  Z$ such that its restriction over $Z_e$ is equal to $\varphi_e$, 
and that it is an isomorphism elsewhere. 
In particular, $Y$ has quotient singularities. 
Let $f\colon Y \to X$ be natural morphism. 
Then $f$ is an isomorphism over the  snc locus of $(X,\Delta)$ and over general points of  $S_1,...,S_r$.  
This completes the proof of the theorem. 
\end{proof}

\begin{remark}\label{rmk:exceptional-ample-divisor}
We may improve the construction above so that there is some $f$-exceptional and $f$-ample divisor, by the method of \cite{KollarWitaszek}.   
Let $f_1\colon Y_1 \to X$ be a morphism constructed in Theorem \ref{thm:main-thm},  and let $\mathcal{H}$ be a $f_1$-ample line bundle.   
We set  $\mathcal{L} = ((f_1)_*\mathcal{H})^*$ and $Y_2 = \mathrm{Proj} \bigoplus_{m\ge 0} \mathcal{L}^{\otimes m}$. 
Let $W$ be the normalization of the main component of $Y_1\times_X Y_2$.   
Then the indeterminacy locus of the natural map $X\dashrightarrow W$ is contained to the one of $(f_1)^{-1}$.  
We recall that $S_1,...,S_r$ are the irreducible components of $X_{\sing}$ which have codimension $2$. 
We apply  Lemma \ref{lemma:QcDV-modification}  to $W$ with respect to the strict transforms of $S_1,...,S_r$, and obtain a morphism $V\to W$.  
Afterwards, we apply Theorem \ref{thm:main-thm} to $V $ to obtain a modification $Y\to V$. 
Then we claim that the  natural morphism  $f\colon Y\to X$ satisfies the properties of Theorem \ref{thm:main-thm}.  
In fact, assume by contradiction that there is   a component $Z$ of the indeterminacy locus of $f^{-1}$, which has codimension $2$. 
Then the strict transform $Z_W$ of $Z$ in $W$ is well-defined. 
Since the indeterminacy locus of $V\dashrightarrow Y$ has codimension at least 3, we see that $Z_W$ is contained in the indeterminacy locus of $W\dashrightarrow V$. 
By the construction of $V$, see Lemma \ref{lemma:QcDV-modification} (2), 
 $Z$ is not equal to anyone of the subvarieties $S_1,...,S_r$.  
Therefore, $Z$ meets the smooth locus of $X$, and hence $Z_W$ meets $W_{\sm}$. 
This contradicts Lemma \ref{lemma:QcDV-modification} (1).  
In the end, by using Lemma \ref{lemma:excep-ample-divisor} below, we can adapt the  argument of  \cite[Lemma 8]{KollarWitaszek} to show  that,  
there is some $f$-exceptional and $f$-ample divisor.   
\end{remark}

\begin{lemma}
\label{lemma:excep-ample-divisor}  
Let $f\colon Y\to X$ be a projective bimeromorphic morphism between compact normal complex analytic singularities.  
Assume that $\mathcal{L}$ is a line bundle on $Y$,  
such that $\mathcal{L}|_{Y\setminus \Ex(f)}$ is trivial, 
where $\Ex(f)$ is the $f$-exceptional locus.  
Then there is some $f$-exceptional Cartier divisor $D$ such that $\mathcal{L}\cong \mathcal{O}_Y(D)$. 
\end{lemma}

\begin{proof}
We set $\mathcal{I} = f_*\mathcal{L}$. 
Then $\mathcal{I}$ is a torsion-free coherent sheaf on $X$  and $\mathcal{I}^{**}\cong  \mathcal{O}_X$. 
Hence we may assume that $\mathcal{I}$ is an ideal sheaf on $X$, whose cosupport is contained in $f(\Ex(f))$.  
There is a natural morphism $f^*\mathcal{I} \to \mathcal{L}$, which factors through $\mathcal{J}:= f^*\mathcal{I}/(\mathrm{torsion})$. 
We note that $\mathcal{J}$ is indeed equal to the ideal sheaf $f^{-1}\mathcal{I}\cdot \mathcal{O}_Y$, whose cosupport is contained in $\Ex(f)$.

We first assume that there is some effective Cartier divisor $\Gamma$, 
whose support is equal to $\Ex(f)$. 
Then there is some integer $j>0$, such that $\mathcal{O}_Y(-j\Gamma)$ is contained in $\mathcal{J}$. 
We hence obtain an injective morphism $ \mathcal{O}_Y\to  \mathcal{L} \otimes \mathcal{O}_Y( j\Gamma)$. 
By taking its dual, we obtain an injective morphism $ \mathcal{L}^* \otimes \mathcal{O}_Y( -j\Gamma) \to  \mathcal{O}_Y$. 
It follows that $\mathcal{L}^* \otimes \mathcal{O}_Y( -j\Gamma) \cong \mathcal{O}_Y(-E)$ for some effective $f$-exceptional divisor $E$ supported inside $\Ex(f)$. 
Thus  $\mathcal{L}\cong \mathcal{O}_Y(E-j\Gamma )$.

In general, let $g\colon Z\to Y$ be a desingularization, let $h=f\circ g$ and let $\mathcal{M}=g^*\mathcal{L}$.   
We may assume that   $\Ex(h)$ is pure of codimension 1.    
From the previous paragraph, there is some $h$-exceptional divisor $D_1$ on $Z$, such that $\mathcal{M}\cong \mathcal{O}_Z(D_1)$. 
Let $D= g_*D_1$. 
Then $D$ is supported in $\Ex(f)$, and  the reflexive sheaf $\mathcal{O}_Y(D)\otimes \mathcal{L}^*$ is isomorphic to $\mathcal{O}_Y$ on $Y\setminus g(\Ex(g))$. 
This implies that $\mathcal{L}\cong \mathcal{O}_Y(D)$, 
and completes the proof of the lemma.
\end{proof}

\bibliographystyle{alpha}
\bibliography{reference}

\end{document}